\documentclass{preprint}
\usepackage[utf8]{inputenc}
\usepackage{amssymb}
\usepackage{amsmath}
\usepackage{color}
\usepackage{hyperref}
\usepackage{mhenvs}
\usepackage{mhequ}
\usepackage{mhsymb}
\usepackage{booktabs}
\usepackage{tikz}
\usepackage{mathrsfs}
\usepackage{times}
\usepackage{microtype}
\usepackage{wasysym}
\usepackage{breakurl}

\usetikzlibrary{external}
\usetikzlibrary{decorations.text}
\usetikzlibrary{decorations.pathmorphing}

\def\slash{\kern0.18em/\penalty\exhyphenpenalty\kern0.18em}
\def\dash{\kern0.18em--\penalty\exhyphenpenalty\kern0.18em}

\DeclareFontFamily{U}{BOONDOX-calo}{\skewchar\font=45 }
\DeclareFontShape{U}{BOONDOX-calo}{m}{n}{<-> s*[1.05] BOONDOX-r-calo}{}
\DeclareFontShape{U}{BOONDOX-calo}{b}{n}{<-> s*[1.05] BOONDOX-b-calo}{}
\DeclareMathAlphabet{\mcb}{U}{BOONDOX-calo}{m}{n}
\SetMathAlphabet{\mcb}{bold}{U}{BOONDOX-calo}{b}{n}

\makeatletter

\newcommand*{\fat}{}
\DeclareRobustCommand*{\fat}{%
\mathbin{\mathpalette\bigcdot@{}}}
\newcommand*{\bigcdot@scalefactor}{.5}
\newcommand*{\bigcdot@widthfactor}{1.15}
\newcommand*{\bigcdot@}[2]{%
  \sbox0{$#1\vcenter{}$}
  \sbox2{$#1\cdot\m@th$}%
  \hbox to \bigcdot@widthfactor\wd2{%
    \hfil
    \raise\ht0\hbox{%
      \scalebox{\bigcdot@scalefactor}{%
        \lower\ht0\hbox{$#1\bullet\m@th$}%
      }%
    }%
    \hfil
  }%
}

\pgfdeclareshape{crosscircle}
{
  \inheritsavedanchors[from=circle] 
  \inheritanchorborder[from=circle]
  \inheritanchor[from=circle]{north}
  \inheritanchor[from=circle]{north west}
  \inheritanchor[from=circle]{north east}
  \inheritanchor[from=circle]{center}
  \inheritanchor[from=circle]{west}
  \inheritanchor[from=circle]{east}
  \inheritanchor[from=circle]{mid}
  \inheritanchor[from=circle]{mid west}
  \inheritanchor[from=circle]{mid east}
  \inheritanchor[from=circle]{base}
  \inheritanchor[from=circle]{base west}
  \inheritanchor[from=circle]{base east}
  \inheritanchor[from=circle]{south}
  \inheritanchor[from=circle]{south west}
  \inheritanchor[from=circle]{south east}
  \inheritbackgroundpath[from=circle]
  \foregroundpath{
    \centerpoint%
    \pgf@xc=\pgf@x%
    \pgf@yc=\pgf@y%
    \pgfutil@tempdima=\radius%
    \pgfmathsetlength{\pgf@xb}{\pgfkeysvalueof{/pgf/outer xsep}}%
    \pgfmathsetlength{\pgf@yb}{\pgfkeysvalueof{/pgf/outer ysep}}%
    \ifdim\pgf@xb<\pgf@yb%
      \advance\pgfutil@tempdima by-\pgf@yb%
    \else%
      \advance\pgfutil@tempdima by-\pgf@xb%
    \fi%
    \pgfpathmoveto{\pgfpointadd{\pgfqpoint{\pgf@xc}{\pgf@yc}}{\pgfqpoint{-0.707107\pgfutil@tempdima}{-0.707107\pgfutil@tempdima}}}
    \pgfpathlineto{\pgfpointadd{\pgfqpoint{\pgf@xc}{\pgf@yc}}{\pgfqpoint{0.707107\pgfutil@tempdima}{0.707107\pgfutil@tempdima}}}
    \pgfpathmoveto{\pgfpointadd{\pgfqpoint{\pgf@xc}{\pgf@yc}}{\pgfqpoint{-0.707107\pgfutil@tempdima}{0.707107\pgfutil@tempdima}}}
    \pgfpathlineto{\pgfpointadd{\pgfqpoint{\pgf@xc}{\pgf@yc}}{\pgfqpoint{0.707107\pgfutil@tempdima}{-0.707107\pgfutil@tempdima}}}
  }
}
\makeatother

\colorlet{symbols}{blue!90!black}
\colorlet{constants}{red!90!black}
\colorlet{testcolor}{green!60!black}

\definecolor{darkred}{rgb}{0.7,0.1,0.1}
\definecolor{darkblue}{rgb}{0.1,0.1,0.8}

\def\symbol#1{\textcolor{symbols}{#1}}
\def\rd#1{\textcolor{darkred}{#1}}
\def\1{\mathbf{\symbol{1}}}
\let\f\frac
\def\X{\symbol{X}}

\usetikzlibrary{shapes.misc}
\usetikzlibrary{shapes.symbols}
\usetikzlibrary{snakes}
\usetikzlibrary{calc}
\usetikzlibrary{fit}
\usetikzlibrary{decorations}
\usetikzlibrary{decorations.markings}

\def\drawx{\draw[-,solid] (-3pt,-3pt) -- (3pt,3pt);\draw[-,solid] (-3pt,3pt) -- (3pt,-3pt);}

\def\decorate#1#2{
        \ifnum#2>0
    		\foreach \count in {1,...,#2}{
	       	let
				\p1 = (sourcenode.center),
                \p2 = (sourcenode.east),
				\n1 = {\x2-\x1},
				\n2 = {0.1*\n1},
				\n3 = {(1.3+0.6*(\count-1))*\n1},
				\n4 = {0.7*\n1}
			in 
        		node[rectangle,fill=current,rotate=30,inner sep=0pt,minimum width=0.4*\n2,minimum height=1.5*\n2] at ($(sourcenode.center) + (\n3,\n4)$) {}
				}
		\fi
        \ifnum#1>0
    		\foreach \count in {1,...,#1}{
	       	let
				\p1 = (sourcenode.center),
                \p2 = (sourcenode.east),
				\n1 = {\x2-\x1},
				\n2 = {0.1*\n1},
				\n3 = {(1.3+0.6*(\count-1))*\n1},
				\n4 = {0.7*\n1}
			in 
        		node[rectangle,fill=current,rotate=-30,inner sep=0pt,minimum width=0.4*\n2,minimum height=1.5*\n2] at ($(sourcenode.center) + (\n3,-\n4)$) {}
				}
		\fi
}

\tikzset{
    deccircle/.style 2 args={
        circle,
        alias=sourcenode,
        append after command={\decorate{#1}{#2}}
    },
    deccircle/.default={0}{0},
    deccross/.style 2 args={
        crosscircle,
        alias=sourcenode,
        append after command={\decorate{#1}{#2}}
    },
    decsquare/.style 2 args={
        diamond,
        alias=sourcenode,
	append after command={
			let
			\p1 = (sourcenode.center),
            \p2 = (sourcenode.east),
			\n1 = {0.12*(\x2-\x1)}
			in
            node [
                draw=current,thin,
                inner sep=0pt,
                minimum width=\n1,
                minimum height=\n1,
                cross out
            ] at (sourcenode.center) {}
        },
        append after command={\decorate{#1}{#2}}
            },
    deccross/.default={0}{0},
}
\usetikzlibrary{shapes.geometric}

\tikzstyle{tinydots}=[dash pattern=on \pgflinewidth off \pgflinewidth]

\tikzset{
	root/.style={circle,fill=testcolor,inner sep=0pt, minimum size=2mm},
	dot/.style={circle,fill=black,inner sep=0pt, minimum size=1mm},
	var/.style={circle,fill=black!10,draw=black,inner sep=0pt, minimum size=2mm},
	dotred/.style={circle,fill=black!50,inner sep=0pt, minimum size=2mm},
	generic/.style={semithick,shorten >=1pt,shorten <=1pt},
	dist/.style={ultra thick,draw=testcolor,shorten >=1pt,shorten <=1pt},
	testfcn/.style={ultra thick,testcolor,shorten >=1pt,shorten <=1pt,<-},
	testfcnx/.style={ultra thick,testcolor,shorten >=1pt,shorten <=1pt,<-,
		postaction={decorate,decoration={markings,mark=at position 0.6 with {\drawx}}}},
	kprime/.style={semithick,shorten >=1pt,shorten <=1pt,densely dashed,->},
	kprimex/.style={semithick,shorten >=1pt,shorten <=1pt,densely dashed,->,
		postaction={decorate,decoration={markings,mark=at position 0.4 with {\drawx}}}},
	kernel/.style={semithick,shorten >=1pt,shorten <=1pt,->},
	kernelboundary/.style={very thick,draw=darkred},
	multx/.style={shorten >=1pt,shorten <=1pt,
		postaction={decorate,decoration={markings,mark=at position 0.5 with {\drawx}}}},
	kernelx/.style={semithick,shorten >=1pt,shorten <=1pt,->,
		postaction={decorate,decoration={markings,mark=at position 0.4 with {\drawx}}}},
	kernel1/.style={->,semithick,shorten >=1pt,shorten <=1pt,postaction={decorate,decoration={markings,mark=at position 0.45 with {\draw[-] (0,-0.1) -- (0,0.1);}}}},
	kernel2/.style={->,semithick,shorten >=1pt,shorten <=1pt,postaction={decorate,decoration={markings,mark=at position 0.45 with {\draw[-] (0.05,-0.1) -- (0.05,0.1);\draw[-] (-0.05,-0.1) -- (-0.05,0.1);}}}},
	kernelBig/.style={decorate, decoration={zigzag,amplitude=1pt,segment length = 1pt,pre length=2pt,post length=2pt}},
	rho/.style={dotted,semithick,shorten >=1pt,shorten <=1pt},
	xix/.style 2 args={deccross={#1}{#2},fill=current!10,draw=current,thin,inner sep=0pt,minimum size=1.2mm},
    xix/.default={0}{0},
	xixb/.style 2 args={deccross={#1}{#2},fill=current!10,draw=current,thin,inner sep=0pt,minimum size=1.8mm},
    xixb/.default={0}{0},
	xi/.style 2 args={deccircle={#1}{#2},solid,fill=current!10,draw=current,thin,inner sep=0pt,minimum size=1.2mm},
	xis/.style 2 args={deccircle={#1}{#2},diamond,fill=current!10,draw=current,thin,inner sep=0pt,minimum size=1.5mm},
	xisx/.style 2 args={decsquare={#1}{#2},fill=current!10,draw=current,thin,inner sep=0pt,minimum size=1.5mm},
    xi/.default={0}{0},
    xis/.default={0}{0},
    xisx/.default={0}{0},
	xib/.style 2 args={deccircle={#1}{#2},fill=current!10,draw=current,thin,inner sep=0pt,minimum size=1.8mm},
    xib/.default={0}{0},
	not/.style={circle,solid,thin,fill=current,draw=current,inner sep=0pt,minimum size=0.5mm},
	>=stealth,
	}
\colorlet{current}{symbols}
\makeatletter
\def\DeclareSymbol#1#2#3{
\expandafter\gdef\csname MH@symb@#1\endcsname{\,\tikzsetnextfilename{symbol#1}\colorlet{current}{symbols}%
\tikz[baseline=#2,scale=0.15,fill=symbols,draw=symbols]{#3}\,}
\expandafter\gdef\csname MH@symb@#1s\endcsname{\tikzsetnextfilename{symbol#1}\colorlet{current}{symbols}\scalebox{0.7}%
{\tikz[baseline=#2,scale=0.15,fill=symbols,draw=symbols]{#3}}}
\expandafter\gdef\csname MH@symb@#1c\endcsname{\,\tikzsetnextfilename{symbol#1c}\colorlet{current}{constants}%
\tikz[baseline=#2,scale=0.15,fill=constants,draw=constants]{#3}\,}
}

\def\<#1>{\ifthenelse{\boolean{mmode}}{\mathchoice{\csname MH@symb@#1\endcsname}{\csname MH@symb@#1\endcsname}{\csname MH@symb@#1s\endcsname}{\csname MH@symb@#1ss\endcsname}}{\csname MH@symb@#1\endcsname}}
\makeatother

\DeclareSymbol{IXiIXi}{1}{\draw (0,0) node[not] {} -- (1,1) node[xi] {} -- (0,2) node[xi] {};}
\DeclareSymbol{IXi1IXi}{1}{\draw (0,0) node[not] {} -- (1,1) node[xi=10] {} -- (0,2) node[xi] {};}
\DeclareSymbol{XiIXiIXi}{0.5}{\draw (0,0) node[xi] {} -- (1,1) node[xi] {} -- (0,2) node[xi] {};}
\DeclareSymbol{Xi1IXi1IXi}{0.5}{\draw (0,0) node[xi=10] {} -- (1,1) node[xi=10] {} -- (0,2) node[xi] {};}
\DeclareSymbol{Xi21IXi1IXi}{2}{\draw (0,0) node[xi=21] {} -- (-0.75,1.3) node[xi=10] {} -- (0,2.6) node[xi] {};}
\DeclareSymbol{IXiIXiIXi}{0.5}{\draw (-1,1) node[xi] {} -- (0,0) node[not] {} -- (1,1) node[xi] {} -- (0,2) node[xi] {};}
\DeclareSymbol{Xi2IXi2}{-0.5}{\draw (-1,1.5) node[xi] {} -- (0,0) node[xi=20] {} -- (1,1.5) node[xi] {};}
\DeclareSymbol{IXi^2}{0.5}{\draw (-.75,1.5) node[xi] {} -- (0,0) node[not] {}-- (.75,1.5) node[xi] {};}
\DeclareSymbol{IXiXIXi}{0.5}{\draw (-.75,1.5) node[xix] {} -- (0,0) node[not] {}-- (.75,1.5) node[xi] {};}
\DeclareSymbol{Xi2IXi^2}{-0.5}{\draw (-.75,1.5) node[xi] {} -- (0,0) node[xi=20] {}-- (.75,1.5) node[xi] {};}
\DeclareSymbol{XiIXi^2}{-0.5}{\draw (-.75,1.5) node[xi] {} -- (0,0) node[xi] {}-- (.75,1.5) node[xi] {};}
\DeclareSymbol{IXiIXiX}{0.5}{\draw (-.75,1.5) node[xix] {} -- (0,0) node[not] {} -- (.75,1.5) node[xi] {};}
\DeclareSymbol{IXiIXi11}{0.5}{\draw (-.75,1.5) node[xi] {} -- (0,0) node[not] {} -- (.75,1.5) node[xi=11] {};}

\DeclareSymbol{Xi2IXiIXiX}{-0.5}{\draw (-.75,1.5) node[xix] {} -- (0,0) node[xi=20] {} -- (.75,1.5) node[xi] {};}
\DeclareSymbol{Xi2IXiIXi11}{-0.5}{\draw (-.75,1.5) node[xi] {} -- (0,0) node[xi=20] {} -- (.75,1.5) node[xi=11] {};}

\DeclareSymbol{IXi2IXi^2}{-2.5}{\draw (0,-1.3) node[not] {} -- (0,0); \draw (-.75,1.3) node[xi] {} -- (0,0) node[xi=20] {} -- (.75,1.3) node[xi] {};}

\DeclareSymbol{IXi2X}{0.5}{\draw (-1,1.5) node[xix] {} -- (0,0) node[not] {}-- (1,1.5) node[xi] {};}

\DeclareSymbol{triple}{2}{\draw (0,2.5) node[xi] {} -- (1,1.25) node[xi=10] {} -- (0,0) node[not] {} -- (-1,1.25) node[xi] {};}

\DeclareSymbol{Xi2IXiIXi1IXi}{2}{\draw (0,2.5) node[xi] {} -- (1,1.25) node[xi=10] {} -- (0,0) node[xi=20] {} -- (-1,1.25) node[xi] {};}
\DeclareSymbol{IXiIXi1IXi}{2}{\draw (0,2.5) node[xi] {} -- (1,1.25) node[xi=10] {} -- (0,0) node[not] {} -- (-1,1.25) node[xi] {};}

\DeclareSymbol{Xi1IXi2IXi^2}{-2}{\draw (-.75,1.5) node[xi] {} -- (0,0.3) -- (.75,1.5) node[xi] {}; \draw (0,0.3) node[xi=20] {} -- (0,-1) node[xi=10] {};}
\DeclareSymbol{Xi1IIXi^2}{-2}{\draw (-.75,1.5) node[xi] {} -- (0,0.3) -- (.75,1.5) node[xi] {}; \draw (0,0.3) node[not] {} -- (0,-1) node[xi=10] {};}
\DeclareSymbol{XiIIXi^2}{-2}{\draw (-.75,1.5) node[xi] {} -- (0,0.3) -- (.75,1.5) node[xi] {}; \draw (0,0.3) node[not] {} -- (0,-1) node[xi] {};}

\DeclareSymbol{XiIxXi}{0}{\draw[kernelBig] (0,0) node[xi] {} -- (0,1.5) node[xi] {};}

\DeclareSymbol{Xi1IXib}{0}{\draw[kernelboundary] (0,0) -- (-0.2,1.5);\draw[kernelboundary] (0,0) -- (0.2,1.5);
\draw[kernelboundary] (0,0) node[xis=10] {} -- (0,1.5) node[xi,darkred,fill=darkred] {};}

\DeclareSymbol{IXib}{.5}{\draw[kernelboundary] (0,0) -- (-0.2,1.5);\draw[kernelboundary] (0,0) -- (0.2,1.5); \draw[kernelboundary] (0,0) node[not] {} -- (0,1.5) node[xi,darkred,fill=darkred] {};}

\DeclareSymbol{IXi1IXib}{1}{\draw[kernelboundary] (1,1) -- (0,2) node[xi] {};\draw (0,0) node[not] {} -- (1,1) node[xi=10] {};}

\DeclareSymbol{XiXIXi}{-.5}{\draw (0,0) node[xix] {} -- (0,1.5) node[xi] {};}
\DeclareSymbol{XiIXiX}{-.5}{\draw (0,0) node[xi] {} -- (0,1.5) node[xix] {};}

\DeclareSymbol{Xi1IXiX}{-.5}{\draw (0,0) node[xi=10] {} -- (0,1.5) node[xix] {};}
\DeclareSymbol{Xi21IXiX}{-.5}{\draw (0,0) node[xi=21] {} -- (0,1.5) node[xix] {};}
\DeclareSymbol{Xi1IXi1}{-.5}{\draw (0,0) node[xi=10] {} -- (0,1.5) node[xi=10] {};}
\DeclareSymbol{Xi1IXi11}{-.5}{\draw (0,0) node[xi=10] {} -- (0,1.5) node[xi=11] {};}

\DeclareSymbol{XiS1IXi1}{-.5}{\draw (0,0) node[xis=10] {} -- (0,1.5) node[xi=10] {};}

\DeclareSymbol{Xi2IXi}{-.5}{\draw (0,0) node[xi=20] {} -- (0,1.5) node[xi] {};}
\DeclareSymbol{Xi21IXi}{-.5}{\draw (0,0) node[xi=21] {} -- (0,1.5) node[xi] {};}
\DeclareSymbol{Xi21XIXi}{-.5}{\draw (0,0) node[xix=21] {} -- (0,1.5) node[xi] {};}
\DeclareSymbol{Xi21IXi11}{-.5}{\draw (0,0) node[xi=21] {} -- (0,1.5) node[xi=11] {};}

\DeclareSymbol{XiS1IXiX}{-.5}{\draw (0,0) node[xis=10] {} -- (0,1.5) node[xix] {};}
\DeclareSymbol{XiXS1IXi}{-.5}{\draw (0,0) node[xisx=10]  {}  -- (0,1.5) node[xi] {};}

\DeclareSymbol{XiIXi}{-.5}{\draw (0,0) node[xi] {} -- (0,1.5) node[xi] {};}
\DeclareSymbol{Xi'}{-.5}{\draw (0,0) node[xi] {} -- (0,1.5);}
\DeclareSymbol{Xi''}{-.5}{\draw (-0.5,1.5) -- (0,0) node[xi] {} -- (0.5,1.5);}

\DeclareSymbol{Xi1IXi}{-.5}{\draw (0,0) node[xi=10] {} -- (0,1.5) node[xi] {};}
\DeclareSymbol{XiS1IXi}{-.5}{\draw (0,0) node[xis=10] {} -- (0,1.5) node[xi] {};}
\DeclareSymbol{IXi}{.5}{\draw (0,0) node[not] {} -- (0,1.5) node[xi] {};}
\DeclareSymbol{IwXi}{.5}{\draw[tinydots] (0,0) node[not] {} -- (0,1.5) node[xi] {};}
\DeclareSymbol{IXiS}{.5}{\draw (0,0) node[not] {} -- (0,1.5) node[xis] {};}
\DeclareSymbol{dIXiS}{.5}{\draw[kernelboundary] (0,0) node[not] {} -- (0,1.5) node[xis] {};}
\DeclareSymbol{XiSdIXiS}{.5}{\draw[kernelboundary] (0,0) node[xis=10] {} -- (0,1.5) node[xis] {};}
\DeclareSymbol{XiSdIXiSX}{.5}{\draw[kernelboundary] (0,0) node[xis=10] {} -- (0,1.5) node[xisx] {};}
\DeclareSymbol{IXi1}{.5}{\draw (0,0) node[not] {} -- (0,1.5) node[xi=10] {};}
\DeclareSymbol{XiSIXi}{.5}{\draw (0,0) node[xis=10] {} -- (0,1.5) node[xi] {};}
\DeclareSymbol{IXi2}{.5}{\draw (0,0) node[not] {} -- (0,1.5) node[xi=20] {};}
\DeclareSymbol{IXiX}{.5}{\draw (0,0) node[not] {} -- (0,1.5) node[xix] {};}
\DeclareSymbol{IXi11}{.5}{\draw (0,0) node[not] {} -- (0,1.5) node[xi=11] {};}
\DeclareSymbol{Xi1XIXi}{-.5}{\draw (0,0) node[xix=10] {} -- (0,1.5) node[xi] {};}

\DeclareSymbol{Xismall}{-2.8}{\node[xi] {};}
\def\invariant{{\<Xismall>}}

\DeclareSymbol{XiX}{-2.8}{\node[xixb] {};}
\DeclareSymbol{Xi}{-2.8}{\node[xib] {};}
\DeclareSymbol{XiS}{-2.8}{\node[xis,minimum size=2.1mm] {};}
\DeclareSymbol{Xi1}{-2.8}{\node[xib=10] {};}
\DeclareSymbol{XiS1}{-2.8}{\node[xis=10,minimum size=2.1mm] {};}
\DeclareSymbol{XiS11}{-2.8}{\node[xis=11,minimum size=2.1mm] {};}
\DeclareSymbol{Xi2}{-2.8}{\node[xib=20] {};}
\DeclareSymbol{Xi21}{-2.8}{\node[xib=21] {};}
\DeclareSymbol{Xi22}{-2.8}{\node[xib=22] {};}
\DeclareSymbol{Xi11}{-2.8}{\node[xib=11] {};}
\DeclareSymbol{Xi11X}{-2.8}{\node[xixb=11] {};}
\DeclareSymbol{Xi1X}{-2.8}{\node[xixb=10] {};}
\DeclareSymbol{Xi2}{-2.8}{\node[xib=20] {};}

\def\sXi{\symbol{\Xi}}
\def\Wick#1{\mathord{{:}{#1}{:}}}

\newtheorem{assumption}[lemma]{Assumption}

\def\DD{\mathfrak{d}}

\def\s{\mathfrak{s}}
\def\mft{\mathfrak{t}}
\def\c{\mathfrak{c}}
\def\mfL{\mathfrak{L}}
\def\Lab{\mathfrak{L}}
\def\CCum{\textnormal{\tiny{Cum}}}

\newcommand{\V}  {\overline{V}}

\def\TT{\mathcal{T}}
\def\GG{\mathcal{G}}

\def\V{\overline{V}}
\def\VV{\hat{V}_\eps}

\def\${|\!|\!|}
\def\Eta{\symbol{\Xi_1}}

\def\CI{\symbol{\mathcal{I}}}
\def\dCI{\rd{\pmb{\mathcal{I}}}}
\def\XX#1#2{\symbol{\Xi_{#1}^{#2}}}

\def\X{\symbol{X}}
\def\1{\symbol{\mathbf{1}}}
\def\u{u^{(0)}}
\def\uu{u^{(1)}}

\def\v{v^{(0)}_\eps}
\def\vv{v^{(1)}_\eps}
\def\w#1{v^{(1,#1)}_\eps}
\def\vvv{\bar v_\eps}
\def\Reps{R_\eps^{(d)}}
\def\Repsone{R_\eps^{(1)}}
\def\Rteps{\hat R_\eps^{(d)}}

\def\scale{\mathbf{n}}
\def\tree{\mathbf{t}}
\def\trees{\mathbf{T}}
\def\half{{1\over 2}}
\def\Neu{\textnormal{\tiny \textrm{Neu}}}
\def\Dir{\textnormal{\tiny \textrm{Dir}}}
\def\PPi{\boldsymbol{\Pi}}
\def\hPeps{\hat\PPi_\eps}
\def\Veps{\rlap{$\overline{\phantom{V}}$}V_{\eps}}
\def\Alg{\mathop{\textrm{Alg}}\nolimits}
\def\Vec{\mathop{\textrm{Vec}}\nolimits}
\def\degb{\mathop{\overline{\textrm{deg}}}\nolimits}
\def\BPHZ{{\textnormal{\tiny \textsc{bphz}}}}
\def\Deltam{\Delta^{\!-}}

\def\beps{{\!\!\!\eps}}
\def\id{\mathrm{id}}

\def\restr{\mathrel{\upharpoonright}}

\newcommand\dcirc[1]{%
  {
   \mathop{\kern0pt #1}\limits^{
     \vbox to-1.85ex{
       \kern-2ex 
       \hbox to 0pt{\hss\normalfont\kern0em$\mathring{}\ \mathring{}$\hss}%
       \vss 
     }
   }
  }
}
\newcommand\tcirc[1]{%
  {
   \mathop{\kern0pt #1}\limits^{
     \vbox to-1.85ex{
       \kern-2ex 
       \hbox to 0pt{\hss\normalfont\kern0em$\mathring{}\kern0.2em\mathring{}\kern0.2em\mathring{}$\hss}%
       \vss 
     }
   }
  }
}

\begin{document}

\title{Fluctuations around a homogenised\\ semilinear random PDE}
\author{Martin Hairer$^1$ and \'Etienne Pardoux$^2$}
\institute{Imperial College London, UK, \email{m.hairer@imperial.ac.uk} 
\and  Aix Marseille Univ, CNRS, Centrale Marseille, I2M, Marseille, France, \email{etienne.pardoux@univ-amu.fr}}

\date{5 May 2021}

\maketitle

\begin{abstract}
We consider a semilinear parabolic partial differential equation in $\R_+\times [0,1]^d$, where 
$d=1, 2$ or $3$, with a highly oscillating random potential and either homogeneous Dirichlet or 
Neumann boundary condition. If the amplitude of the oscillations has the right size compared to its typical 
spatiotemporal scale, then the solution of our equation converges to the solution of a deterministic 
homogenised parabolic PDE, which is a form of law of large numbers.  Our main interest is in 
the associated central limit theorem. Namely, we study the limit of a properly rescaled difference between 
the initial random solution and its LLN limit. In dimension $d=1$, that rescaled difference converges 
as one might expect to a centred Ornstein-Uhlenbeck process. However, in dimension $d=2$, the limit 
is a non-centred Gaussian process, while in dimension $d=3$, before taking the CLT limit, we need to 
subtract at an intermediate scale the solution of a deterministic parabolic PDE, subject (in the case of Neumann boundary condition) to a \textit{non-homogeneous} Neumann boundary condition. 
Our proofs make use of the theory of regularity structures, in particular of the very recently 
developed methodology allowing to treat parabolic PDEs with boundary conditions within that theory.\\[.5em]
\textbf{Keywords:} Homogenisation, CLT, regularity structures. 
\end{abstract}

\tableofcontents

\section{Introduction}

Fix $D = [0,1]^d$ with $d \le 3$,
and consider the family of functions $u_\eps \colon [0,T]\times D \to \R$
solving the PDE
\begin{equ}[e:mainLLN]
\partial_t u_\eps(t,x)= \Delta  u_\eps(t,x)+H(u_\eps(t,x))+G(u_\eps(t,x))\eta_\eps(t,x),\quad u_\eps(0,x)=u_0(x),
\end{equ}
endowed with either Dirichlet boundary conditions $u_\eps(t,x) = 0$ for $x\in\partial D$
or Neumann boundary conditions $\scal{n(x), \nabla u_\eps(t,x)} = 0$, where $n$ denotes the outward facing
unit vector normal to the boundary of $D$.
The driving noise $\eta_\eps$ appearing in this equation is given by 
\begin{equ}[e:scalingEta]
	\eta_\eps(t,x)=\eps^{-1}\eta( \eps^{-2}t,\eps^{-1}x)\;,
\end{equ}
where $\eta(t,x)$ is a stationary centred random field, 
which we do not assume Gaussian, but with relatively good mixing properties (see Assumption~\ref{ass:kappa} 
below for details)
and moments of all orders after testing against a test function. 
Note that $\eta_\eps$ is scaled by $\eps^{-1}$ rather than $\eps^{-(d+2)/2}$, so the noise 
from \cite{HP} and \cite{GuTsai} (which were restricted to $d=1$) has been multiplied by $\eps^{d/2}$.

In the case when $G$ is linear and $H=0$, this problem has been well studied. See for example
\cite{Bal} for the case where furthermore $\eta$ is Gaussian and constant in time (note that in our case 
$\mathfrak{m}=2$ so the exponent $\alpha$ appearing there would equal $1$ in our case, as used in
\eqref{e:scalingEta}), \cite{BalGu} for a similar result in the non-Gaussian case.

Although we will allow $\eta$ to be a generalised random field in $d=1$, we assume throughout that there exist 
locally integrable \textit{functions} $\kappa_p \colon (\R^{d+1})^p \to \R$ that are continuous outside of the big diagonal
$\Delta_p = \{(z_1,\ldots,z_p) \in (\R^{d+1})^p\,:\, \exists i\neq j\;\text{with}\; z_i = z_j\}$
and such that, for any $\CC_0^\infty$ test functions $\phi_1,\ldots,\phi_p$,
the joint cumulant $\kappa_p(\phi_1,\ldots,\phi_p)$ of $\eta(\phi_1),\ldots,\eta(\phi_p)$ satisfies
\begin{equ}
\kappa_p(\phi_1,\ldots,\phi_p) = \int \kappa_p(z_1,\ldots,z_p)\,\phi_1(z_1)\cdots \phi_p(z_p)\,dz_1\cdots dz_p\;.
\end{equ}
(By stationarity, the functions $\kappa_p$ only depend on the differences of their arguments.)
Here and below, we always use the convention that $z$ (resp.\ $z_i$, $\bar z$, etc) denotes a space-time coordinate given by $z=(t,x)$ (resp.\ $z_i=(t_i,x_i)$, $\bar z=(\bar t,\bar x)$, etc).
We furthermore normalise our problem by assuming that the covariance of $\eta$ integrates to $1$ in the sense that
\begin{equ}[e:normalisation]
\int_{\R^{d+1}} \kappa_2(0,z)\,dz = 1\;.
\end{equ}
In particular, we assume that $\kappa_2(0,\cdot)$ is absolutely integrable, but this will in any case follow
from Assumption~\ref{ass:kappa} below.
Define furthermore the constants
\begin{equs}[2]
\<XiIXic>&=\int P(z)\,\kappa_2(0,z)\,dz,\quad&
\<XiIXiXc>&=\int P(z)\,x \kappa_2(0,z)\,dz,\\
\<XiIXi^2c>&=\int P(z) P(z')\,\kappa_3(0,z,z')\,dz\,dz',\quad&
\<XiIXiIXic>&=\int P(z) P(z'-z)\,\kappa_3(0,z,z')\,dz\,dz',
\end{equs}
where $P$ denotes the heat kernel, i.e.\ the fundamental solution to the heat equation
on \textit{the whole space}. Here and below, symbols drawn in red denote fixed \textit{constants},
while symbols drawn in blue will later denote basis vectors of a suitable regularity structure associated
to our problem. In dimension $d=1$, our assumptions on $\kappa_2$ will guarantee that the integral 
$\<XiIXic>$ converges absolutely, while in dimensions $2$ and $3$ our assumptions on $\kappa_2$ and $\kappa_3$
will guarantee that all of these integrals converge absolutely.

The scaling $\eps^{-1}$ chosen in $\eta_\eps$ is such that $u_\eps$ converges as $\eps \to 0$ to a limit $\u$,
which is our first result.
Indeed, writing
\begin{equ}[e:Heta]
H_\eta(u) = H(u) + \<XiIXic> G'(u) G(u)\;,
\end{equ}
we have the following ``law of large numbers''.

\begin{theorem}\label{theo:LLN}
Let $u_\eps$ be as above and let $\u$ be the (local) solution to the deterministic PDE
\begin{equ}[e:ubar]
\d_t \u = \Delta \u + H_\eta(\u)\;,
\end{equ}
with the same initial condition $u_0 \in \CC^\alpha$ as \eqref{e:mainLLN} (for some arbitrary $\alpha \in (0,1)$)
and with homogeneous Dirichlet (resp.\ Neumann) boundary condition.
In the case of Dirichlet boundary conditions, we impose that $u_0$ vanishes on the boundary.

Assume that 
the functions $G,H\colon \R \to \R$ are of class $\CC^5$ and $\CC^4$ respectively, that
the driving field $\eta$ satisfies
Assumption~\ref{ass:kappa} below, and let furthermore $T > 0$
be such that the (possible) explosion time for $\u$ is greater than $T$.
Then, in probability and uniformly over $[0,T] \times D$, $u_\eps$ converges to
$\u$ as $\eps \to 0$.
\end{theorem}

The proof of this result will be given in Section~\ref{sec:LLN}.
Our main quantity of interest however are the fluctuations of $u_\eps$
around its limit $\u$. One interesting feature of this problem is that in order to 
see these fluctuations, it is not sufficient to recenter around $\u$. Instead, as soon
as $d \ge 2$, a suitable first-order
correction $\uu$ living at scale $\eps$ has to be subtracted first. Our precise
``central limit theorem'' then takes the following form.

\begin{theorem}\label{theo:main}
Let $u_0$ be such that its extension to all of $\R^d$ by reflections\footnote{Here we perform the reflections consistent with the reflection principle
for our choice of boundary conditions.} is of class $\CC^3$,
let $u_\eps$ and $\u$ be as above and assume $G$, $H$, $T$ and $\eta$ are as in Theorem~\ref{theo:LLN}.
Let furthermore $\uu = 0$ for $d=1$ and, for $d \in \{2,3\}$, let $\uu$ be the solution to
\begin{equs}\label{e:ubar1}
\d_t \uu &= \Delta \uu  + H'_\eta(\u) \uu + \Psi(\u, \nabla \u)\;, \\
\Psi(u, p) &= {1\over 2}\<XiIXi^2c>  (G^2G'')(u) + \<XiIXiIXic> (G(G')^2)(u) + (G')^2(u) \,\scal{\<XiIXiXc>,p}\;,
\end{equs}
with zero initial condition. (Note that $\<XiIXiXc>$ is an $\R^d$-valued constant.)
In the case of Neumann boundary condition, we furthermore impose that 
$\scal{\nabla \uu(t,x), n(x)} = c(x) GG'(\u(t,x))$ for $x\in\d D$, where $c$ is an explicit function on the boundary of $D$ which is constant
on each of its faces (the precise values of $c$ on each face will be given in \eqref{e:intQ} below),
while we impose homogeneous Dirichlet boundary conditions
otherwise.

Then, in law and in $\CC^\alpha([0,T] \times D)$ for any
$\alpha < 1-{d\over 2}$,
one has
\begin{equ}[e:multiscale]
\lim_{\eps \to 0} {u_\eps - \u - \eps\uu \over  \eps^{d/2}}= v\;,
\end{equ}
where $v$ is the Gaussian process solving
\begin{equ}[e:deflimitv]
\d_t v = \Delta v + H_\eta'(\u) v+ G(\u) \xi\;,
\end{equ}
endowed with homogeneous Dirichlet (resp.\ Neumann) boundary condition and $0$ initial condition, and $\xi$ denotes
a standard space-time white noise.
\end{theorem}

\begin{proof}
Combining Proposition~\ref{prop:idenSol} with Proposition~\ref{prop:localSol} and \eqref{e:defveps}
shows that if we set $v_\eps = \eps^{-d/2}\big(u_\eps - \u - \eps\uu\big)$
then we do indeed have $\lim_{\eps \to 0} v_\eps = v$
(weakly in $\CC^{(d-2)/2-\kappa}$ on $[0,T]$). The limit $v$ is identified as the solution
to \eqref{e:deflimitv} by combining the second part of Proposition~\ref{prop:localSol}
with Lemma~\ref{lem:V0}.
\end{proof}

\begin{remark}
If all we were interested in is the law of large numbers, then the conditions of 
Assumption~\ref{ass:kappa} on $\eta$ could easily be weakened.
\end{remark}

\begin{remark}
In the case of Neumann boundary conditions, it may appear paradoxal that, even though 
$u_\eps$, $\u$ and $v$ all satisfy homogeneous boundary conditions, $\uu$ does not!
This phenomenon is very similar to the presence of the ``boundary renormalisation''
that can appear in the context of singular SPDEs \cite{Mate}. There is no contradiction since the
convergence $v_\eps \to v$ takes place in a very weak topology in which the notion of ``normal derivative
at the boundary'' is meaningless in a pointwise sense.
(A very simple example displaying a similar phenomenon is $n^{-1/2}\sin(nx)$, whose derivative at the origin 
diverges like $\sqrt n$ while that of its limit vanishes.)
\end{remark}

\begin{remark}\label{rem:defSol}
Regarding the precise meaning of the equation fulfilled by $\uu$ in the case of Neumann boundary condition,
 denote by $\delta_{\d D}$ the distribution
on $\R \times \d D$ given by
\begin{equ}
\delta_{\d D}(\phi) = \int_{\R} \int_{\d D} \phi(t,x)\,dx\,dt\;,
\end{equ}
where the integration over the faces of $\d D$ are performed against the two-dimensional 
Lebesgue measure. With this notation, the solution to any equation of the form
\begin{equ}[e:nonhomogeneous]
\d_t u = \Delta u + f\;\text{in $D$,}\quad \scal{\nabla u(t,x), n(x)} = g(x)\;\text{on $\d D$,}\quad
u(0,x) = u_0(x)\;,
\end{equ}
(and in particular the equation determining $\uu$) is defined as the solution to the integral equation
\begin{equ}[e:mildForm]
u(z) = \int_{D}P_\Neu(z,(0,x'))\,u_0(x')\,dx'
+ \int_{\R^4} P_\Neu(z,z') \bigl(f\one_+^D + g \delta_{\d D}\bigr)(dz')\;,
\end{equ}
where $P_\Neu$ denotes the homogeneous Neumann heat kernel,
with the convention that $g(t,x) = 0$ for $t \le 0$, and where $\one_+^D(t,x) = \one_{\{t \ge 0\}}\one_{\{x \in D\}}$.
Here, we used the notation $\int \phi(z)\,\eta(dz)$ for the usual pairing between a distribution $\eta$ and a suitable
test function $\phi$. To see that solutions to \eqref{e:mildForm} and \eqref{e:nonhomogeneous} do indeed coincide
if $f$ and $g$ are sufficiently regular for the solution to be differentiable up to the boundary,
it suffices to note that the mild formulation is equivalent to the weak formulation, see for example \cite{Walsh},
with the term $g \delta_{\d D}$ appearing as the boundary term when integrating by parts.
\end{remark}

\begin{remark}
In dimension $d=1$, the term $\uu$ in \eqref{e:multiscale} is of course redundant.
In dimension $d=2$, it is still the case that $\eps^{-d/2}(u_\eps - \u)$ converges to a limit,
but this limit is not centred anymore. In higher dimensions, additional corrections appear.
We expect to have a result of the form
\begin{equ}[e:developmentueps]
\lim_{\eps \to 0} \eps^{-d/2} \Big(u_\eps - \sum_{k=0}^{\lfloor d/2\rfloor} \eps^k u^{(k)} \Big) 
 = v\;,
\end{equ}
where $\u$ is as above and the $\bar u^{(k)}$ satisfy an equation of the type
\begin{equ}
\d_t  u^{(k)} = \Delta  u^{(k)}  + H'_\eta(\u) u^{(k)} + \Psi_k\;,
\end{equ}
for some inhomogeneity $\Psi_k$ depending on the $ u^{(\ell)}$ for $\ell < k$ and
some of their derivatives.
Since $v$ has vanishing expectation, we expect to also have
\begin{equ}
\lim_{\eps \to 0} \eps^{-d/2} \big(u_\eps - \E u_\eps \big) = v\;,
\end{equ}
so that the $u^{(k)}$ provide an expansion of $\E u_\eps$ in powers of $\eps$. Note however
that the techniques used in this paper do not provide moment bounds on the solution, so 
that even in $d \le 3$ this would require some additional work.
\end{remark}

\begin{remark}
The form of the terms appearing in the successive correctors as well as the constants multiplying
them can in principle be derived from \cite[Eq.~2.12]{Ilya} which describes the form of the counterterms
$\Upsilon$ associated to a given tree $\tau$. The recursion given there 
suggests a correspondence $\<Xic> \sim G(u)$
with incoming edges corresponding to functional derivatives with respect to $u$ (so $\<Xi'c> \sim G'(u)$
and $\<Xi''c> \sim G''(u)$)
and powers of $X$ corresponding to formal directional derivatives, so $\<XiXc> \sim G'(u)\nabla u$.
This shows a priori that the counterterm multiplying $\<XiIXic>$ for example must be of the form $G'(u)G(u)$, 
etc.
The numerical constants multiplying these terms do however differ from those appearing in 
\cite{Ilya} since their meaning is slightly different: in \cite{Ilya} we use counterterms to 
``recenter'' the original equation in order to obtain a finite limit while here we leave the original
equation untouched and compute its ``centering''. If we leave aside the behaviour at the boundary,
this in principle allows to guess the general form of the equations for the $u^{(k)}$ appearing 
in \eqref{e:developmentueps} for any dimension.
\end{remark}

The most surprising part of Theorem~\ref{theo:main} is surely the fact that in the case when
$u_\eps$ has homogeneous Neumann boundary conditions, even though 
$v$ and $\u$ both also have homogeneous boundary conditions, $\uu$ does not, which seems
to contradict \eqref{e:multiscale}. This is of course not a contradiction but merely suggests 
that if we write $v_\eps$ for the expression appearing under the limit in \eqref{e:multiscale},
then $v_\eps$ exhibits a kind of boundary layer. Note also that the statement that ``$v$
satisfies homogeneous boundary conditions'' only makes sense in terms of the integral equation that
it solves since $v$ itself is not differentiable at the boundary. (It is not even a function!)

Before we proceed, let us give a heuristic explanation for the appearance of this 
boundary layer. Consider the simplest case $H=0$, $G(u) = u$ and $u_0 > 0$, in which case we can consider
the Hopf-Cole transform $h_\eps = \log u_\eps$, yielding
\begin{equ}
\d_t h_\eps = \Delta h_\eps + |\nabla h_\eps |^2 + \eta_\eps\;.
\end{equ}
To leading order, one would expect the right hand side to behave like 
$|\nabla h_\eps |^2 \simeq \E |\nabla Z_\eps |^2$, where $Z_\eps$ solves 
$\d_t Z_\eps = \Delta Z_\eps + \eta_\eps$ endowed with homogeneous Neumann boundary conditions. 
It turns out that, in the interior of the domain, one
has \[\lim_{\eps \to 0} \E |\nabla Z_\eps |^2 = \<XiIXic>\;,\] which allows one to ``guess'' the
correct limit $\u$. On the boundary however $\nabla Z_\eps = 0$, so that one expects 
$\E |\nabla Z_\eps |^2 - \<XiIXic>$ to be of order $\CO(1)$ 
in a layer of width $\CO(\eps)$ around $\d D$. When going to the next scale, this results in
a boundary correction of order $\CO(\eps^{-1})$ in this boundary layer, which precisely scales
like a surface measure on the boundary. Remark~\ref{rem:defSol} shows that the net effect of
this correction is to modify the boundary condition.

The remainder of the article is structured as follows. First, in Section~\ref{sec:ass},
we formulate our main assumption on the driving noise $\eta$ and we show that this assumption
is ``reasonable'' by exhibiting an explicit class of examples for which it is satisfied.
In Section~\ref{sec:LLN}, we then show 
that the law of large numbers holds. Although this could probably be shown by ``classical''
means without too much effort, we will use the theory of regularity structures because it shortens the argument and 
allows us to introduce some results and notation that will be of use later on. In Section~\ref{sec:CLT}, we then 
show that the central limit theorem holds. The main tool in this proof is the convergence of
a certain ``model'' for an appropriate regularity structure as well as refinements of the type
of boundary estimates first considered in \cite{Mate}. The convergence of the model is 
given in Section~\ref{sec:model}.
Appendix~\ref{sec:almostComm} is devoted to the proof of a result showing that the operations
of ``convolution by a singular kernel'' and ``multiplication by a smooth function'' almost commute,
modulo a much smoother remainder, a fact that will undoubtedly sound familiar to anyone acquainted with
microlocal analysis. Appendix~\ref{sec:reconstruction} contains a version of the reconstruction theorem that is
purpose-built to allow us to deal with modelled distribution that have very singular boundary 
behaviour and goes beyond the version obtained in \cite{Mate}. This appendix was written in collaboration
with Máté Gerencsér.

\subsection*{Acknowledgements}

{\small
The authors gratefully acknowledge financial support from the
 Leverhulme Trust through a leadership award,
from the European Research Council through a consolidator grant, project 615897,
and from the Royal Society through a research professorship.
We are also grateful to Máté Gerencsér for allowing us to use 
his ideas for the results of Appendix~\ref{sec:reconstruction} and to Ajay Chandra for
discussions on the application of the results of \cite{Ajay}.
}

\section{Assumptions on the Noise}
\label{sec:ass}

In this section, we formulate our precise assumptions on the driving noise and we show that 
they are satisfied for example by a mollified Poisson process.
In a nutshell, we want to assume that correlations are bounded by $\|z-\bar z\|^{-2\underline c}$ at
small scales and $\|z-\bar z\|^{-2\overline c}$ at large scales with $\underline c = {1\over 2}-\delta$
and $\overline c = {d+2\over 2}+\delta$ for some $\delta \in (0,{1\over 2})$. 
However, we also want to encode the fact that higher order cumulants behave ``better'' than
what is obtained from simply using the Cauchy-Schwartz inequality. Note that our assumptions are
trivially satisfied by any continuous Gaussian process with correlations that decay 
at least like $\|z-\bar z\|^{-2\overline c}$. Here and below, $\|\cdot\|$ will always denote
the parabolic distance between space-time points. It will be convenient (in particular in Appendix~\ref{app:extHeat})
to make sure that $\|\cdot\|$ is smooth away from the origin, so we set for example $\|z\|^4 = \|(t,x)\|^4 = |x|^4 + |t|^2$.

\subsection{Coalescence trees}
\label{sec:coal}

In order to formulate this precisely, we need a simplified version of the
construction of \cite[Appendix~A]{Jeremy}. 
Given any configuration $(z_1,\ldots,z_p)$ of $p$ points in $\R^{d+1}$ with all 
distances distinct, we associate to it a binary tree $T$ in the following way.
Consider Kruskal's algorithm \cite{Kruskal} for constructing the minimal spanning tree of 
the complete graph with vertices $\{z_1,\ldots,z_p\}$ and edge-weights given by their
(parabolic) distances. One way of formalising this is the following. Consider the
set $\CP_p$ of partitions of $\Omega_p = \{1,\ldots,p\}$.
We define a distance $d_z$ between subsets of $\Omega_p$ as the Hausdorff distance induced by $\{z_1,\ldots,z_p\}$, namely
\begin{equ}
d_z(A,B) = \max\Big\{\sup_{i \in A}\inf_{j \in B}\|z_i - z_j\|, \sup_{j \in B}\inf_{i \in A}\|z_i - z_j\|\Big\}\;.
\end{equ}

We then define a map $K\colon \CP_p \to \CP_p$ in the following way. 
If $\pi = \{\Omega_p\}$, then $K(\pi) = \pi$. Otherwise, let $A \neq B \in \pi$ be such that
$d_z(A,B) \le d_z(C,D)$ for all $C,D \in \pi$. Thanks to our assumption on the $z_i$,
this pair is necessarily unique. We then set
\begin{equ}
K(\pi) = \bigl(\pi \setminus \{A,B\}\bigr) \cup \{A \cup B\}\;,
\end{equ}
i.e.\ $K(\pi)$ is obtained by coalescing the two sets $A$ and $B$ in the partition $\pi$.
The vertices of $T$ are then given by $V_T = \bigcup_{n \ge 0} K^n(\{\{1\},\ldots,\{p\}\})$, i.e.
$V_T$ consists of all the blocks of those partitions.
The set $V_T$ comes with a natural partial order given by inclusion: $A \le B$ if and 
only if $A \supset B$. The (directed) edge set $E_T \in V_T\times V_T$ of $T$ is then given by the 
Hasse diagram of $(V_T,\le)$: $(A,B) \in E_T$ if and only if $A < B$ and there is no $C \in V_T$ such that
$A < C < B$. It is easy to verify that $T$ is a binary tree and that its leaves are
precisely given by the singletons. 
It will be convenient to also add to $V_T$ a 
``point at infinity'' $\logof$ which is connected to $\Omega_p$ by an edge $(\Omega_p,\logof)$
and to view $\logof$ as the minimal element of $V_T$.

We write $\mathring V_T = V_T \setminus \{\{1\},\ldots,\{p\},\logof\}$ for the interior nodes.
Each interior node $A \in \mathring V_T$ has exactly two children $A_1$ and $A_2$ such that 
$(A,A_i) \in E_T$ for $i = 1,2$.
We then define an integer labelling $\scale \colon \mathring V_T \to \Z$ 
by $\scale(A) = -\lceil \log_2 d_z(A_1,A_2) \rceil$.
We will always view
$\scale$ as a function on all of $V_T$ with values in $\Z \cup \{\pm \infty\}$
by setting $\scale(\logof) = -\infty$ and $\scale(\{i\}) = +\infty$ for
$i=1,\ldots,p$.
Note now that if $A$, $B$ and $C$ are three disjoint sets, then 
\begin{equ}
d_z(A,B) \le \min\{d_z(A,C), d_z(B,C)\} \quad\Rightarrow\quad d_z(A,B) \le d_z(A\cup B, C)\;. 
\end{equ}
As a consequence, the map $\scale$ is monotone increasing on $V_T$.
Furthermore, as in \cite[Eq.~A.15]{Jeremy}, there exist constants $c, C$ depending only on $p$ such that 
\begin{equ}
c 2^{-\scale(\{i\} \wedge \{j\})} \le \|z_i - z_j\| \le C 2^{-\scale(\{i\} \wedge \{j\})}\;,
\end{equ}
for all $i,j$. 

Given a configuration of points $z = (z_1,\ldots,z_p) \in (\R^{d+1})^p$, we now write
$\tree_z = (T,\scale)$ for the corresponding data constructed as above.
We furthermore define a function $\rho \colon \R_+\to \R_+$ by
\begin{equ}
\rho(r) = r^{-\overline c} \wedge r^{-\underline c}\;.
\end{equ}
(Beware that $\rho$ is an upper bound for the \textit{square root} of the covariance between
two points.)
We then assume that the following bound holds.

\begin{assumption}\label{ass:kappa}
With the notations as above, for any $p \ge 2$ and any $\{k_i\}_{i=1}^p \subset \Z_+^{d+1}$, the $p$th joint cumulant for $\eta$ satisfies the bound
\begin{equ}[e:boundkappa]
\Big|\Big(\prod_{i=1}^p D_i^{k_i}\Big)\kappa_p(z_1,\ldots,z_p)\Big| \lesssim \rho\big(2^{-\scale(\Omega_p)}\big) \prod_{A \in \mathring V_T}  \rho\big(2^{-\scale(A)}\big)\prod_{i=1}^p 2^{|k_i| \,\scale(i^\uparrow)}\;,
\end{equ}
uniformly over all $z \in (\R^{d+1})^p$.
(Recall that $\Omega_p$ is the root of the tree $T$.)
Here and below, the length of the multiindex $k$ should be interpreted in the parabolic sense, namely
$|k| = 2k_0 + \sum_{i=1}^d k_i$.

In dimensions $d \in \{2,3\}$, we furthermore assume that $\eta \colon \Omega \times \R^{d+1} \to \R$ 
is a measurable function with $\E |\eta(0)|^p <\infty$ for $p = (d+2)/{\underline c}$.
\end{assumption}

\begin{remark}
The additional condition that $\eta$ takes values in $L^p$ for sufficiently high $p$ is mainly technical
and could probably be dropped with some additional effort. It will be used to bound
$\Rteps$ in the proof of Proposition~\ref{prop:localSol} below. The exponent $(d+2)/\underline c$
is consistent with the condition on the correlation function in the sense that 
this is the lowest value of $p$ for which $L^p_{\mathrm{loc}} \subset \CC^{-\underline c}$.
\end{remark}

Note also that the cumulants $\kappa_p^{(\eps)}$
of the rescaled process $\eta_\eps$ satisfy
\begin{equ}[e:scaleKappa]
\kappa_p^{(\eps)}(z_1,\ldots,z_p) = 
\eps^{-p} \kappa_p(S_\eps z_1,\ldots,S_\eps z_p)\;,
\end{equ}
where $S_\eps(t,x) = (t/\eps^2,x/\eps)$.

\subsection{Justification}

We claim that the assumption on the noise is rather weak on the ground that many natural
constructions yield stationary random processes that satisfy it. We provide details for the
following example.

\begin{proposition}\label{prop:justif}
Let $\theta \colon \R^{d+1} \to \R$ be smooth away from $0$ and such that for all $k\in\Z_+^d$, $|D^k\theta(z)| \lesssim \|z\|^{-2\overline c-|k|}$ for $\|z\| > 1$
and $|D^k\theta(z)| \lesssim \|z\|^{-\underline c-|k|}$ for $\|z\| \le 1$.
Let $\mu$ be a Poisson point measure over $\R^{d+1}$ with intensity $1$ 
and set $\eta = \mu \star \theta$, then
$\eta$ satisfies the above assumption.
\end{proposition}
Before proving this proposition, let us first establish a property of joint cumulants of
integrals of deterministic functions with respect to a Poisson point measure.
\begin{lemma}\label{lem:Poisson}
Let $p\ge1$ and let $f_1,\ldots,f_p$ be elements of $L^1(\R^{d+1})\cap L^p(\R^{d+1})$, and again $\mu$ be a 
Poisson point measure over $\R^{d+1}$ with intensity $1$. Then the joint cumulant
$\kappa_p(\mu(f_1),\ldots,\mu(f_p))$ of the random variables $\mu(f_1),\ldots,\mu(f_p)$ satisfies
\[ \kappa_p(\mu(f_1),\ldots,\mu(f_p))=\int_{\R^{d+1}}f_1(z)\times\cdots\times f_p(z)\; dz.\]
\end{lemma}
\begin{proof}
It is sufficient to prove the result in case there exist  disjoint Borel subsets $A_1,\ldots,A_k$ of 
$\R^{d+1}$ with finite Lebesgue measure such that for $1\le i\le p$, 
\[ f_i(z)=\sum_{j=1}^k a_{i,j}{\bf1}_{A_j}(z).\]
But in that case the result follows readily from the fact that the joint cumulant is $p$-linear,
and that the joint cumulant of a collection of random variables which can be split into two mutually independent subcollections vanishes, see e.g.\ property (iii) in \cite[p.~32]{PeccatiTaqqu}.
\end{proof}

\begin{proof}[of Proposition~\ref{prop:justif}]
It follows from Lemma \ref{lem:Poisson} that
\begin{equ}[e:cumulant]
\kappa_p(z_1,\ldots,z_p)=\int_{\R^{d+1}} \theta(z_1-z)\cdot\ldots\cdot\theta(z_p-z)\,dz\;,
\end{equ}
so it remains to obtain a bound on this integral.
We now consider $z_1,\ldots, z_p$ to be fixed and we shall make use of the labelled tree 
$(T,\scale)$ built from these points as above.
We are first going to treat the simpler case with all $k_i$ vanishing and then show how the argument 
can be modified to deal with the general case.

\smallskip\noindent\textbf{Case 1.} 
The case with all $k_i=0$.  For every edge $e = (\underline e,\bar e) \in E_T$ and 
every $n \in \Z$ with $\scale(\underline e) < n < \scale(\bar e)$, we 
define the domain
\begin{equ}
D_{(e,n)} = \big\{z\in \R^{d+1} \,:\, c^{-1}2^{-n_i} \le \|z- z_i\| \le c2^{-n_i},\;\forall i\in \{1,\ldots,p\} \big\}\;,
\end{equ}
where
\begin{equ} 
n_i = 
\left\{\begin{array}{cl}
	n & \text{if $\{i\} \ge \bar e$,} \\
	\scale(\{i\}\wedge \bar e) & \text{otherwise.}
\end{array}\right.
\end{equ}
It is possible to convince oneself that, provided that the constant $c$ appearing in the
definition of $D_{(e,n)}$ is sufficiently large, one has
\begin{equ}
\bigcup_{e \in E_T} \bigcup_{\scale(\underline e) < n < \scale(\bar e)} D_{(e,n)} = \R^{d+1}\setminus\{z_1,\ldots, z_p\}\;,\qquad
|D_{(e,n)}| \le (2c)^{d+2} 2^{-(d+2)n}\;.
\end{equ}
As a consequence, the integral appearing in \eqref{e:cumulant} is bounded by
some constant times
\begin{equ}[e:intsum]
\sum_{e \in E_T} \sum_{\scale(\underline e) < n < \scale(\bar e)}
2^{-(d+2)n} \prod_{i=1}^p \bar\rho(2^{-n_i})\;,\quad
\bar \rho(r) = r^{-\underline c} \wedge r^{-2\overline c}\;.
\end{equ}
We first use the fact that
$\bar \rho$ is decreasing to conclude that, for $\scale(\underline e) < n < \scale(\bar e)$,
one has the bound
\begin{equ}[e:boundcum]
2^{-(d+2)n} \prod_{i=1}^p \bar\rho(2^{-n_i})
\le 2^{-(d+2)n} \bar\rho(2^{-n}) \prod_{v \in \mathring V_T} \bar \rho(2^{-\scale(v)})\;.
\end{equ}
This can be seen as follows. Write $\{v_1,\ldots,v_k\}$ for the (possibly empty)
set of nodes in $\mathring V_T$ lying on the shortest path joining $\bar e$ to $\logof$
(not including $\bar e$ and $\logof$ themselves). We then have, for every $j=1,\ldots,k$,
\begin{equ}
\prod_{i\,:\, \{i\}\wedge \bar e = v_j} \bar \rho(2^{-n_i})
= \prod_{i\,:\, \{i\}\wedge \bar e = v_j} \bar \rho(2^{-\scale(v_j)}) \le 
\prod_{v \in \mathring V_T\,:\, v\wedge \bar e = v_j}   \bar \rho(2^{-\scale(v)})\;,
\end{equ}
since the number of factors appearing in each term is the same.
Similarly, we have 
\begin{equ}
\prod_{i\,:\, \{i\} \ge \bar e} \bar \rho(2^{-n})
\le \bar \rho(2^{-n}) 
\prod_{v\in \mathring V_T\,:\, v \ge \bar e} \bar \rho(2^{-\scale(v)})\;,
\end{equ}
hence \eqref{e:boundcum}.
Since $\bar \rho \le \rho \wedge \rho^2$, it follows from \eqref{e:boundcum} that
\begin{equ}
 2^{-(d+2)n} \prod_{i=1}^p \bar\rho(2^{-n_i})
 \le
2^{-(d+2)n} \bar\rho(2^{-n}) \rho(2^{-\scale(\Omega_p)}) \prod_{v \in \mathring V_T} \rho(2^{-\scale(v)})\;.
\end{equ}
It remains to observe that $\sum_{n \in \Z} 2^{-(d+2)n} \bar\rho(2^{-n})
= \sum_{n \in \Z} 2^{(\underline c-d-2)n} \wedge 2^{(2\overline c - d-2)n} < \infty$,
so that \eqref{e:intsum} is indeed bounded by the required expression.

\smallskip\noindent\textbf{Case 2.} 
Note that we actually showed that the expression \eqref{e:cumulant} with $\theta$ replaced by $\bar \rho$
is bounded by the right hand side of \eqref{e:boundkappa} with $k_i = 0$.
To obtain the general case, it therefore suffices to show that 
\begin{equ}[e:requiredDer]
\int_{\R^{d+1}}\prod_{i=1}^dD^{k_i}\theta(z_i-z)\,dz \lesssim \prod_{i=1}^p 2^{|k_i| \,\scale(i^\uparrow)} \int_{\R^{d+1}}\prod_{i=1}^d \bar \rho(z_i-z)\,dz\;.
\end{equ}
Let $\chi\in C^\infty(\R^{d+2})$ be such that
\begin{equ}
\chi(z)=\begin{cases} 1,&\text{on $B(0,1/4)$};\\
0,&\text{on $B(0,1/2)^c$.}
\end{cases}
\end{equ}
For $1\le j\le p$, we define $\chi_j(z)=\chi\big(2^{\scale(j^\uparrow)}(z-z_j)\big)$, and $\chi_0(z)=1-\sum_{j=1}^p\chi_j(z)$.
It is clear that
\begin{equs}
\int_{\R^{d+1}}\prod_{i=1}^dD^{k_i}\theta(z_i-z)\,dz&=\sum_{j=0}^p\int_{\R^{d+1}}\chi_j(z)\prod_{i=1}^dD^{k_i}\theta(z_i-z)\,dz\\
&=\int_{\R^{d+1}}\chi_0(z)\prod_{i=1}^dD^{k_i}\theta(z_i-z)\,dz\\ &\quad +
\sum_{j=1}^p\int_{\R^{d+1}}\theta(z_j-z)D^{k_j}\Big(\chi_j(\cdot)\prod_{i\not=j}D^{k_i}\theta(z_i-\cdot)\Big)(z)\,dz\,.
\end{equs}
We note that for $z$ in the support of $\chi_0$, for $1\le i\le p$,  $2\| z_i-z\|\ge2^{-\scale(i^\uparrow)}$,
\begin{equs}
|D^{k_i}\theta(z_i-z)|&\lesssim\bar \rho(\|z_i-z\|)\cdot\| z_i-z\|^{-|k_i|} \\
&\lesssim \bar\rho(\|z_i-z\|)\cdot 2^{|k_i|\scale(i^\uparrow)},
\end{equs}
thus yielding \eqref{e:requiredDer} as required.

To bound the final term, we note that its integrand can be written as a finite sum of terms of the form
\begin{equ}
M(z) = \theta(z_j-z)D^k\chi_j(z)\prod_{i\not=j}D^{k_i+k_{j,i}}\theta(z_i-z),
\end{equ}
where $k, k_{j,i}\in\Z_+^{d+1}$ and $k+\sum_{i\not=j}k_{j,i}=k_j$. Each of these terms is bounded above by the 
indicator function of the support of $\chi_j$ times
\begin{equ}
\bar\rho(\|z-z_j\|) 2^{|k|\scale(j^\uparrow)} \prod_{i\not=j}\bar\rho(\|z-z_i\|) \|z-z_i\|^{-|k_i|-|k_{j,i}|}\;.
\end{equ}
Since for $z$ in the support of $\chi_j$ and $i\not=j$, 
$2\|z-z_j\|\le\|z_i-z_j\|$, so that       
\begin{equ}
\|z_i-z_j\|\le\|z_i-z\|+\|z-z_j\| \le\|z_i-z\|+\frac12\|z_j-z_i\|,
\end{equ}
one has $2\|z-z_i\| \ge\|z_j-z_i\|\ge2^{-\scale(i^\uparrow)}\wedge2^{-\scale(j^\uparrow)}$.
Combining all of these bounds, we finally obtain 
\begin{equs}
|M(z)| &\le \bar\rho(\|z-z_j\|) 2^{|k|\scale(j^\uparrow)} \prod_{i\not=j}\bar\rho(\|z-z_i\|)2^{-|k_i| \scale(i^\uparrow)} 2^{-|k_{j,i}| \scale(j^\uparrow)} \\
&\le \prod_i \bar\rho(\|z-z_i\|) 2^{|k_i|\scale(i^\uparrow)}\;,
\end{equs}
at which point we apply again \eqref{e:requiredDer} to obtain the required bound.
\end{proof}

\section{Law of Large Numbers}
\label{sec:LLN}

The aim of this section is to use a simplified\footnote{except for the treatment of the boundary conditions which leads to
non-trivial complications} variant of the arguments in \cite{HP}
to show that Theorem~\ref{theo:LLN} holds.
Although it would probably not be much more involved to obtain this proof by usual
techniques, we give a proof using regularity structures. 
The main reason is that this allows us to introduce in a simpler setting a number of notions
and notations that will be useful in the proof of our main result later on.

Before we turn to the proof proper, let us comment on the way in which we deal with the
Neumann boundary conditions. Writing $P_\Neu$ for the Neumann heat kernel and using the
notation $z=(t,x)$ (and similarly for $z'$), we rewrite \eqref{e:mainLLN}
as an integral equation:
\begin{equ}
u_\eps(z) = \int_{D_t} P_\Neu(z,z') \bigl(H(u_\eps(z')) + G(u_\eps(z'))\eta_\eps(z')\bigr)\,dz'
+ \int_{D} \!\!\!  P_t^\Neu(x,x')u_0(x')\,dx'\;,
\end{equ}
where $D_t = [0,t]\times D$
and $P_t^\Neu(x,x') = P_\Neu((t,x),(0,x'))$. We also fix an arbitrary time horizon $T \le 1$ which is not
a restriction since the argument can be iterated.

Following \cite{Mate}, we then construct two functions $K$ on $\R^{d+1}$ and $K_\d$ on $\R^{d+1}\times \R^{d+1}$ 
such that $K$ is compactly supported, $K_\d$ is supported on a strip of finite width around the diagonal,
and the identity
\begin{equ}
P_\Neu(z,z') = K_\Neu(z,z')\;,
\end{equ}
holds for $z,z' \in [0,1] \times D$, where we set
\begin{equ}[e:decompose]
K_\Neu(z,z') = K(z-z') + K_\d(z,z')\;.
\end{equ}
See Appendix~\ref{app:extHeat} for more details on the 
construction of these kernels and a proof that this can be done in a way that is compatible with the
results of \cite{regularity,Mate} that we will use in our argument.

\begin{remark}
We make no claim on the values of $K$ and $K_\d$ for arguments outside of
$[0,1] \times D$. This is because these will always be integrated against functions that are 
supported on $[0,1] \times D$ and only the values of the result inside the domain will matter. 
\end{remark}

We choose $K$ in such a way that it coincides with the whole space heat kernel $P$ on the (parabolic) ball of radius $1$ and
is compactly supported in the ball of radius $2$. 
We furthermore choose $K$ in such a way that it annihilates polynomials of degree up to $3$, is invariant
under the transformation $(t,x) \mapsto (t,-x)$, and is such that the sum of its reflections agrees with the
Neumann heat kernel on $[0,1] \times D$. (See Appendix~\ref{app:extHeat} for more details.)
For example, we can choose $K$ as in \cite{HP}. 
The kernel $K_\d$ is a correction term
that encodes the effect of the boundary condition.  
Regarding our regularity structure, we then proceed as if there was no boundary condition whatsoever:
we construct models defined on the whole
space that are translation invariant and we use convolution with $K$ as our integration operator.
We then define an operator $\CP_\Neu$ on modelled distributions by setting
\begin{equ}[e:defCP]
\CP_\Neu = \CK + \tilde \CK_\d\;,\quad\text{where}\quad \tilde \CK_\d = \CL_2 K_\d \CR\;,
\end{equ}
and $\CK$ is built in exactly the same way as in \cite[Sec.~4]{regularity}.
Note that $\tilde \CK_\d$ encodes the effect of the boundary condition. 
(There is a completely analogous definition in the case of Dirichlet boundary conditions.)

Here, $\CL_\gamma \colon \CC^\gamma \to \CD^\gamma$ denotes the ``Taylor lift'' given by 
\begin{equ}[e:Taylor]
(\CL_\gamma f)(z) = 
\sum_{|k| \le \gamma} {f^{(k)}(z) \over k!} \X^k\;,
\end{equ}
where $z = (t,x)$ and $k$ denotes a multiindex in $\N^{1+d}$.

We now have the preliminaries in place to turn to the proof of Theorem~\ref{theo:LLN}.

\begin{proof}[of Theorem~\ref{theo:LLN}]
We use a strategy similar to that in \cite{HP,Jeremy}, combining this with results from \cite{Mate}
to deal with the boundary conditions. We refer to \cite{IntroReg,RP,IntroAjay} for introductions to the theory of regularity 
structures, as well as to \cite{regularity} for details.
In our present context, we use the regularity structure obtained by extending the usual polynomial structure
with parabolic scaling with a
symbol $\Eta$ of degree $-1-\kappa$ representing
the driving noise $\eta_\eps$, as well as an abstract integration operator $\CI$ of order $2$ representing convolution
with the (singular part of) heat kernel. 
As usual, we will often use graphical representations for the basis vectors in our regularity structure(s),
and we decree that $\<Xi1>$ is our symbolic representation for $\Eta$. (The reason for
introducing the ``accent'' representing the index ``$1$'' will become clear later on where more general
notations of this type are needed.)
Although our goal is to consider \eqref{e:mainLLN} on the bounded
domain $D\subset\R^d$, we construct the models for our regularity structure on the whole of $\R\times \R^d$.

With notations almost identical to those 
in \cite{HP} and the formula (3.19) there, it would then be natural to consider a fixed point 
problem of the type
\begin{equ}[e:abstrLLN]
U = \CP_\Neu \one_+^D \bigl(\hat H_\eta(U) + \hat G(U)\<Xi1>\bigr) + P_\Neu u_0\;,
\end{equ}
where $\one_+^D$ denotes the indicator function of the space-time domain $\R_+\times D$.
Leaving considerations regarding the precise spaces of modelled distributions in which this
equation makes sense aside for the moment, 
it is straightforward to see as in \cite{regularity} that if we solve \eqref{e:abstrLLN} for the 
renormalised lift of $\eta_\eps$, i.e. the admissible model such that
\begin{equ}[e:renormModel]
\hPeps \<Xi1> = \eta_\eps\;,\qquad 
\hPeps\<Xi1IXi1> = \eta_\eps (K\star \eta_\eps) - \<XiIXic> \;,
\end{equ}
then the function $u_\eps = \CR_\eps U$ actually solves \eqref{e:mainLLN}.

Indeed, iterating \eqref{e:abstrLLN}, we see that any solution $U$ to such a fixed point problem
is necessarily of the form
\begin{equ}
U=u\1+G(u)\, \<IXi1> +\nabla u\, \X\;,
\end{equ}
for some continuous functions $u$ and $\nabla u$. (This is purely notational, $\nabla u$ is \textit{not}
the gradient of $u$, but can be interpreted as a kind of ``renormalised gradient''.) 
In particular, the factor multiplying $\CP_\Neu \one_+^D$ in the right hand side of 
\eqref{e:abstrLLN} is given by
\begin{equ}[e:RHSLLN]
L \eqdef H_\eta(u)\1 + G(u)\, \<Xi1> + G'(u)G(u)\,\<Xi1IXi1> + G'(u)u'\,\<Xi1X>\;,
\end{equ}
where we projected onto terms of negative (or vanishing) degree.
At this point, it then suffices to note that the application of the reconstruction operator
to $L$ yields
\begin{equs}
(\CR_\eps L)(z) &= \bigl(\hat \Pi_z^\eps L(z)\bigr)(z) = H_\eta(u(z)) + G(u(z))\eta_\eps(z)
- (G'G)(u(z))\<XiIXic>\\
&= H(u(z)) + G(u(z))\eta_\eps(z)\;,
\end{equs} 
as required.

The problem with the argument outlined above is that since $\deg \<Xi1> < -1$, the behaviour of the modelled 
distribution $\one_+^D L$ near $\d D$ is such that the reconstruction operator is not a priori well-defined
on it, see \cite[Secs~4.1 \& 4.2]{Mate}. This is for precisely the same reason why the restriction of a generic distribution 
$\zeta \in \CC^\alpha$ to a ``nice'' domain $D$ is only well-defined if $\alpha > -1$. (For $\alpha \le -1$ there are non-zero distributions with support contained in $\d D$.)

Before we tackle this problem, recall the definition of the spaces $\CD^{\gamma,\eta}$ as in 
\cite[Sec.~6]{regularity} (the hyperplane $P$ being given by the time slice $t=0$)
as well as the spaces $\CD^{\gamma,w}$ as in \cite[Sec.~4]{Mate} (in which
case $P_0$ is again the time $0$ slice while $P_1 = \R \times \d D$).
These two spaces are distinguished by the fact that $\eta$ is a real exponent while
$w$ denotes a triple of exponents describing the singular behaviour near $t=0$, $\d D$ and
the intersection of both regions respectively.
It will be convenient to use the notation $(\alpha)_3$ with $\alpha \in \R$ as a 
shorthand for the triple $(\alpha,\alpha,\alpha)$.

The idea then is the following. 
First, we introduce a new symbol $\<XiS1>$, also of degree $-1-\kappa$,
but representing the function $\eta_\eps  \one_{\R\times D}$ instead of
representing $\eta_\eps$ and we add to our regularity structure the symbols
$\X \<XiS1>$ and $\<XiS1IXi1>$. 
Write then $V$ for the sector spanned by $\<Xi1>$, $\<Xi1IXi1>$, and
$\X \<Xi1>$, $\hat V$ for the sector spanned by $\<XiS1>$, $\<XiS1IXi1>$, and
$\X \<XiS1>$, and $\iota \colon V \to \hat V$ for the linear map with $\iota \<Xi1> = \<XiS1>$
and similarly for the remaining basis vectors. 
We will furthermore only ever consider models $\PPi$ with the property that
\begin{equ}[e:consistency]
\bigl(\PPi \iota\tau\bigr)(\phi) = 
\left\{\begin{array}{cl}
	0 & \text{if $\supp \phi \subset \R\times D^c$,} \\
	\bigl(\PPi \tau\bigr)(\phi) & \text{if $\supp \phi \subset \R\times D$,}
\end{array}\right.
\end{equ}
for all $\tau \in V$.
Since $\iota$ commutes with the structure group and preserves degrees,
it follows that $F \mapsto \iota F$ is continuous from
$\CD^{\gamma,w}(V)$ to $\CD^{\gamma,w}(\hat V)$ for all choices of exponents $\gamma$ and $w$
and, for $\gamma > 0$, the local reconstruction  operator $\tilde \CR$ (which yields a distribution on
$\R\times (\R^d \setminus \d D)$ and is always well-defined)
 satisfies
\begin{equ}
\bigl(\tilde\CR \iota F\bigr)(\phi) = 
\left\{\begin{array}{cl}
	0 & \text{if $\supp \phi \subset \R\times D^c$,} \\
	\bigl(\tilde\CR F\bigr)(\phi) & \text{if $\supp \phi \subset \R\times D$.}
\end{array}\right.
\end{equ}

The reason for the introduction of these extra symbols is that we would like to interpret \eqref{e:abstrLLN}
as a fixed point problem in the space $\CD^{1+2\kappa,(2\kappa)_3}$ with values in $V_0$ (for small enough $\kappa$),
where $V_0$ is spanned by $\1$, $\<IXi1>$, and
$\X$. The problem now is that, for $F \in \CD^{1+2\kappa,(2\kappa)_3}$, we have 
$F\<Xi1> \in \CD^{\kappa,(\kappa-1)_3}$, but we lose some regularity at the boundary when multiplying by 
the indicator function of our domain (see \cite{Mate}), so that we only have 
$\one_+^D F\<Xi1> \in \CD^{\kappa,(-\kappa-1)_3}$. Since the boundary index is now below $-1$,
it follows that the reconstruction operator of \cite{Mate} is not well-defined on $\one_+^D F\<Xi1>$.
By Theorem~\ref{thm:reconstructDomain} below, it is however perfectly well-defined on $\one_+ F\<XiS1> = \one_+ \iota (F\<Xi1>)$ since 
only the temporal singularity index is below $-1$ (but above $-2$). Furthermore, one
has the identity $\CR (\one_+ F\<XiS1>) = \tilde \CR (\one_+^D F\<Xi1>)$ for test functions whose support
does not intersect the boundary of $D$.

Recall now that, given a modelled distribution $F$ and a distribution $\zeta$ agreeing with $\tilde \CR F$ outside
the boundary of $D$, \cite[Lem~4.12]{Mate} defines a modelled distribution $\CK^\zeta F$ with improved regularity 
and such that $\CR \CK^\zeta F = K \star \zeta$. Furthermore, the map $(\zeta,F) \mapsto \CK^\zeta F$ is Lipschitz 
continuous in the natural topologies.
This allows us to define an operator 
$\CP_D\colon \CD^{\kappa,(\kappa-1)_3}(V) \to \CD^{\kappa+2,(1-\kappa)_3}(V)$ by setting
\begin{equ}
\CP_D\colon F \mapsto \CK^{\CR \iota F} \one_+^D F + \tilde \CK_\d \one_+ \iota F\;.
\end{equ}
Our discussion suggests that,  instead of \eqref{e:abstrLLN}, we should consider the fixed point problem
\begin{equ}[e:abstrLLN2]
U = \CP_\Neu \one_+^D \hat H_\eta(U) + \CP_D \hat G(U)\<Xi1> + P_\Neu u_0\;,
\end{equ}
which admits unique local solutions in $\CD^{1+2\kappa,(2\kappa)_3}(V)$ by \cite{Mate}, which are
continuous with respect to admissible models on the full regularity structure satisfying 
furthermore the consistency condition \eqref{e:consistency}.
Here, we need to choose $\kappa$ small enough to guarantee that $P_\Neu u_0$ does indeed belong to 
the space $\CC^{\kappa+2,(1-\kappa)_3}$, which is possible thanks to our assumption that $u_0$ itself
is H\"older continuous for some positive exponent.

Retracing the discussion given at the beginning of the proof, but now with 
the renormalised model $\hPeps$ such that, in addition to \eqref{e:renormModel},
one has $\hPeps \iota \tau = \one_{\R \times D} \hPeps \tau$ for $\tau \in V$, we conclude that for $\eps > 0$, 
solutions to \eqref{e:abstrLLN2} coincide with those of \eqref{e:mainLLN}.
We now refer to Theorem~\ref{theo:model} below which shows that the sequence of models $\hPeps$
converges, as $\eps\to0$, to a limiting model $\hat\PPi$ such that 
$\hat \PPi \<Xi1> = \hat \PPi \<XiS1> = 0$, extended canonically to the whole regularity structure.
It follows immediately that the solution $\bar U$ to \eqref{e:abstrLLN} with the model $\hat\PPi$
is such that  $\u = \CR \bar U$ does indeed solve \eqref{e:ubar} as claimed.
\end{proof}

\begin{remark}
Note that \eqref{e:consistency} does not \textit{force} us to set 
$\hPeps \<XiS1> = \one_{\R\times D} \hPeps \<Xi1> = \one_{\R\times D} \eta_\eps$, but we could have added a sufficiently
regular distribution supported on $\d D$. This however would break the identity
\begin{equ}
\CR_\eps \CP_D G = P_\Neu \one_+^D \CR_\eps G\;,
\end{equ}
on $[0,1]\times D$ and would therefore modify the boundary condition of the resulting solution.
\end{remark}

\section{Central Limit Theorem}
\label{sec:CLT}

We now turn to the proof of the main result of this article, Theorem~\ref{theo:main}.
We will mainly focus on the case of Neumann boundary conditions in dimension $d=3$, which is the 
most interesting (and technically most difficult) case. 
We set
\begin{equ}[e:defveps]
v_\eps=\frac{u_\eps-\u - \eps \uu}{\eps^{d/2}}\;,\qquad
\xi_\eps=\eps^{-d/2}\eta_\eps\;,\quad
\sigma_\eps=\eps^{d/2}\eta_\eps\;.
\end{equ}
With this notation, we then have in the case $d\in \{2,3\}$
\begin{equs}
\partial_t v_\eps
&=\Delta v_\eps+H_\eta'(\u) v_\eps + \eps^{-{d\over 2}} \bigl(H(u_\eps) - H_\eta(\u) - H_\eta'(\u)(u_\eps-\u)\bigr) \\
&\quad + G(u_\eps) \xi_\eps - \eps^{1-{d\over 2}}\Psi(\u, \nabla \u)\\
&=\Delta v_\eps+H_\eta'(\u) v_\eps + G(\u) \xi_\eps \\
&\quad+ \eps^{-{d\over 2}} \bigl(H_\eta(u_\eps) - H_\eta(\u) - H_\eta'(\u)(u_\eps-\u)\bigr) - \eps^{-{d\over 2}} \<XiIXic> (GG')(\u) \\
&\quad - \eps^{-{d\over 2}} \<XiIXic> \bigl((G')^2 + GG''\bigr)(\u) (u_\eps - \u)
+ G'(\u) (u_\eps - \u)\xi_\eps \\
&\quad + {1\over 2}G''(\u) (u_\eps - \u)^2 \xi_\eps - \eps^{1-{d\over 2}}\Psi(\u, \nabla \u)  + R_{1,\eps}\;,
\end{equs}
where, setting
\begin{equ}
w_\eps = \eps^{-{d\over 2}}(u_\eps-\u) = v_\eps+ \eps^{1-{d\over 2}} \uu\;,
\end{equ}
we have the explicit expression for the remainder term
\begin{equs}
R_{1,\eps} &= -\eps^{d/2}w_\eps^2 \int_0^1 \<XiIXic> (GG')''(\u + s\eps^{d/2}w_\eps )(1-s)\,ds\\
&\quad + \frac{\eps^{{3d\over 2}}}{2}w_\eps^3 \int_0^1 G^{(3)}(\u + s\eps^{d/2}w_\eps )(1-s)^2\,ds\, \xi_\eps  \;.
\end{equs}
Furthermore, due to the non-vanishing boundary condition of $\uu$ in the Neumann case,
$v_\eps$ is then endowed with the inhomogeneous boundary condition
\begin{equ}
\scal{\nabla v_\eps, n(x)} = -\eps^{1-d/2} c(x) GG'(\u(t,x))\;. 
\end{equ}
We now also incorporate the first part of the second line into the remainder, so that we can write
\begin{equs}
\partial_t v_\eps
&=\Delta v_\eps+H_\eta'(\u) v_\eps + G(\u) \xi_\eps - \eps^{-{d\over 2}} \<XiIXic> (GG')(\u) \label{e:CLT}\\
&\quad -  \<XiIXic> (GG')'(\u) (v_\eps+ \eps^{1-{d\over 2}} \uu)
+ G'(\u) (v_\eps+ \eps^{1-{d\over 2}} \uu)\eta_\eps \\
&\quad + {1\over 2}G''(\u) (v_\eps+ \eps^{1-{d\over 2}} \uu)^2 \sigma_\eps - \eps^{1-{d\over 2}}\Psi(\u, \nabla \u)+ \Rteps(v_\eps, \eps^{\alpha}\xi_\eps) \;, 
\end{equs}
where the remainder term $\Rteps$ is given by
\begin{equs}
\Rteps(v,\varsigma) &=\Reps(v + \eps^{1-d/2}\uu,\varsigma)\;,\\
\Reps(w,\varsigma) &= \eps^{d/2}w^2 \int_0^1 H''(\u + s\eps^{d/2}w )(1-s)\,ds \label{e:defReps}\\
&\qquad+\frac{\eps^{{3d\over 2}-\alpha}}{2}w^3 \int_0^1 G^{(3)}(\u + s\eps^{d/2}w )(1-s)^2\,ds\, \varsigma  \;.
\end{equs}
Note that here we have $H$ appearing in \eqref{e:defReps} rather than $H_\eta$ and that the two are related by
\eqref{e:Heta}.
In dimension $d=1$, we set $\uu = 0$ so that $\Rteps = \Reps$
and $v_\eps = (u_\eps - \u)/\sqrt \eps$, and we obtain in the same way
the slightly simpler expression
\begin{equs}
\partial_t v_\eps
&=\Delta v_\eps+H_\eta'(\u) v_\eps + G(\u) \xi_\eps - \eps^{-{1\over 2}} \<XiIXic> (GG')(\u) \label{e:CLT1d}\\
&\quad -  \<XiIXic> (GG')'(\u) v_\eps
+ G'(\u) v_\eps\eta_\eps  + {1\over 2}G''(\u) v_\eps^2 \sigma_\eps+ \hat R^{(1)}_\eps(v_\eps, \eps^{\alpha}\xi_\eps) \;. 
\end{equs}
(The reason why the term containing $\Psi$ does not appear in this expression is because this
was generated by $\d_t \uu$ which vanishes by definition in dimension one.)
The exponent $\alpha$ appearing in this expression is of course arbitrary, but allowing to tune it will
be convenient when expressing this as a fixed point problem.

\subsection{Decomposition of the solution}
\label{sec:decomp}

In order to show that $v_\eps$ converges to a limit, it will be convenient to break it into a sum of three terms.
The first term will be a straightforward approximation to the stochastic heat equation with noise strength $G(\u)$
and homogeneous boundary condition. The second term will converge to $0$, but incorporates the diverging
boundary condition, which is used to compensate a resonance appearing in its right hand side. The final term will be a 
remainder that is sufficiently regular to be dealt with by the techniques of \cite{Mate}.
For this, we write
\begin{equ}
v_\eps = \v + \vv + \vvv\;,
\end{equ}
and, with the convention that $G$ and its derivatives are always evaluated at $\u$, we set
\begin{equs}
\d_t \v &=\Delta \v + G \xi_\eps\;,\label{e:defv0}\\
\d_t \vv &=\Delta \vv + G' \Bigl(\v \eta_\eps  - \eps^{-{d\over 2}} \<XiIXic> G - \eps^{1-{d\over 2}} G' \,\scal{\<XiIXiXc>,\nabla \u}\Bigr) \label{e:defv1} \\
&\quad + {1\over 2} G''\Bigl((\v)^2 \sigma_\eps - 2\<XiIXic>G\,\v- \eps^{1-{d\over 2}} \<XiIXi^2c>G^2\Bigr)\;,
\end{equs}
endowed with the boundary conditions on $\d D$
\begin{equ}
\scal{\nabla \v, n} =0\;,\qquad 
\scal{\nabla \vv, n}=-\eps^{1-{d\over 2}} c GG'\;,
\end{equ}
as well as vanishing initial conditions.
The reason why $\vv$ will actually converge to $0$ despite the diverging boundary condition when $d \ge 2$ is the following.
Consider the function $\PPi_\eps\<IXib>$ defined by
\begin{equ}[e:defPPiepscorr]
\bigl(\PPi_\eps\<IXib>\bigr)(z) = \int_{\R \times D} K_\Neu(z,z') \xi_\eps(z')\,dz' - \int_{\R^{d+1}} K(z-z') \xi_\eps(z')\,dz'\;,
\end{equ}

Then, we will see in \eqref{e:defV0tilde} below that the behaviour
of $\v$ is locally very well described by that of
\begin{equ}[e:v0]
G(\u) \bigl(\PPi_\eps\<IXi> + \PPi_\eps\<IXib>\bigr)\;,
\end{equ}
where $\PPi_\eps\<IXi> = K \star \xi_\eps$ as usual. This implies that the
behaviour of the term 
$G'(\u) \v \eta_\eps$ appearing in the right hand side of the equation for $\vv$ is well described
locally by that of 
\begin{equ}[e:rhs]
\bigl(GG'\bigr)(\u) \bigl(\PPi_\eps\<Xi1IXi> + \PPi_\eps\<Xi1IXib>\bigr)\;,
\end{equ}
Where $\<Xi1>$ denotes multiplication by $\eta_\eps$ as previously, while
$\<XiS1>$ denotes multiplication by $\eta_\eps \one_{\R \times D}$. (This will be formalised later on.)
We will see that, up to vanishingly small errors, 
$\PPi_\eps\<Xi1IXi> \approx \eps^{-{d\over 2}}\<XiIXic>$, while 
$\PPi_\eps\<Xi1IXib> \approx \eps^{1-{d\over 2}}c \delta_{\d D}$ for a suitable constant $c$
(in fact a different constant for each face of $\d D$ in general), so that the first term in 
\eqref{e:rhs} is cancelled up to small errors by the term $- \eps^{-{d\over 2}} \<XiIXic> (GG')(\u)$
appearing in the equation for $\vv$, while the second term in \eqref{e:rhs} is cancelled by the
term $-\eps^{1-{d\over 2}} c GG' \delta_{\d D}$ created by the boundary condition.

Note that this argument does not see much of a difference between the Neumann and Dirichlet cases.
Indeed, if we want the right hand side of the equation for $\vv$ to converge to a limiting distribution for the 
latter, we also need to add a diverging (when $d=3$) term proportional to $\delta_{\d D}$. The difference is that 
$K_\Dir(z,z')$ vanishes for $z' \in \d D$, so that this term has no influence on $\vv$ in the Dirichlet case.

The idea now is to proceed as follows.
\begin{claim}
\item In a first step, we describe in Section~\ref{sec:structure} a regularity structure that is sufficiently rich to allow us to give precise
control on the behaviour of $\v$, $\vv$ and $v_\eps$. As already alluded to, this will in particular include
symbols representing \textit{non-stationary} space-time stochastic processes, but we will try to keep these
to an absolute  minimum. We also describe there the renormalisation procedure which allows us to construct a suitable (random) model.
\item We then make precise the description \eqref{e:v0} for $\v$ by expressing it as a modelled distribution
in this regularity structure, which in particular contains the two symbols $\<IXi>$ and $\<IXib>$. The main conclusion in \eqref{e:defV0tilde}
will be
that the presence of $\<IXib>$ allows us to express $\v$ as a modelled distribution with a good behaviour near $\d D$.
(At this stage, we could of course also use one symbol only and define our model using the Neumann heat kernel only,
but the decompositions \eqref{e:v0} and \eqref{e:rhs} are convenient for the remainder of the argument and to be able to reuse existing results.)
\item In a second step, we show that if we set
\begin{equ}
\hat \PPi_\eps\<Xi1IXi> = \PPi_\eps\<Xi1IXi> - \eps^{-{d\over 2}}\<XiIXic>\;,\quad
\hat \PPi_\eps\<Xi1IXib> = \PPi_\eps\<Xi1IXib> - \eps^{1-{d\over 2}}c\delta_{\d D}\;,
\end{equ}
then $\hat \PPi_\eps$ converges to a limiting model as $\eps \to 0$, which furthermore has good
restriction properties to $D$, uniformly in $\eps$. 
Furthermore, since $\hat \PPi_\eps\<Xi1IXib>$ is only singular near the
boundary of $D$, we can describe $\v$ by another modelled distribution with worse behaviour near the boundary, but which only
uses ``translation invariant symbols'' in its description, see Lemma~\ref{lem:V0}, which allows us to give a description of $\vv$ 
in terms of such symbols in Lemma~\ref{lem:V1}.
\item We set up a sufficiently large regularity structure so that we can formulate a fixed point
problem for the remainder $\vvv$ and control its behaviour as $\eps \to 0$, see Propositions~\ref{prop:localSol}
and~\ref{prop:idenSol}. Combining this with the convergence of the corresponding renormalised model which is 
performed in Theorem~\ref{theo:model} but relies crucially on the next section, we are finally able to conclude.  
\end{claim}

\subsection{Definition of the ambient regularity structure}
\label{sec:structure}

We start by defining a regularity structure that is sufficiently large to allow us to
perform the steps mentioned above and in particular to formulate \eqref{e:CLT} as a 
fixed point problem for a modelled distribution $V$. This fixed point problem will be
chosen precisely in such a way that if we take as our model the \textit{renormalised}
lift of the noise $\eta_\eps$, then the corresponding counterterms are precisely such that 
if $V$ solves the fixed point problem, then $\CR V$ solves \eqref{e:CLT}.

Let  $\XX ij$ be new symbols representing $\xi_\eps \eps^{\alpha(i,j)}$ with
\begin{equ}[e:defalphaij]
\alpha(i,j) \eqdef \Big(1-{d\over 2}\Big)j + {d\over 2}i\;,\qquad 0 \le j \le i \le 2\;.
\end{equ}
(Beware that  $j$ is simply a superscript in $\XX ij$, not a power.) 
In the graphical notation analogous to
\cite{HP}, we will use ``accents'' to denote the upper and lower indices on $\sXi = \XX 00 = \<Xi>$,
so for example $\XX 2{} = \<Xi2>$, $\XX 11 = \<Xi11>$, etc. We will also sometimes write
$\<Xi>_i^j$ instead of $\XX ij$.
The degree of these symbols is given by
\begin{equ}[e:degXX]
\deg \XX ij = {\delta - 1\over 2} \wedge \Bigl(\alpha(i,j) - {d+2\over 2}-\kappa\Bigr)\;.
\end{equ}
The main reason why we never consider these ``noises'' as having degree larger than 
${\delta - 1\over 2}$ is that we want to view them as ``noise types'' and so the structure 
group should act trivially on them.
A byproduct of this choice is that 
it allows us to deal with driving noises $\eta$ that are themselves unbounded.

We now build a regularity structure extending the one built in Section~\ref{sec:LLN},
in which \eqref{e:absCLT} can be formalised, by applying the framework of \cite[Sec.~5]{Lorenzo}.
We take as our basic ``noise types'' the noises $\<Xi>_i^j$ with $0 \le j \le i \le 2$,
as well as an additional noise type $\<XiS>$ which will be used to represent the
noise $\xi_\eps$, restricted to $\R \times D$.
We also introduce two ``edge types''
$\CI$  representing convolution by a suitable cutoff $\CK$ of the standard heat kernel in the whole space
and $\dCI$ representing the integral operator with kernel $K_\Neu$ as in \eqref{e:decompose} (see Appendix~\ref{app:extHeat} for a
precise definition),
both having degree $2$. It will be convenient to also introduce edge types $\CI_i$ and $\dCI_i$
 with $i=1,\ldots,d$ of degree $3$ representing the integral operators with 
kernel $K_i(x,y) =  (y_i-x_i)K(y-x)$ and $K_{\Neu,i}(x,y) =  (y_i-x_i) K_\Neu(y-x)$ respectively. Finally, we 
introduce a ``virtual'' edge types $\hat \CI$ of degree $2$ 
which will allow  us to produce a rule (in the technical sense of \cite[Sec.~5]{Lorenzo})
generating the relevant trees containing non-translation invariant symbols $\dCI$ and $\<XiS>$, but 
without cluttering our regularity structure with unneeded symbols.

The rule used to generate the regularity structure is then given by $R(\<Xi>_i^j) = \{\1\}$ and
\begin{equs}[e:defRule]
R(\CI) &= \{\1,\CI, \CI_i,  \CI_i \<Xi1>\} \cup \{ \CI^k\<Xi>_i^j \,:\, k \le i-j\}\;,\\
R(\CI_i) &= \{\1,\<Xi>\}\;,\qquad 
R(\dCI) = R(\dCI_i) = \{\1, \<XiS>\}\;,\\
R(\hat \CI) &= \{\1, \<XiS>,  \dCI \<XiS1>,  \dCI, \dCI_i, \<XiS1> , \<XiS11>,\CI, \CI \<XiS1>, \CI_i \<XiS1>\}\;.
\end{equs}
Recall that, given a collection $\Lab$ of ``edge types'' (in our case, these are the
``integration'' types $\CI$, $\dCI$, etc as well as the ``noise'' types $\<Xi>_i^j$ and 
$\<XiS>_i^j$ which are also interpreted as edges for the purpose of this discussion), 
a ``rule'' is a map $R \colon \Lab \to \CP(\hat \CP(\Lab)) \setminus \{\emptyset\}$, where $\CP(A)$ denotes the 
powerset of a set $A$, $\hat \CP(\Lab)$ denotes the set of non-empty multisets with 
elements in $\Lab$.\footnote{We ignore the possibility of having ``derivatives'' on our edges
as in \cite{Lorenzo}, as these do not occur in this article.} In other words, an element $R(\mft)$ is a collection of ``node types'', 
with each node type being a collection of edge types, with repetitions allowed.
In \eqref{e:defRule}, we use the identification between a multiset and a formal monomial, i.e.\ 
$\<Xi>\CI^2$ denotes the multiset with one copy of $\<Xi>$ and two copies of $\CI$.

The basis vectors for our regularity structure are rooted trees $(V, E, \rho)$ with edges $e \in E$ labelled by $\Lab$
and nodes $v \in V$ labelled by $\N^d$, denoting polynomial factors, with the convention that $0 \in \N$.
By convention, edges are oriented away from the root. Any node $v\in V$ then has a ``type''
$\CN(v) \in \hat \CP(\Lab)$ given by the collection of the types of the outgoing edges adjacent to $v$.
If $v \in V \setminus \{\rho\}$, there is a unique edge coming into $v$ and we write $\mft(v) \in \Lab$ for the type of the incoming edge.
A tree is said to ``conform'' to a rule $R$ if, whenever $v \in V \setminus \{\rho\}$, one has $\CN(v) \in R(\mft(v))$,
while $\CN(\rho) \in \bigcup_{\mft \in \Lab} R(\mft)$.

\begin{remark}
Note that $\hat \CI$ never appears inside any $R(\mft)$, so that a tree conforming to the rule $R$
is not allowed to have any edge of type $\hat \CI$. The only reason for its presence is to allow the root
of a conforming tree to be of type $\CN(\rho) \in R(\hat \CI)$, since otherwise the tree 
$\<XiSdIXiS>$ which will be used later on would not be conforming to our rule.
Actually, the symbol $\<Xi1IXib>$, which was introduced in a completely ad hoc manner so far, will then be 
interpreted as $\<Xi1IXib>=\<XiSdIXiS> - \<XiSIXi>$
\end{remark}

We assign degrees to the components appearing here by \eqref{e:degXX} as well as
\begin{equs}[2]
\deg \<XiS> &= \deg \<Xi>\;,\quad&\quad
\deg \<XiS1> &= \deg \<Xi1>\;,\\ 
\deg \CI &= \deg \hat \CI = \deg \dCI = 2\;,\quad&\quad
\deg\CI_i &= \deg \dCI_i = 3\;.
\end{equs}
It is straightforward to verify that this rule is subcritical and complete.
We henceforth denote by $(\TT,\GG)$ the (reduced in the terminology of \cite[Sec.~6.4]{Lorenzo}) 
regularity structure generated by the rule $R$.

\subsection{Description of the models}
\label{sec:defModel}

Throughout this article, we will use the notation $\PPi$ (possibly with additional decorations) for a 
continuous linear map $\PPi \colon \CT \to \CD'(\R^{d+1})$ such that there 
exists a (necessarily unique) admissible model $(\Pi,\Gamma)$ 
related to $\PPi$ by \cite[Sec.~8.3]{regularity}. 
Here, we associate the kernel $K$ to $\CI$ and $\hat \CI$, $K_i$ 
to $\CI_i$, and $K_\Neu$, $K_{\Neu,i}$  to 
$\dCI$, $\dCI_i$ respectively.

We henceforth make a slight abuse of terminology and
call $\PPi$ itself a model.
The following notion of an ``admissible'' model is a slight strengthening of the usual one to our context which essentially
states that our ``square'' symbols are the restrictions of the ``round'' symbols to $\R \times D$.

\begin{definition}\label{def:admissible}
A model $\PPi$ for $(\TT,\GG)$ is \textit{admissible} if it is admissible in the sense of 
\cite[Def.~8.29]{regularity} for the kernels listed above
  and furthermore, for any $\tau \in \TT$ and $i,j$ such that $\<Xi>_i^j\tau, \<XiS>_i^j\tau \in \TT$,
one has
\begin{equ}
\bigl(\PPi \<XiS>_i^j\tau\bigr)(\phi) = \bigl(\PPi \<Xi>_i^j\tau\bigr)(\phi)\;,
\end{equ}
for test functions $\phi$ with $\supp \phi \subset \R \times D$ and 
$\bigl(\PPi \<XiS>_i^j\tau\bigr)(\phi) = 0$
for test functions with $\supp \phi \subset \R \times D^c$.
\end{definition}

\begin{remark}
Note furthermore that all basis vectors $\hat \tau \in T$ with $\deg\hat \tau < -1$
are of the form $\hat \tau = \<Xi>_i^j\tau$ (or similar with $\<Xi>$ replaced by $\<XiS>$) 
for some $i,j$ and $\tau$ as in Definition~\ref{def:admissible}. As a consequence, admissible
models in our sense admit a canonical decomposition $\PPi = \PPi^+ + \PPi^-$
as in Assumption~\ref{as:decomposition} below by setting $\PPi^+ \<Xi>_i^j\tau = \PPi \<XiS>_i^j\tau$.
This is a crucial fact which allows us to use Theorem~\ref{thm:reconstructDomain}. 
\end{remark}
We also define the following notion:

\begin{definition}\label{def:ruleInvariant}
Consider the rule $R_{\invariant}$ given by $R_{\invariant}(\<Xi>_i^j) = R_{\invariant}(\dCI) = R(\dCI_i) =  \{\1\}$ and
\begin{equs}
R_{\invariant}(\CI) &= \{\1,\CI, \CI_i,  \CI_i \<Xi1>\} \cup \{ \<Xi>_i^j \CI^k\,:\, k \le i-j\}\;,\\
R_{\invariant}(\CI_i) &= \{\1,\<Xi>\}\;,\quad 
R_{\invariant}(\hat \CI) = \{\1,\CI\}\;.
\end{equs}
We write $\CT_{\invariant} \subset \CT$ for the sector spanned by trees conforming to $R_{\invariant}$
and we call $\CT_{\invariant}$ the \textit{translation invariant sector}.
\end{definition}

Given any admissible model for $(\TT,\GG)$, we then write $\CK$, $\CK_i$, $\CK_\Neu$, $\CK_{\Neu,i}$ for the corresponding
integration operators as defined in \cite[Sec.~5]{regularity}. 
We now provide a complete description of the renormalised models $\hPeps$ we will use for this
regularity structure. From now on, whenever we do not explicitly mention a model on $(\CT,\CG)$,
we assume that we talk about the specific (random) model $\hPeps$, and not a general admissible model.
In particular, all the variants of the spaces $\CD^\gamma$ used below will be the spaces associated to
this model. In order to specify the $\eps$-dependence, we will write $\CR_\eps$ for the reconstruction
operator associated to the model $\hPeps$. Theorem~\ref{theo:model} below shows that $\hPeps$ converges to some limiting model
$\PPi$ as $\eps \to 0$, giving rise to a reconstruction operator $\CR_0$. Whenever we simply write $\CR$, 
it denotes the reconstruction operator for a generic admissible model.

Writing $\one^D$ for the indicator function of $\R\times D$,
we first define a family of non-renormalised (random) models $\PPi_\eps$ as the canonical lift for the 
``noise'' given by
\begin{equ}
\PPi_\eps \XX ij = \xi_\eps \eps^{\alpha(i,j)}\;,\quad 
\PPi_\eps \<XiS> = \one^D\xi_\eps \;,\quad 
\PPi_\eps \<XiS1> = \one^D\eta_\eps\;,
\end{equ}
with $\alpha(i,j)$ as in \eqref{e:defalphaij}. 
We then define an intermediate model $\tilde \PPi_\eps$ related to $\PPi_\eps$ by
\begin{equs}[e:renormalisedModel]
\tilde \PPi_\eps \<XiSdIXiS>  &= \PPi_\eps \<XiSdIXiS> - \eps^{-\f d2} \<XiIXic> \one^D - \sum_{i=1}^d \eps^{1-{d\over 2}} \bigl(c_{i,0} \delta_{\d_{i,0} D}
+c_{i,1} \delta_{\d_{i,1} D}\big)\;,\\
\tilde \PPi_\eps \<XiSIXi> &= \PPi_\eps \<XiSIXi> - \eps^{-\f d2} \<XiIXic>\one^D \;,\\
\tilde \PPi_\eps \<XiS1IXiX> &= \PPi_\eps \<XiS1IXiX>- \eps^{-\f d2}\<XiIXic> \one^D \PPi_\eps \X - \eps^{1-\f d2} \<XiIXiXc> \one^D  \;,
\end{equs}
as well as $\tilde \PPi_\eps\<XiS1> \CI_i(\<Xi>) = \tilde \PPi_\eps (\<XiS1IXiX>-\<XiXS1IXi>)$.
We furthermore impose admissibility which forces us to set 
$\tilde \PPi_\eps \X\tau = \PPi_\eps\X\cdot \tilde \PPi_\eps \tau$.
 One can verify that the only remaining basis vector of $\CT$ of negative degree and 
not belonging to the translation invariant sector is $\<XiSdIXiSX>$. However,
this element will never be needed for our considerations, so we do not need to
specify the action of $\tilde \PPi_\eps$ on it.

We now construct a renormalised model $\hPeps$ from $\tilde \PPi_\eps$ by applying a slight modification of the BPHZ
renormalisation procedure \cite{Ajay,Lorenzo} to the translation invariant sector (which can be viewed
as a regularity structure in its own right, generated by the rule $R_{\invariant}$).
Writing $\CT_-$ for the subspace of $\CT_{\invariant}$ consisting of symbols of strictly negative degree,
we will define $\hPeps$ by an expression of the form
\begin{equ}
\hPeps \tau = \bigl(g_\eps \otimes \tilde \PPi_\eps\bigr)\Deltam\tau\;,
\end{equ}
where $\Delta^-$ is a certain linear operator from $\CT$ into $\Alg \CT_-\otimes  \CT$, with $\Alg \CT_-$
the free unital algebra generated by $\CT_-$, and $g_\eps$ is a character of $\Alg \CT_-$, which is 
canonically identified with an element of the dual space $\CT_-^*$.
In the BPHZ renormalisation procedure, one should choose $g_\eps$ of the form 
\begin{equ}[e:choicegeps]
g_\eps^\BPHZ = (\E \tilde \PPi_\eps) \circ \CA\;,
\end{equ}
where $\CA \colon \Alg \CT_- \to \Alg \CT$ is the ``twisted antipode'' \cite{Lorenzo} and
$\E \tilde \PPi_\eps$ is the character of $\Alg \CT$ determined by $(\E \tilde \PPi_\eps)(\tau) = \E (\tilde \PPi_\eps\tau)(\phi)$,
for any fixed test function $\phi$ with $\int z^k \phi(z)\,dz = \delta_{k,0}$ for $|k|$ small enough. (In our case $|k| \le 1$ suffices.) 

Recall that $\Delta^-$ is an ``extraction / contraction'' operator which iterates over all possible ways of extracting 
divergent subsymbols of its argument, so for example
\begin{equ}
\Delta^- \<Xi2IXiIXi11> = \<Xi2IXiIXi11> \otimes \one + \one \otimes \<Xi2IXiIXi11>
+ \<IXiIXi11> \otimes \<Xi2> + \<Xi2IXi> \otimes \<IXi11>\;.
\end{equ}
The twisted antipode behaves in a somewhat similar fashion, in this case
\begin{equ}
\CA \<Xi2IXiIXi11> = - \<Xi2IXiIXi11> 
+ \<IXiIXi11>\, \<Xi2> + \<Xi2IXi> \, \<IXi11>\;.
\end{equ}

\begin{remark}
Inspection of the rule $R_{\invariant}$ and our degree assignment shows that the
basis vectors of $\CT_{-}$ are given by 
\begin{equs}
\CT_{-} = \Vec\Big\{&\<Xi>, \<XiX>, \<Xi> \X^2, \<Xi1>, \<Xi1X>, \<Xi11> , 
\<Xi11X>, \<Xi2> , \<Xi22> , \<Xi21>,  \<IXi>, \<Xi1IXi>,  \<Xi1IXiX>, \<Xi1XIXi> 
,\<Xi1IXi11> , \<Xi1IXi1>,  \<Xi21IXi>, \label{e:listBasis}\\
& \<Xi21XIXi>,  \<Xi21IXiX>, \<Xi2IXi>, \<Xi21IXi11>, \<IXi^2>, \<IXiIXiX>, \<IXiIXi11>, \<Xi2IXi^2>,  \<Xi2IXiIXiX>, \<Xi2IXiIXi11>,\<Xi21IXi1IXi>,\<Xi1IXi1IXi>, \<Xi2IXiIXi1IXi>, \<Xi1IXi2IXi^2>, \<IXiIXi1IXi>, \<Xi1IIXi^2>\Big\}\;.
\end{equs}
\end{remark}

It will be convenient to have an alternative degree assignment $\degb$ on $\CT_{\invariant}$  which better reflects
the self-similarity properties of our objects given by setting 
\begin{equ}
\degb\XX ij = \alpha(i,j) - {d+2\over 2}\;,
\end{equ}
and then extending it as usual.
Instead of choosing the character $g_\eps$ as in the BPHZ specification \eqref{e:choicegeps}, it
will be convenient to choose it in a way such that, for some constants $C(\tau)$ that are \textit{independent} of $\eps$,
one has
\begin{equ}
g_\eps(\tau) = 
\left\{\begin{array}{cl}
	0 & \text{if $\degb\tau > 0$,} \\
	-\eps^{\degb \tau} C(\tau) & \text{if $\degb\tau < 0$.}
\end{array}\right.
\end{equ}
For $\degb\tau = 0$, we choose $g_\eps(\tau) = -C(\tau)$ whenever the symbol $\tau$ contains an ``accent'',
i.e.\ one of the noises $\XX ij$ with $i+j > 0$. The only symbols of negative degree without accents that 
appear in the translation invariant sector of our regularity structure and that contain at least two copies
of the noise $\<Xi>$ are $\<IXi^2>$ and $\<IXiXIXi>$, which are of vanishing degree $\degb$ 
in dimensions $2$ and $3$ respectively, and will be considered separately below.

The constants $C(\tau)$ themselves are chosen to coincide with the ones appearing in 
Theorems~\ref{theo:LLN} and~\ref{theo:main} with the convention that for any symbol $\tau$
the constant $C(\tau)$ is also written as the same symbol $\tau$, but drawn in red and with its ``accents'' stripped.
For example, we set
\begin{equ}
C(\<Xi1IXi>) = C(\<Xi2IXi>) = C(\<Xi1IXi1>) = C(\<Xi1IXi11>) = \<XiIXic>\;.
\end{equ}
We also set $C(\tau) = 0$ whenever $\tau$ contains only one instance of the noise, namely we set
\begin{equ}[e:zerotypes]
\<Xic> = \<IXic> = 0\;.
\end{equ} 
(In general, we should also set $C(\tau) = 0$ if $\tau$
is of the form $\tau = \CI(\tau')$ for some $\tau'$, but the only symbol of negative
degree of this type appearing in this work is $\<IXi>$ which is already covered by \eqref{e:zerotypes}.)
Other constants that will be relevant for our analysis (in dimension $d=3$) are given by
$\<XiXIXic> = 0$ as well as
\begin{equs}
\<IXi^2c> &= \int P(z)P(z')\kappa_2(z,z')\,dz\,dz',\quad
\<IXiXIXic> = \int x\,P(z)P(z')\kappa_2(z,z')\,dz\,dz',\\
\multicol{2}{\<IXiIXiIXic> = \int P(-z)P(-z')P(z-z'')\kappa_3(z,z',z'')\,dz\,dz'\,dz'',}\\
\multicol{2}{\<XiIIXi^2c> = \int P(-z)P(-z')P(z'')\kappa_3(z,z',z'')\,dz\,dz'\,dz''.}
\end{equs}
The convergence of integrals corresponding to $\<IXi^2c>$, $\<IXiIXiIXic>$ and $\<XiIIXi^2c>$
in dimension $3$ can easily be verified by using our assumption on the 
cumulants and the self-similarity of the heat kernel. The convergence of the integral for $\<IXiXIXic>$
is more subtle since $\degb \<IXiXIXi> = 0$. As a consequence, although $\kappa_2(z,z')$ decays fast enough when $\|z-z'\|$ is large,
the function $z \mapsto x P^2(z)$ is homogeneous of (parabolic) degree $-5$ and is therefore not 
absolutely integrable at large scales. However, since it is odd under $(t,x) \mapsto (t,-x)$, additional
cancellations occur and the integral should be interpreted as
\begin{equ}[e:integralX]
{}\<IXiXIXic> = {1\over 2}\int x\,\bigl(P(z')-P(z)\bigr)P(z) \kappa_2(z,z')\,dz\,dz'\;,
\end{equ}
which does converge absolutely, so we set $g_\eps(\<IXiXIXi>) = - \<IXiXIXic>$ in dimension $3$.

In dimension $2$, $\degb\<IXiXIXi> > 0$, but $\degb\<IXi^2> = 0$ and the expectation of 
$\tilde \PPi_\eps \<IXi^2>$ diverges logarithmically and the expression $\<IXi^2c>$ given above
fails to converge. We then have no choice but to set
\begin{equ}
g_\eps(\<IXi^2>) = -\<IXi^2c>_{\beps}\;,\quad \<IXi^2c>_{\beps} \eqdef \int K_\eps(z)K_\eps(z')\kappa_2(z,z')\,dz\,dz'\;,
\end{equ} 
where $K_\eps(t,x) = \eps^2 K(\eps^2 t, \eps x)$. We then have the following preliminary result.

\begin{proposition}\label{prop:convTransl}
The model $\hPeps$ restricted to the translation invariant sector converges as $\eps \to 0$ to the
BPHZ model $\hat \PPi$ such that $\hat \PPi \<Xi> = \xi$ and $\hat \PPi \XX ij = 0$ for $i+j > 0$.
\end{proposition}

\begin{proof}
Consider the ``BPHZ model'' $\PPi_\eps^\BPHZ$ on $\CT_{\invariant}$ given by 
\begin{equ}
\PPi_\eps^\BPHZ \tau = (g_\eps^\BPHZ \otimes \PPi_\eps)\Deltam \tau\;,
\end{equ}
with $g_\eps^\BPHZ$ defined as in \eqref{e:choicegeps}.
Thanks to Proposition~\ref{prop:conditionAjay} below, our noises are such that the
norm of \cite[Def.~A.18]{Ajay} is finite, uniformly in $\eps$, for the cumulant homogeneity
described in \eqref{e:defCum}.

It therefore follows from \cite[Thm~2.33]{Ajay} that $\PPi_\eps^\BPHZ$ converges to $\hat \PPi$.
Since furthermore the action of the character group of $\CT_-$ on the space of
admissible models is continuous \cite{Lorenzo}, it suffices to show that one can write 
\begin{equ}[e:deltageps]
\hPeps = (\delta g_\eps \otimes \PPi_\eps^\BPHZ)\Deltam \;, 
\end{equ}
for some $\delta g_\eps \in \CT_-^*$ with $\lim_{\eps \to 0} \delta g_\eps = 0$.
For this, the following result is useful.

\begin{lemma}\label{lem:additive}
Let $g,\bar g \in \CT_-^*$ such that furthermore $g(\tau) = \bar g(\tau) = 0$ for 
every $\tau$ of the type \eqref{e:zerotypes}. Then, one has the identity
\begin{equ}
\big(g \otimes (\bar g \otimes \id)\Deltam\big)\Deltam \tau
= \big((g+\bar g) \otimes \id\big)\Deltam \tau\;,
\end{equ}
for all $\tau \in \CT$.
\end{lemma}
\begin{proof}
It was shown in \cite{regularity,Lorenzo} that 
\begin{equ}
\big(g \otimes (\bar g \otimes \id)\Deltam\big)\Deltam \tau
= \big((g \otimes \bar g)\Deltam \otimes \id\big)\Deltam \tau\;,
\end{equ}
where $\Deltam \colon \Alg \CT_- \to \Alg \CT_- \otimes \Alg \CT_-$ is an extraction / contraction
operator defined just like above, but extended multiplicatively to $\Alg \CT_-$ and such that 
only those terms are kept that actually belong to $\Alg \CT_- \otimes \Alg \CT_-$ (i.e.\ every factor
needs to be of negative degree on both sides of the tensor product).
Inspection of the list \eqref{e:listBasis} reveals that in our case, the only situation
in which we have a ``subsymbol'' of negative degree appearing in any of our symbol in such a way that the 
contracted symbol is still of negative degree is when the subsymbol contains only one noise.
We conclude that 
\begin{equ}
(g \otimes \bar g)\Deltam \tau
= (g \otimes \bar g)(\tau \otimes \one + \one \otimes \tau) = g(\tau) + \bar g(\tau)\;,
\end{equ}
and the claim follows.
\end{proof}

We conclude from Lemma~\ref{lem:additive} that \eqref{e:deltageps} holds with 
$\delta g_\eps = g_\eps - g_\eps^\BPHZ$, so that it remains to show that 
$\lim_{\eps \to 0} \delta g_\eps(\tau) = 0$ for every $\tau \in \CT_-$. 
For elements $\tau$ of the form $\tau = X^k \Xi_i^j$ we have 
$g_\eps(\tau) =  g_\eps^\BPHZ(\tau) = 0$. For all other elements $\tau$ with
$\degb \tau \le 0$, a simple scaling argument shows that $g_\eps^\BPHZ(\tau)$ is given 
by the exact same formula as $g_\eps(\tau)$, except that all instances of the heat kernel $P$ are
replaced by $K_\eps$, where
\begin{equ}
K_\eps(t,x) \eqdef \eps^{d} K(\eps^2t,\eps x)\;.
\end{equ}
Note that $K_\eps$ coincides with $P$ in a parabolic ball of radius $\CO(1/\eps)$
around the origin and vanishes outside of another ball of radius $\CO(1/\eps)$.

This in particular shows that 
\begin{equ}[e:deltag1]
\delta g_\eps(\<Xi1IXi>) = \eps^{-d/2} \int \bigl(P(z) - K_\eps(z)\bigr)\,\kappa_2(0,z)\,dz\;.
\end{equ}
Since $\kappa_2(0,z)$ decreases like $\|z\|^{-2\bar c} = \|z\|^{-(d+2+2\delta)}$ for large $z$
and $P$ decreases like $\|z\|^{-d}$, it follows that 
\begin{equ}
|\delta g_\eps(\<Xi1IXi>)| \lesssim \eps^{d/2 + 2\delta}\;,
\end{equ}
which of courses converges to $0$. The symbols $\tau$ differing from 
$\<Xi1IXi>$ only by the placement of their accents then also converge since 
$\delta g_\eps(\tau)$ is given by the same expression as \eqref{e:deltag1}, except for
being multiplied by a higher power of $\eps$.

Turning now to $\<IXi^2>$ (which only appears when $d \in \{2,3\}$), it follows from 
\cite[Lem.~6.8]{Jeremy} that $\delta g_\eps(\<IXi^2>)$ is a sum of 
terms of the form
\begin{equ}
\eps^{2-d} \int  \bigl(P(z) - K_\eps(z)\bigr) G(z)\,dz\;,
\end{equ}
where $|G(z)| \lesssim (1+\|z\|)^{2-2\bar c} =  (1+\|z\|)^{-(d+2\delta)}$.
It follows that $|\delta g_\eps(\<IXi^2>)| \lesssim \eps^{2\delta}$ as desired, and 
$\delta g_\eps(\<IXiIXi11>)$ is controlled in the same way by a higher power.
In dimension $3$, $\degb \<IXiIXiX> = 0$ and it was shown in 
\eqref{e:integralX} that $\<IXiIXiXc>$ converges absolutely, which immediately implies
that $\delta g_\eps(\<IXiIXiX>) \to 0$. 

To deal with the symbol $\<Xi1IXi1IXi>$ (again with $d \in \{2,3\}$), 
we first note that Assumption~\ref{ass:kappa} implies the bound 
\begin{equ}
|\kappa_3(z_1,z_2,z_3)| \lesssim \rho(\|z_1-z_2\|)\,\rho(\|z_2-z_3\|)\,\rho(\|z_1-z_3\|)\;.
\end{equ}
We also note that for any $\kappa \in [0,d]$ one has the bound
\begin{equ}[e:boundDiffPK]
|P(z) - K_\eps(z)| \lesssim \eps^{\kappa} \big(1 \wedge \|z\|^{\kappa-d}\big)\;.
\end{equ}
This allows us to make use of \cite[Thm~4.3]{Feynman}. Since $\degb\<Xi1IXi1IXi> = {2-d\over 2}$
we apply the bound \eqref{e:boundDiffPK} with $\kappa = \delta + {d-2\over 2}$
which, in the notation of \cite{Feynman}, yields a bound of the type
\begin{equ}
|\delta g_\eps(\<Xi1IXi1IXi>)| \le \eps^\delta |\Pi^{\bar K,\bar R} \Gamma|\;,
\end{equ}
for some $\eps$-dependent kernel assignment $(\bar K,\bar R) \in \CK_0^- \times \CK_0^+$
with bounds that are independent of $\eps$ and the Feynman diagram
\begin{equ}
\Gamma 
= 
\begin{tikzpicture}[style={thick},baseline=-0.1cm]
\node[dot] (l) at (0,0) {};
\node[dot] (r) at (3.5,0.6) {};
\node[dot] (d) at (3.5,-0.6) {};
\draw (l) -- (r) node [midway, above, sloped] {\small$(\delta - {1\over 2}, - d -2)$};
\draw (l) -- (d) node [midway, below, sloped] {\small$(\delta - {1\over 2}, -\delta - {d+2\over 2})$};
\draw (d) -- (r) node[midway,right=-0.1] {\small$(\delta -d-{1\over 2}, -\delta - {3d+2\over 2})$};
\draw[thick,red] (l) -- ++(180:0.5);
\end{tikzpicture} \;.
\end{equ}
Here, the first coordinate of the label for each edge denotes its 
small-scale degree assignment while the second coordinate denotes its large-scale
degree assignment. It is straightforward to verify that the small-scale degree assignment
for this diagram satisfies the assumption of \cite[Prop.~2.3]{Feynman}, so that it does
not require renormalisation. Furthermore, the large-scale degree assignment is seen to satisfy
the assumption of \cite[Thm~4.3]{Feynman}, which guarantees 
that the integral converges absolutely and is bounded independently of $\eps$,
so that $|\delta g_\eps(\<Xi1IXi1IXi>)| \le \eps^\kappa$. 

The symbols
$\<IXiIXi1IXi>$ and $\<Xi1IIXi^2>$ (in dimension $3$) which have vanishing degree $\degb$, 
can be dealt with using the same technique,
leading to the bound $|\delta g_\eps(\<IXiIXi1IXi>)| + |\delta g_\eps(\<Xi1IIXi^2>)|\lesssim \eps^{1/2}$ by using the Feynman diagrams 
\begin{equ}
\begin{tikzpicture}[style={thick},baseline=-0.1cm]
\node[dot] (l) at (0,0) {};
\node[dot] (r) at (2,0.6) {};
\node[dot] (d) at (2,-0.6) {};
\node[dot] (rr) at (4,0) {};
\draw (l) -- (r) node [midway, above, sloped] {\scriptsize$(-3,-3)$};
\draw (l) -- (d) node [midway, below, sloped] {\scriptsize$(-3,-3)$};
\draw (d) -- (r) node[midway] {\scriptsize$(\delta-{1\over 2}, -\delta - {5\over 2})$};
\draw (r) -- (rr) node [midway, above, sloped] {\scriptsize$(\delta-{1\over 2}, -\delta - {5\over 2})$};
\draw (d) -- (rr) node [midway, below, sloped] {\scriptsize$(\delta-{1\over 2},  -\delta - 5)$};
\draw[thick,red] (l) -- ++(180:0.5);
\end{tikzpicture} \;,\quad
\begin{tikzpicture}[style={thick},baseline=-0.1cm]
\node[dot] (l) at (0,0) {};
\node[dot] (r) at (2,0.6) {};
\node[dot] (d) at (2,-0.6) {};
\node[dot] (rr) at (4,0) {};
\draw (l) -- (r) node [midway, above, sloped] {\scriptsize$(-3,-3)$};
\draw (l) -- (d) node [midway, below, sloped] {\scriptsize$(0,-{5\over 2})$};
\draw (d) -- (r) node[midway] {\scriptsize$(\delta-{1\over 2}, -\delta - {5\over 2})$};
\draw (r) -- (rr) node [midway, above, sloped] {\scriptsize$(\delta-{1\over 2}, -\delta - {5\over 2})$};
\draw (d) -- (rr) node [midway, below, sloped] {\scriptsize$(\delta-{7\over 2}, -\delta - {11\over 2})$};
\draw[thick,red] (l) -- ++(180:0.5);
\end{tikzpicture} \;,
\end{equ}
for bounding $\<IXiIXi1IXi>$ and
\begin{equ}
\begin{tikzpicture}[style={thick},baseline=-0.1cm]
\node[dot] (l) at (0,0) {};
\node[dot] (r) at (3,1) {};
\node[dot] (d) at (3,-1) {};
\node[dot] (c) at (2,0) {};
\draw (l) -- (r) node [midway, above, sloped] {\scriptsize$(\delta-{1\over 2}, -\delta - {5\over 2})$};
\draw (l) -- (d) node [midway, below, sloped] {\scriptsize$(\delta-{1\over 2}, -\delta - {5\over 2})$};
\draw (l) -- (c) node[near end,above=-0.1cm] {\scriptsize$(0,-{5\over 2})$};
\draw (c) -- (r) node [midway, below, sloped] {\tiny$(-3,-3)$};
\draw (c) -- (d) node [near start, below, sloped] {\tiny$(-3,-3)$};
\draw (d) -- (r) node [midway, right] {\scriptsize$(\delta-{1\over 2}, -\delta - {5\over 2})$};
\draw[thick,red] (l) -- ++(180:0.5);
\end{tikzpicture} \;,
\end{equ}
for bounding $\<Xi1IIXi^2>$. All three are easily seen to satisfy both the
small-scale and large-scale integrability conditions. 

The remaining three symbols $\symbol\tau \in \{\<Xi21IXi1IXi>,\<Xi2IXiIXi1IXi>, \<Xi1IXi2IXi^2>\}$ in the list \eqref{e:listBasis} are all such
that $\degb \symbol\tau > 0$, so we need to show that $g_\eps^\BPHZ(\symbol\tau) \to 0$.
This can in principle be shown again by using the bounds from \cite{Feynman}. 
A ``cheaper'' way of showing that $g_\eps^\BPHZ(\symbol\tau) \to 0$ is to note that in all three cases 
we can make use of a combination of Proposition~\ref{prop:conditionAjay} (used in the same
way as in the proof of Proposition~\ref{prop:boundsUseful}) and \eqref{e:choicegeps} to conclude that 
one can build a regularity structure 
$\hat \CT$ extending $\CT_\invariant$ (by adding additional ``noises'' representing $\eta_\eps^{(\alpha)}$
for suitable choices of $\alpha$) such that, for every $\symbol\tau$, one can find $\kappa > 0$ and a symbol 
$\symbol{\hat\tau} \in \hat \CT$ such that $\deg \symbol{\hat\tau} > 0$ and such that 
\begin{equ}
g_\eps^\BPHZ(\symbol\tau) = \eps^\kappa \E \bigl(\PPi_\eps^\BPHZ \symbol{\hat\tau}\bigr)(\phi)\;,
\end{equ}
for some suitable fixed test function $\phi$. Since we know from \cite[Thm~2.33]{Ajay} 
that the BPHZ renormalised model $\PPi_\eps^\BPHZ$ converges
(so in particular remains uniformly bounded), we conclude that $g_\eps^\BPHZ(\symbol\tau) \to 0$ as required.
\end{proof}

\begin{theorem}\label{theo:model}
The random models $\hPeps$ converge weakly to a limiting admissible model $\hat \PPi$ 
which, on the translation invariant sector, is given by the BPHZ lift of
\begin{equ}
\hat \PPi \<Xi> = \xi\;,\quad
\hat \PPi \sXi_i^j = 0\;, \quad i+j > 0\;.
\end{equ}
For the remaining symbols, it is given by the unique admissible model such that
\begin{equ}
\hat \PPi \<XiS> = \xi \one_{\R \times D}\;,\quad
\hat \PPi \<XiS1> = 0\;,
\end{equ}
as well as $\hat \PPi \tau = 0$ for any symbol $\tau$ containing the noise $\<XiS1>$.
\end{theorem}

\begin{remark}
Note that only symbols $\sXi_i^j$ with $i \ge j$ appear in our regularity structure.
\end{remark}

\begin{proof}
Convergence on the translation invariant sector was already shown in Proposition~\ref{prop:convTransl},
so it only remains to consider the non-translation invariant symbols of negative degree.
In dimension $3$, these are $\<XiS>$, $\<XiS1>$, $\<XiS11>$, $\<dIXiS>$, $\<XiSdIXiS>$, $\<XiS1IXi>$, $\<XiS1IXi1>$, $\<XiS1IXiX>$
and $\<XiXS1IXi>$. (There is also the symbol $\<XiS1> \CI_i(\<Xi>)$,
but applying the model to it yields the
exact same distribution as when applying it to $\<XiS1IXiX> - \<XiXS1IXi>$.)

The convergence on the remainder of the regularity structure is shown in the next section, but we collect 
the various parts of the proof here.
Convergence of $\hPeps\<XiS>$, $\hPeps\<XiS1>$ and $\hPeps\<XiS11>$ to $\xi \one^D$ and $0$ in 
$\CC^{-\f 52-\kappa}$, $\CC^{-1-\kappa}$ and  $\CC^{-{3\over 2}-\kappa}$ respectively follows from
Corollary~\ref{cor:tightSimple} and Corollary~\ref{cor:tight}.

Convergence of $\hPeps\<XiS1IXi>$ essentially follows from \cite[Thm~2.31]{Ajay}, noting that the bound
for $\tau = \<Xi1IXi>$ does not require any derivative of the test function in this case, so that we immediately obtain 
the required bound by noting that  $\bigl(\hPeps\<XiS1IXi>\bigr)(\phi) = \bigl(\hPeps\<Xi1IXi>\bigr)(\one_{\R \times D}\phi)$.
The reason why this is so is that the only point in the proof where derivatives of the test function
could potentially appear is in the bound \cite[Eq.~A.29]{Ajay} in the proof of \cite[Thm~A.32]{Ajay}. 
By the definition of $R(\mathbf{S})$, these derivatives can only hit a test function in a situation where 
$\tau$ contains a connected subtree containing its root and of degree less than $-{d+2\over 2}$. This
is not the case for $\<Xi1IXi>$.

Convergence of $\hPeps\<XiS1IXi1>$, $\hPeps\<XiS1IXiX>$ and $\hPeps\<XiXS1IXi>$  also follows in the same way.
Regarding $\<XiSdIXiS>$, we note that $\<XiSdIXiS> = \<Xi1IXib> + \<XiS1IXi>$ by \eqref{e:deffunnysymbol}
and we already obtained convergence of $\hPeps\<XiS1IXi>$. Convergence of $\hPeps\<Xi1IXib>$ is the content of
Theorem~\ref{theo:correctionTermLp} below.
The convergence of $\hPeps\<dIXiS>$ follows from Corollary~\ref{cor:boundaryTerm}, combined
with the usual Schauder estimates for integration against $K$.
\end{proof}

Most of Section~\ref{sec:model} is devoted to filling in the missing parts in 
the proof of Theorem~\ref{theo:model}, namely the proofs of Theorem~\ref{theo:main1}
and Theorem~\ref{theo:correctionTermLp}.

\subsection{Description of \texorpdfstring{$\v$}{v0}}

We now have the notation in place to formalise the discussion 
given above regarding the local behaviour of $\v$ and $\vv$. Recall the definition 
\eqref{e:defv0} of $\v$ which we rewrite in integral form as
\begin{equ}[e:defveps0]
\v = K_\Neu \bigl(G(\u) \, \one_+^D\xi_\eps\bigr)\;.
\end{equ}
Depending on context, we 
will model $\v$ by
three different modelled distributions $V_\eps^{(0)}$, $\tilde V_\eps^{(0)}$ and $\hat V_\eps^{(0)}$.

Regarding $V^{(0)}_\eps$, we use Proposition~\ref{prop:commuteSpecialCase} 
which guarantees that one can find
$\Phi_\eps \in \CC^{{3\over 2}-\kappa,-{1\over 2}-\kappa}$ such that, setting\footnote{with the convention $\CL = \CL_1$ for $\CL_1$ the Taylor lift \eqref{e:Taylor}}
\begin{equ}[e:defV0]
V^{(0)}_\eps = \Phi_\eps + \CL G(\u) \,\CK_\Neu(\one_+ \<XiS>)
+ G'(\u) \d_i \u \,\CK_{\Neu,i}( \one_+ \<XiS>)\;,
\end{equ}
where the integration operators $\CK_\Neu$, and $\CK_{\Neu,i}$ should be interpreted as the 
natural extensions of the corresponding integration operators described in Proposition~\ref{prop:commuteSpecialCase}, 
one has
\begin{equ}
V^{(0)}_\eps \in \CD^{{3\over 2}-\kappa,-{1\over 2}-\kappa}\;,\qquad \CR_\eps V^{(0)}_\eps \restr \R_+\times D = \v\;,
\end{equ}
provided that we consider admissible models $\hat\PPi_\eps$ with 
$\hat\PPi_\eps \<XiS> = \xi_\eps \one_{\R \times D}$. Note that $V^{(0)}_\eps$ depends on $\eps$ via
the choice of model $\hat\PPi_\eps$. 
Furthermore, Theorem~\ref{theo:model} and Lemma~\ref{lem:convergeRestrict}
guarantee that the bounds on $V^{(0)}_\eps$ are uniform over $\eps$ and
that one has $V^{(0)}_\eps \to V^{(0)}$ in $\CD^{{3\over 2}-\kappa,-{1\over 2}-\kappa}$
with respect to the renormalised models $\hat \PPi_\eps$ and the limiting model $\hat \PPi$.

It is natural at this stage, as already mentioned earlier, to define the symbol $\<IXib>$ as the element of $\CT$
of degree $-{1\over 2}-\kappa$ given by
\begin{equ}[e:deffunnysymbol]
\<IXib> = \<dIXiS>  - \<IXi>\;,
\end{equ}
which is indeed consistent with \eqref{e:defPPiepscorr}.
With this notation, Propositions~\ref{prop:diffKernel} and~\ref{prop:R+} combined with \eqref{e:defV0} 
guarantee the existence of 
a function $\tilde \Phi_\eps \in \CC^{{3\over2}-\kappa,w}$ with $w = \bigl(-{1\over 2}-\kappa,{1\over 2}-\kappa,-{1\over 2}-\kappa\bigr)$ in the sense of \cite[Def.~3.2]{Mate} such that if we set
$\tilde V^{(0)}_\eps = \tilde V^{(0,1)}_\eps + \tilde V^{(0,2)}_\eps$ with 
\begin{equs}[e:defV0tilde]
\tilde V^{(0,1)}_\eps &= \CL G(\u) \<IXib>\;, \\
\tilde V^{(0,2)}_\eps &= \CL G(\u) \<IXi>
+ G'(\u) \d_i\u \,\CK_i(\<Xi>) + \CL \tilde \Phi_\eps\;,
\end{equs}
then we have $\tilde V^{(0,1)}_\eps \in  \CD^{\f32-\kappa}$, $\tilde V^{(0,2)}_\eps \in  \CD^{\f32-\kappa,w}$ and $\bigl(\CR_\eps \tilde V^{(0)}_\eps\bigr)(\phi) = \bigl(\CR_\eps V^{(0)}_\eps\bigr)(\phi)$ for all test functions
$\phi$ supported in $\R_+ \times D$. As before, all these objects converge in the limit $\eps \to 0$ provided that the
underlying models converge.
To see this, we note that we can choose
\begin{equs}
\tilde \Phi_\eps &= \Phi_\eps + \one_+ G(\u)K_\Neu((\one_+ - 1)\one_D \xi_\eps) \\
&\quad + \one_+^D G'(\u) \d_i\u\, K_i * \bigl((\one_+^D - 1) \xi_\eps\bigr) \\
&\quad + \one_+^D G'(\u) \d_i\u \, K_{\d,i} \bigl(\one_+^D\xi_\eps\bigr)\;,
\end{equs}
and then apply Proposition~\ref{prop:R+} to bound the first term and Proposition~\ref{prop:diffKernel}
to bound the remaining two terms.

On the other hand, we define $\hat V^{(0)}_\eps$ by setting
\begin{equ}
\hat V^{(0)}_\eps = \CK \bigl(\CL_2 G(\u) \one_+\<Xi>\bigr) + \CL \one_+^D\Bigl(K_\d \bigl(G(\u)\xi_\eps\one_+^D\bigr)
- K \bigl(G(\u)\xi_\eps\one_{+}^{D^c} \bigr)\Bigr)\;,
\end{equ}
where we use the convention that $\u$ is extended outside of $\R_+ \times D$ in any way that makes it
globally $\CC^{3}$. (For positive times, this is possible since the 
extension of $u_0$ to the whole space by suitable reflections is of class $\CC^{3}$ by our assumptions.
For negative times, this is possible by Whitney's extension theorem \cite{Whitney}.) 
Here we made a slight abuse of notation: the operator $\CK$ should be interpreted in the sense of 
\cite[Sec.~4.5]{Mate} with $\hat \CR \bigl(\CL_2 G(\u) \one_+\<Xi>\bigr) = G(\u) \one_+\xi_\eps$, which 
converges as $\eps \to 0$ by an argument very similar to that of Proposition~\ref{prop:commuteSpecialCase}, 
combined with the fact that $G(\u)$ is $\CC^3$.

Note that we have
\begin{equ}
\CR_\eps \hat V^{(0)}_\eps = \CR_\eps \tilde V^{(0)}_\eps = \CR_\eps V^{(0)}_\eps = \v \,\qquad \text{on $\R_+ \times D$}\;.
\end{equ}
The modelled distribution $\hat V^{(0)}_\eps$ 
exhibits rather singular behaviour near the boundary of the domain, but by Proposition~\ref{prop:diffKernel}
 it converges as $\eps \to 0$ in 
$\CD^{1,w}$ with $w = \big(-{1\over 2} - \kappa\big)_3$ (we use again the notation $(\eta)_3 = (\eta,\eta,\eta)$). 
It has the advantage however of not involving the integration map $\dCI$ and the restricted
noise $\<XiS>$, so it
is purely described in terms of the ``translation invariant'' part of the regularity structure, i.e.\
the sector generated by $\CI$ and $\<Xi>$. We summarise the above discussion with the following statement.

\begin{lemma}\label{lem:V0}
We have $\hat V^{(0)}_\eps \in \CD^{1,(-\f12-\kappa)_3}_{-\f12-\kappa}$, 
$\CR_\eps \hat V^{(0)}_\eps = \v$ on $\R_+ \times D$, and $\hat V^{(0)}_\eps$ is of the form
\begin{equ}
\hat V^{(0)}_\eps = \one_+ \big(G(\u) \<IXi> + G'(\u)\,\nabla \u \<IXiX>\big) + \Phi^{(0)}_\eps \1\;, 
\end{equ}
for some continuous function $\Phi^{(0)}_\eps$. Furthermore, 
$\lim_{\eps \to 0} \hat V^{(0)}_\eps = \hat V^{(0)} \in \CD^{1,(-\f12-\kappa)_3}$
with $\CR_0 \hat V^{(0)} = P_\Neu \bigl(\one_+^D G(\u)\xi\bigr)$. 
\end{lemma}

\subsection{Description of \texorpdfstring{$\vv$}{v1}}

Recall that $\vv$ was defined in \eqref{e:defv1} as the solution to an inhomogeneous linear equation
with \textit{inhomogeneous} boundary conditions (et least in the Neumann case).
Regarding $\vv$, we would like to describe it by a modelled distribution 
$V^{(1)}_\eps$ given by 
\begin{equ}[e:wantedV1]
V^{(1)}_\eps = \one_+^D \CP_\Neu\one_+ \bigl(\CL(G'(\u)) \tilde V_\eps^{(0)} \<XiS1>\bigr) + \w\star\1\;,
\end{equ}
with $\CP_\Neu$ as in \eqref{e:defCP}
and where $\w\star$ would be given by the solution to
\begin{equs}[e:defw2]
\d_t \w\star &= \Delta \w\star + {1\over 2} G''(\u)\, W_{1,\eps}\;,\\
W_{1,\eps} &= (\v)^2 \sigma_\eps  - 2\<XiIXic>G(\u)\v- \eps^{1-{d\over 2}} \<XiIXi^2c>G^2(\u)\;,
\end{equs}
endowed with \textit{homogeneous} Neumann boundary conditions.
The reason why the identity $\CR_\eps V^{(1)}_\eps = \vv$ holds (on $\R_+\times D$ as usual) is as follows. By \eqref{e:defV0tilde},
\begin{equ}[e:tVXiS]
 \tilde V_\eps^{(0)} \<XiS1>
 = \CL(G(\u)) \bigl(\<Xi1IXib> + \<XiS1IXi>\bigr)
 + G'(\u)\d_i \u \<XiS1>\CK_i(\<Xi>) + \CL \tilde \Phi_\eps \<XiS1>\;.
\end{equ} 
It then follows from the definition \eqref{e:renormalisedModel} of the renormalised model and the fact that
$\CR_\eps(\tilde V_\eps^{(0)}) = \v$ that, applying the reconstruction operator to this expression yields
\begin{equs}
 \CR_\eps \bigl(\tilde V_\eps^{(0)} \<XiS1>\bigr)
&=  \one^D \v  \eta_\eps - \sum_{i=1}^d G(\u) \eps^{1-{d\over 2}} \bigl(c_{i,0} \delta_{\d_{i,0} D} 
+ c_{i,1} \delta_{\d_{i,1} D} \bigr) \\
&\quad  - \eps^{-{d\over 2}} \<XiIXic> \one^D G(\u) - \eps^{1-{d\over 2}} \one^D G'(\u) \,\scal{\<XiIXiXc>,\nabla \u}\;.
\end{equs}
We then note that the two terms appearing on the second line, when multiplied by another factor
of $G(\u)$, are exactly the additional two terms appearing on the first line of \eqref{e:defv1},
while the singular term involving Dirac masses on the boundary of $D$, when hit by $K_\Neu$, is responsible
for the non-homogeneous boundary conditions.

The problem with such a definition is that it yields a description of 
$V^{(1)}_\eps$ in terms of symbols involving $\<XiS1>$, while the general convergence results
of \cite{Ajay} require translation invariance of the noise objects, which is not the case here.
If we were to try to improve the situation by replacing $\one_+ \<XiS1>$ by $\one_+^D\<Xi1>$ in
\eqref{e:wantedV1}, then we immediately run into the problem that the behaviour of this modelled
distribution on the boundary of $D$ is too singular for the general results of
\cite{Mate} to apply. This is not just a technicality: this singular behaviour is precisely what is responsible
for the additional boundary renormalisation!

Instead, we define $V^{(1)}_\eps$ as a sum of terms that is ``equivalent'' to the
definition \eqref{e:wantedV1} in the sense that they reconstruct the same function, 
but such that each of the terms can be controlled 
in a slightly different, situation-specific, way.

We first deal with the boundary correction by setting
\begin{equ}[e:exprV10]
V^{(1,0)}_\eps = \one_+^D\CL\,K_\d \bigl(\one_+ G'(\u) \CR_\eps\bigl(\tilde V_\eps^{(0)} \<XiS1>\bigr)\bigr)\;.
\end{equ}
Since $\tilde V_\eps^{(0)} \<XiS1>\in \CD^{\f12-2\kappa, w}$ with $w = (-\f32-2\kappa,-\f12-2\kappa,-\f32-2\kappa)$
and since it belongs to a sector of regularity $-\f32-2\kappa$, its reconstruction belongs to $\CC^{-\f32-2\kappa}$.
It then follows from Proposition~\ref{prop:diffKernel} that $V^{(1,0)}_\eps  \in \CD^{2,(0)_3}$.

We now break up $\tilde V_\eps^{(0)}$ in \eqref{e:wantedV1} as in \eqref{e:defV0tilde} and deal with the first term. 
By Proposition~\ref{prop:specialSymbol}, we can find $V^{(1,1)}_\eps \in \CD^{2-2\kappa,\bar w}$ with
$\bar w = \big({1\over 2}-2\kappa,{1\over 2}-2\kappa,0\big)$ of the form
\begin{equ}[e:exprV11]
V^{(1,1)}_\eps = \one_+^D\,G'G(\u) (\hat\PPi_\eps \<IXib>) \<IXi1> + \Phi_\eps^{(1,1)}\;,
\end{equ}
with $\Phi_\eps^{(1,1)}$ taking values in the classical
Taylor polynomials, and such that 
\begin{equ}
\CR_\eps V^{(1,1)}_\eps = K \big((G'G)(\u)\one_+ \hat \PPi_\eps \<Xi1IXib>\big) = 
\CR_\eps \CK\one_+ \bigl(\CL(G'(\u)) \tilde V_\eps^{(0,1)} \<XiS1>\bigr)\;.
\end{equ}

The second term is dealt with similarly. As a consequence of Proposition~\ref{prop:specialSymbol2} with
$g_1 = G'G(\u)$, $g_{2,i} = G'(\u)^2 \d_i \uu$ and $g_3 = G'(\u) \tilde \Phi_\eps$ (with $\tilde \Phi_\eps$
as in \eqref{e:wantedV1}),
we can find $V^{(1,2)}_\eps \in \CD^{2-2\kappa,\bar w}$ such that
\begin{equ}
\CR_\eps V^{(1,2)}_\eps = K \one_+ \bigl(G'G(\u) \hat \PPi_\eps \<XiS1IXi>
+  G'(\u)^2 \d_i \uu \CR_\eps(\<XiS1> \CK_i(\<Xi>)) + G'(\u)\tilde \Phi_\eps \hat \PPi_\eps \<XiS1>\bigr)\;,
\end{equ}
and such that furthermore $V^{(1,2)}_\eps$ takes values in the translation invariant sector and is of the form
\begin{equ}[e:exprV12]
V^{(1,2)}_\eps = \one_+^D \big((G'G)(\u) \<IXi1IXi> +  \tilde \Phi_\eps \<IXi1>\big) + \Phi^{(1,2)}_\eps\;,
\end{equ}
for some $\Phi^{(1,2)}_\eps$ taking values in the Taylor polynomials.
In order to define $V^{(1,3)}_\eps$, we make use of the following lemma.

\begin{lemma}\label{lem:boundPhi}
Let $\phi_\eps$ be such that on $\R_+ \times D$ one has the identity
\begin{equ}
\v = G(\u)\, \hat \PPi_\eps \<IXi> + \phi_\eps\1\;,
\end{equ}
and one has $\phi_\eps(t,x) = 0$ for $t < 0$ or $x \not \in D$.
Then, for any $\alpha \in [0,1)$, one has the bound $\E \|\phi_\eps\|_{\alpha+\f12-\kappa,w} \lesssim \eps^{-\alpha}$ 
with $w = \bigl(\alpha-\f12-\kappa\bigr)_3$. 
\end{lemma}

\begin{proof}
We decompose $\phi_\eps$ as
\begin{equs}
\phi_\eps &= K_\d \bigl(\one_+^D G(\u) \xi_\eps\bigr)+  G(\u) K \bigl((1-\one_+^D) \xi_\eps\bigr)\\
&\quad + \Big(K \bigl(\one_+^D G(\u) \xi_\eps\bigr) - G(\u) K \bigl(\one_+^D \xi_\eps\bigr)\Big) \;,
\end{equs}
and we treat the three terms separately. The first two terms are estimated by 
combining Proposition~\ref{prop:diffKernel} with Corollary~\ref{prop:boundXieps}.

The bound on the last term follows from combining Proposition~\ref{prop:boundsUseful} with  Corollary~\ref{cor:commute}.
To apply the latter, we set $\theta = \kappa$ (small enough) and $\chi = \alpha - {5\over 2}-\kappa$, which yields a 
bound in $\CC^{\kappa+2}$ on
\begin{equ}
K \bigl(\one_+^D G(\u) \xi_\eps\bigr) - G(\u) K \bigl(\one_+^D \xi_\eps\bigr)
- \sum_i G'(\u)\d_i \u K_i \bigl(\one_+^D \xi_\eps\bigr)\;.
\end{equ}
Since $K_i$ gains three derivatives, the term $K_i \bigl(\one_+^D \xi_\eps\bigr)$ itself satisfies the required bound
and we are done.
\end{proof}

Recalling $W_{1,\eps}$ as defined in \eqref{e:defw2}, it follows from Lemma~\ref{lem:boundPhi} combined with the
definition of the renormalised model that we can rewrite it as
\begin{equ}[e:writeW1]
W_{1,\eps} = G^2(\u) \hat \PPi_\eps \<Xi2IXi^2> + 2 G(\u) \phi_\eps  \hat \PPi_\eps \<Xi2IXi>
+ \phi_\eps^2 \hat \PPi_\eps \<Xi2>\;.
\end{equ}
As a consequence of Proposition~\ref{prop:boundsUseful} below (with $\alpha = \kappa$), combined with
Lemma~\ref{lem:boundPhi} (with $\alpha = \f12$), we conclude that one has
\begin{equ}[e:limWeps]
\lim_{\eps \to 0} \|W_{1,\eps}\|_{-{1\over 2}-3\kappa} = 0\;.
\end{equ}
Indeed, the reconstruction theorem \cite[Thm~4.9]{Mate} and the multiplication rules \cite[Lem.~4.3]{Mate} imply 
that if $g \in \CC^{\gamma,(\eta)_3}$ for $\eta \le 0$ and $\gamma > 0$, 
and $\zeta \in \CC^{\beta}$ with $\beta \le 0$ and $\gamma + \beta > 0$ then, provided that 
$\eta + \beta > -1$, one has $g \zeta \in \CC^{\eta + \beta}$.
(View $\zeta$ as the constant function in a regularity structure containing only one symbol of 
degree $\beta$ and apply the reconstruction theorem to $g\zeta$.)

Since $-\f12 - 3\kappa > -1$, we can multiply such a distribution by the indicator
function of $\R_+ \times D$. It follows that, setting
\begin{equ}[e:exprV13]
V^{(1,3)}_\eps \eqdef \f12 \CL K_\Neu \bigl(\one_{D}^+ G''(\u)\, W_{1,\eps}\bigr)\;,
\end{equ}
we have $\lim_{\eps \to 0} V^{(1,3)}_\eps = 0$ in $\CD^{\f32-3\kappa}$
and furthermore $\CR_\eps V^{(1,3)}_\eps = \w\star$.
Combining these definitions, we set
\begin{equ}
\hat V^{(1)}_\eps = V^{(1,0)}_\eps+V^{(1,1)}_\eps+V^{(1,2)}_\eps+V^{(1,3)}_\eps\;.
\end{equ}
Summarising this discussion, one has the following result.

\begin{lemma}\label{lem:V1}
We have $\hat V^{(1)}_\eps \in \CD^{\f32-3\kappa,(0)_3}$, 
$\CR_\eps \hat V^{(1)}_\eps = \vv$ and $\hat V^{(1)}_\eps$ is of the form
\begin{equ}[e:exprV1]
\hat V^{(1)}_\eps = \one_+^D \,G'G(\u) \<IXi1IXi> + \one_+^D\,G'(\u) \Phi^{(0)}_\eps \<IXi1>+ \Phi^{(1)}_\eps\;, 
\end{equ}
for some $\Phi^{(1)}_\eps$ taking values in the Taylor polynomials, and where $\Phi^{(0)}_\eps$ is as in Lemma~\ref{lem:V0}.
Furthermore, $\lim_{\eps \to 0} \hat V^{(1)}_\eps = \hat V^{(1)} \in \CD^{\f32-3\kappa,(0)_3}$
with $\CR_0 \hat V^{(1)} = 0$. 
\end{lemma}

\begin{proof}
Collecting \eqref{e:exprV10}, \eqref{e:exprV11}, \eqref{e:exprV12} and \eqref{e:exprV13}, we see that \eqref{e:exprV1} holds, but with $\Phi^{(0)}_\eps$ replaced by
$G(\u) (\hat\PPi_\eps \<IXib>) + \tilde \Phi_\eps$. These two expression are seen to coincide on $\R_+ \times D$ by 
comparing Lemma~\ref{lem:V0} with \eqref{e:defV0tilde}. 

The only statement we haven't shown yet is that  $\CR_0 \hat V^{(1)} = 0$. Since we already know by \eqref{e:limWeps}
that $W_{1,\eps}$ converges to $0$ and since $\CR_\eps \hat V^{(1)}_\eps = \CR_\eps V^{(1)}_\eps$, it remains by \eqref{e:wantedV1}
to show that $\lim_{\eps \to 0} \CR_\eps \bigl(\tilde V_\eps^{(0)} \<XiS1>\bigr) = 0$. This in turn
is immediate from \eqref{e:tVXiS} when combined with Theorem~\ref{theo:model} which guarantees that the 
limiting model vanishes on all accented symbols.
\end{proof}

\subsection{Formulation of the fixed point problem}

Introduce now a modelled distribution $\Veps$ and, using the shorthand
$\VV = \VV^{(0)} + \VV^{(1)} + \Veps$, consider the fixed point problem
\begin{equs}
\Veps &= \CP_\Neu \one_+^D \Big(H_\eta'(\u)\VV +   \CL(G'(\u))(\VV - \VV^{(0)})\<Xi1> +\CL(G'(\u)\uu) \<Xi11>  \\
&\quad + {1\over 2}\CL(G''(\u))(\VV^2-(\VV^{(0)})^2)\<Xi2>+
\CL(G''(\u)\uu) \VV \<Xi21> \label{e:absCLT} \\&\quad
+   {1\over 2}\CL(G''(u^{(0)})(u^{(1)})^2) \<Xi22> \Big)
+\CL P_\Neu \one_+^D\bigl(\Rteps(\CR_\eps \VV, \varsigma) + \tilde R_\eps^{(d)}(\bar \varsigma)\bigr) \;,\quad   
\end{equs}
 where $\tilde R_\eps^{(d)} = 0$ for $d = 1$ and 
\begin{equs}
\tilde R_\eps^{(2)}(\bar \varsigma) &=  \eps^\kappa {1\over 2} G^2G'' \<IXi^2c>_{\beps}\, \bar \varsigma\;, \label{e:deftildeReps}\\
\tilde R_\eps^{(3)}(\bar \varsigma) &=  \eps^\kappa (G''G)\Bigl({1\over 2}\<IXi^2c> G + \eps G' \bigl(\<IXiIXiIXic> G + \scal{\<IXiXIXic>,\nabla \u} + \<IXi^2c>\uu\bigr)\Bigr) \bar \varsigma\;.
\end{equs}

\begin{remark}\label{rem:integration}
We will set this up as a fixed point problem in the space $\CD^{\f32-3\kappa,(0)_3}$. Since 
$\deg \<Xi1>< -1$ and $\deg \<Xi11> < -1$ (in $d \in \{2,3\}$ for the latter), this forces us to rely on Theorem~\ref{thm:reconstructDomain} below for the reconstruction of the
right hand side of \eqref{e:absCLT} and to combine this with \cite[Lem.~4.12]{Mate} to provide an interpretation 
for the integration operator $\CK$ appearing in the definition \eqref{e:defCP} of $\CP_\Neu$.
\end{remark}

\begin{remark}
Recall that the definition \eqref{e:defReps} of the remainder $\Rteps$ involves an arbitrary exponent $\alpha$.
We henceforth fix a choice $\alpha = \alpha(d)$ depending on the dimension, namely
\begin{equ}[e:defAlphaExponents]
\alpha(1) = {5\over 4}\;,\qquad \alpha(2) = {9\over 4}
\;,\qquad \alpha(3) = {11\over 4}\;.
\end{equ} 
All further statements about $\Rteps$ hold for this particular choice.
\end{remark}

We claim that with this definition and provided that we consider the renormalised model constructed in Section~\ref{sec:defModel}, \eqref{e:absCLT} admits a unique solution in $\CD^{\f32-3\kappa,(0)_3}$
and, provided that we set $\varsigma = \eps^\alpha \xi_\eps$ and $\bar \varsigma = \eps^{-1-\kappa} \sigma_\eps$,
one has $v_\eps = \CR_\eps \VV$.
The reason for the appearance of $\tilde R_\eps^{(d)}$ is to cancel out some additional 
unwanted terms arising from the renormalisation procedure. 
Before this, we formulate a technical lemma, where we write $\|\theta\|_\alpha$ for the $\CC^\alpha$ norm
of the function\slash distribution $\theta$ on $D_T = [0,T] \times D$ with $T$ as in Theorem~\ref{theo:main}.

\begin{lemma}\label{lem:boundsR}
Let $w, \bar w$ with $\|w\|_{L^\infty} + \|\bar w\|_{L^\infty} \le \eps^{-d/2}$ on the 
domain $D_T$ and let $\kappa \in (0,{1\over 4})$.
Writing $X = \|\varsigma\|_{-\frac{1}{2} + 2\kappa} + 1$,
one has the bounds
\begin{equs}
\|\Repsone(w,\varsigma)\|_{-\half+2\kappa} &\lesssim \eps^{1/4}\bigl(1 + \|w\|_{\frac{1}{2} - \kappa}\bigr)^3 X \;,\\
\|\Repsone(w,\varsigma) - \Repsone(\bar w,\varsigma)\|_{-\half+2\kappa} &\lesssim \eps^{1/4}\|w-\bar w\|_{\frac{1}{2} - \kappa} \bigl(1 + \|w\|_{\frac{1}{2} - \kappa} + \|\bar w\|_{\frac{1}{2} - \kappa}\bigr)^3 X\;,
\end{equs}
for some proportionality constants depending only on $\u$, $G$ and $H$.
In dimensions $2$ and $3$, we set $X = \|\varsigma\|_{L^p} + 1$ (for any fixed $p \in [1,\infty]$) and
we have the bounds
\begin{equs}
\|\Reps(w,\varsigma)\|_{L^p} &\lesssim \eps^\kappa\bigl(1 + \eps^\beta\|w\|_{L^\infty}\bigr)^3 X \;,\\
\|\Reps(w,\varsigma) - \Reps(\bar w,\varsigma)\|_{L^p} &\lesssim \eps^{\kappa+\beta}\|w-\bar w\|_{L^\infty} \bigl(1 + \eps^\beta\|w\|_{L^\infty} + \eps^\beta\|\bar w\|_{L^\infty}\bigr)^3 X\;,
\end{equs}
with $\beta(2) = {1\over 4} - {\kappa\over 3}$ and $\beta(3) = {7\over 12} - {\kappa\over 3}$
for $\kappa$ sufficiently small.
\end{lemma}

\begin{proof}
The case of dimensions $2$ and $3$ is straightforward to verify since all bounds are uniform.
In dimension $1$, the first term of \eqref{e:defReps}
is easy to bound. To bound the second term, we use the fact that composition with a smooth function
is a (locally) Lipschitz continuous operation in $\CC^{\frac{1}{2}-\kappa}$, combined with the fact that the product is 
continuous as a bilinear map from $\CC^{\frac{1}{2}-\kappa} \times \CC^{-{1\over 2}+2\kappa}$ into $\CC^{-\half+2\kappa}$, see \cite{BookChemin}
or \cite[Thm~13.16]{RP}.
\end{proof}

\begin{proposition}\label{prop:localSol}
Fix an initial condition $u_0$, a final time $T < 1$ and nonlinearities $G$ and $H$, all 
as in Theorem~\ref{theo:main}, as well as the random model $\hPeps$ as 
defined in Section~\ref{sec:defModel}.
Choose $\varsigma \in L^p$ with $p = (d+2)/\underline c$ (for $d \in \{2,3\}$)
or $\varsigma \in \CC^{2\kappa-{1\over 2}}$ (for $d=1$), as well
as $\bar \varsigma \in \CC^{-{1\over 2}-2\kappa}$  (for $d \in \{2,3\}$).
Then, the the fixed point problem \eqref{e:absCLT} admits a unique local 
solution $\Veps$ in $\CD^{\f32-3\kappa,(0)_3}$.
Furthermore, bounds on the solution are uniform over $\eps \in [0,1]$ and over  
$\varsigma$, $\bar \varsigma$ in bounded balls in their respective spaces.

Furthermore, for $\eps = 0$ and $\hat \PPi$ as in Theorem~\ref{theo:model}, 
the solution $\V$ is such that $\bar v = \CR_0 \V$ solves 
\begin{equ}[e:limitEquation]
\d_t \bar v = \Delta \bar v + H_\eta'(\u) (\bar v + \CR_0 \hat V^{(0)})\;,
\end{equ}
with homogeneous boundary conditions, where $\hat V^{(0)}$ is as in Lemma~\ref{lem:V0}. 
In particular, $\lim_{\eps \to 0} \VV = \hat V$ is such that 
$\CR_0 \hat V$ coincides with the process $v$  defined in \eqref{e:deflimitv}.
\end{proposition}

\begin{remark}
In the case $d=1$, there is no condition on $\bar \varsigma$ since the fixed point problem 
does not depend on it.
\end{remark}

\begin{proof}
We first consider the case $d=3$.
Note first that, for any $\CC^4$ function $\tilde G$, we have $\CL(\tilde G(\u,\uu)) \in \CD^{2,(0)_3}$.
Since $\hat V^{(0)}_\eps \in \CD^{1,(-{1\over 2}-\kappa)_3}$ by Lemma~\ref{lem:V0}
and $\hat V^{(1)}_\eps \in \CD^{{3\over 2}-3\kappa,(0)_3}$ by Lemma~\ref{lem:V1},
it then follows from \cite[Lem.~4.3]{Mate} that, for $\Veps \in \CD^{{3\over 2}-3\kappa,(0)_3}$, 
all the terms appearing after $\CP_\Neu \one_+^D$ in the right hand side of \eqref{e:absCLT} belong to 
$\CD^{{1\over 2}-4\kappa,(-{3\over 2}-2\kappa)_3}$, provided that $\kappa$ is sufficiently
small. In particular, the operator $\CK$ (defined as described in Remark~\ref{rem:integration})
maps this continuously into $\CD^{2,(0)_3}$, with arbitrarily small norm for small time intervals.

Furthermore, the reconstruction operator of Theorem~\ref{thm:reconstructDomain} continuously 
maps the space $\CD^{{1\over 2}-4\kappa,(-{3\over 2}-2\kappa)_3}$ into $\CC^{-{3\over 2}-2\kappa}$,
which is then mapped continuously into $\CC^{2,(0)_3}$ by $K_\d$ by 
Proposition~\ref{prop:diffKernel}, and therefore into $\CD^{2,(0)_3}$ by the Taylor lift $\CL$, 
again with arbitrarily small norm for small time intervals as a consequence of 
the bound \eqref{e:wantedBound} which also holds for $K_\d$.

Note now that by Corollary~\ref{cor:tight}, we have
$\E \|\xi_\eps\|_{\kappa-2}\lesssim \eps^{-{1\over 2}-\kappa}$.
As a consequence of \eqref{e:defveps0}, we conclude from this that 
$\E |\v|_{L^\infty} \lesssim \eps^{-{1\over 2}-\kappa}$.
Since $\beta > {1\over 2}$, it follows from Lemma~\ref{lem:boundsR} that,
for $p = (d+2)/\underline c$,
\begin{equ}
\|\Rteps(\CR_\eps \VV, \varsigma) \|_{-\underline c}
\lesssim \|\Reps(\CR_\eps \VV + \eps^{-{1\over 2}}\uu, \varsigma) \|_{L^p}
\lesssim \|\varsigma\|_{L^p}\;,
\end{equ}
uniformly over bounded sets for the model $\hPeps$ and over bounded sets for 
$\Veps +\hat V^{(1)}_\eps$ in $\CD^{{3\over 2}-3\kappa,(0)_3}$.
Since $\u$ and $\uu$ are bounded in $\CC^1$, it is immediate from \eqref{e:deftildeReps} that
one has a bound of the type
\begin{equ}
\|\tilde R_\eps^{(3)}(\bar\varsigma) \|_{-{1\over 2}-2\kappa}
\lesssim \|\bar\varsigma\|_{-{1\over 2}-2\kappa}\;.
\end{equ}
In particular, the argument of $P_\Neu$ appearing in the last term on the right hand 
side of \eqref{e:absCLT} is mapped continuously by $P_\Neu$ into $\CC^{{3\over 2}-3\kappa,(0)_3}$,
again with arbitrarily small norm when considering a short enough time interval.
Furthermore, all of these expressions are locally Lipschitz continuous (with similar bounds)
as a function of $\Veps$ in $\CD^{{3\over 2}-3\kappa,(0)_3}$ and of the model $\hPeps$, uniformly over $\eps \in [0,1]$
which yields the first claim over a short enough time (but bounded from below independently of $\eps$) interval.
This can be maximally extended as usual, and the claim follows from the fact that we know  
a priori that solutions to \eqref{e:limitEquation} do not explode.

The second claim is straightforward by simply setting $\eps = 0$ and applying the reconstruction operator
to both sides of \eqref{e:absCLT}.
The case of $d=2$ is virtually identical, noting in particular that even though $\<IXi^2c>_{\beps}$ diverges
in this case, it only does so logarithmically and is therefore compensated by the factor $\eps^\kappa$ in
\eqref{e:deftildeReps}. We leave the verification of the case $d=1$ to the reader.
\end{proof}

\begin{proposition}\label{prop:idenSol}
Let $\varsigma_\eps = \eps^\alpha \xi_\eps$, $\bar \varsigma_\eps = \eps^{2-d-\kappa} \sigma_\eps$, 
and define $\hPeps$ as in Section~\ref{sec:defModel}.
Then, the assumptions of Proposition~\ref{prop:localSol} are satisfied and we have
$\varsigma_\eps, \bar \varsigma_\eps \to 0$ in their respective spaces.
Furthermore, for any $\eps > 0$, the modelled distribution $\VV$ constructed in 
Proposition~\ref{prop:localSol} is such that $\CR_\eps \VV$ coincides with the process $v_\eps$ 
defined in \eqref{e:defveps}.
\end{proposition}

\begin{proof}
We first show that the assumptions of Proposition~\ref{prop:localSol} are satisfied.
The fact that the random models $\hPeps$ are uniformly bounded (in probability) as $\eps \to 0$
and converge in probability to $\hat \PPi$ is the content of Theorem~\ref{theo:model}.
By the second part of Assumption~\ref{ass:kappa} combined with stationarity, we 
furthermore see that 
\begin{equ}
\E\|\varsigma_\eps\|_{L^p}^p = \E \|\eps^\alpha \xi_\eps\|_{L^p}^p = \eps^{p\alpha - {(d+2)p\over 2}} T \E |\eta(0)|^p
\lesssim \eps^{p/4}\;,
\end{equ}
when $d \in \{2,3\}$. For $d=1$, we have
\begin{equ}
\E\|\varsigma_\eps\|_{2\kappa-{1\over 2}} = \|\eps_\eps^{(1/4)}\|_{2\kappa-{1\over 2}}
\le \|\eps_\eps^{(1/2-3\kappa)}\|_{2\kappa-{1\over 2}}\;,
\end{equ}
which converges to $0$ in probability by Corollary~\ref{cor:tight}.
We also conclude from Corollary~\ref{cor:tight} and our definitions that, 
for $d \in \{2,3\}$,  $\|\bar \varsigma_\eps\|_{-{1\over 2}-2\kappa} = 
\|\eta_\eps^{({d\over 2}-1+\kappa)}\|_{-{1\over2}-2\kappa} \to 0$
in probability.

It remains to show that solutions coincide with $v_\eps$.
This is a special case of the general result obtained in \cite{Ilya} and could in principle 
also be obtained in a way similar to \cite{HP}. We present a short derivation here in 
order to remain reasonably self-contained.

The powercounting of the various symbols appearing in our structure depends on the dimension, so we
first restrict ourselves to the case $d = 3$, which is the one with the largest number of terms of 
negative degree appearing.
Combining \eqref{e:absCLT} with Lemmas~\ref{lem:V0} and~\ref{lem:V1}, we conclude that if we take for $\Veps$
any solution to \eqref{e:absCLT}, there exist functions $v$ and $\nabla v$ such that, for $\Phi_\eps^{(0)}$ and $\Phi_\eps^{(1)}$ 
as in Lemma~\ref{lem:V1}, the following identities hold on $\R_+\times D$: 
\begin{equs}
\hat V_\eps^{(0)} &= G \<IXi> + \Phi_\eps^{(0)}\1 + G' \nabla \u \<IXiX>,\\
\hat V_\eps^{(1)} &= \Phi_\eps^{(1)}\1 + GG' \<IXi1IXi> + G' \Phi^{(0)}_\eps \<IXi1>,\\
\hat V_\eps &= G \<IXi> + v\1 + GG' \<IXi1IXi> + G' \nabla \u \<IXiX> 
 + G'\uu \<IXi11> 
 + 
  {G''G^2\over 2} \<IXi2IXi^2> + G' v \<IXi1> + \nabla v\,\X\;.
\end{equs}
Developing the argument of 
$\CP_\Neu \one_+^D$ in \eqref{e:absCLT} up to order $0$, we conclude that it is given by
\begin{equs}
H_\eta'&G\<IXi> + H_\eta' v\1 + {1\over 2}G'' \nabla^2\u \,\sXi \X^2 + G' \big(v - \Phi_\eps^{(0)}\big)\<Xi1>+ G(G')^2 \<Xi1IXi1IXi> 
\\ & + (G')^2\nabla \u \<Xi1IXiX> 
  + (G')^2\uu\<Xi1IXi11> + (G')^2v\<Xi1IXi1> +G' \nabla v \<Xi1X> + {1\over 2}G^2 G'G'' \<Xi1IXi2IXi^2>
\\ & + GG''\nabla \u \<Xi1XIXi> + G'' v \nabla \u \<Xi1X> \label{e:RHS}
+ G' \uu \<Xi11> + G' \nabla\uu \<Xi11X>\\ & + G'' \nabla\u \uu \<Xi11X> 
 + {1\over 2} G^2G'' \<Xi2IXi^2> +{1\over 2}G''v^2\<Xi2> + {1\over 2}G'' (\uu)^2 \<Xi22> 
+ GG'' \uu \<Xi21IXi>\\ &+ GG'' \nabla \uu \<Xi21XIXi> + GG''v \<Xi2IXi> + G'G'' (\uu)^2 \<Xi21IXi11>+ G^2 G'G''\<Xi2IXiIXi1IXi>
 + GG'G''\nabla \u  \<Xi2IXiIXiX> \\&
+ GG'G'' \uu \<Xi2IXiIXi11> + G'' \uu v \<Xi21> + GG'G'' \uu \<Xi21IXi1IXi> + G'G'' \uu \nabla \u \<Xi21IXiX>\;.
\end{equs}
At this point, we apply the results of \cite{Ilya}. Comparing \cite[Eq.~2.20]{Ilya} with
\cite[Def.~3.20]{Ilya} and \cite[Thm.~3.25]{Ilya}, we see that each term appearing on the
right hand side generates a counterterm for the renormalised equation. Each of these terms is of
the form $\hat F(v,\nabla v, \u, \nabla \u, \uu, \nabla \uu)\symbol\tau$ for some
function $\hat F$ and some symbol $\symbol \tau$.
The counterterm generated by any such term is then
obtained precisely by simply replacing $\symbol\tau$ by the corresponding
renormalisation constant and by interpreting the first two arguments of $\hat F$ as the 
value and gradient of the actual solution (after reconstruction).

\begin{remark}
One may worry that we are not quite in the framework of \cite{Ilya} because of the 
special treatment of $\hat V_\eps^{(0)}$ and $\hat V_\eps^{(1)}$. This however is due to
purely analytical reasons that only affect the boundary behaviour. The computation of the
renormalisation terms on the other hand is a purely algebraic affair which is 
not affected by this. The boundary conditions of $v_\eps$ however \textit{are} affected 
by our decomposition and need to be determined separately.
\end{remark}

It follows that, in dimension $3$, the solution $v_\eps = \CR_\eps \hat V_\eps$ to the fixed point problem 
with the renormalised model satisfies in $\R_+ \times D$ the PDE
\begin{equs}
\d_t v_\eps &= \Delta v_\eps + H_\eta' v_\eps + G\,\xi_\eps +  (v_\eps + \eps^{-1/2} \uu)G'\eta_\eps
 + {1\over 2} (v_\eps + \eps^{-1/2} \uu)^2G''\sigma_\eps \\
 &- \eps^{-3/2} G'G \<XiIXic> - \eps^{-1/2} G(G')^2 \<XiIXiIXic> - \eps^{-1/2}\scal{\nabla \u, (G')^2 \<XiIXiXc> + GG'' \<XiXIXic>} \\
 &- (v_\eps+\eps^{-1/2}\uu) (GG')' \<XiIXic> - {\eps^{-1/2}\over 2} G^2G'' \<XiIXi^2c>  - {\eps^{-1}\over 2} G^2G'' \<IXi^2c>\, \sigma_\eps\\
 & -G'' \bigl(G^2G' \<IXiIXiIXic> + GG' \scal{\<IXiXIXic>, \nabla \u}  + GG' \uu \<IXi^2c>\bigr)\sigma_\eps \\
 & +\hat R_\eps^{(3)}(v_\eps, \varsigma) + \tilde R_\eps^{(3)}(\bar \varsigma)\;.
\end{equs}
Furthermore, both $\CR_\eps \hat V_\eps$ and $\v = \CR_\eps \hat V_\eps^{(0)}$ have homogeneous boundary conditions, 
so that the boundary conditions of $v_\eps$ coincide with those of $\vv = \CR_\eps \hat V_\eps^{(1)}$. 

By \eqref{e:deftildeReps} and since we chose $\bar \varsigma = \eps^{-1-\kappa} \sigma_\eps$, there is a cancellation between
$\tilde R_\eps(\bar \varsigma)$ and some of the other terms appearing in this equation.
Since furthermore $\<XiXIXic> = 0$, we obtain
\begin{equs}
\d_t v_\eps &= \Delta v_\eps + H_\eta' v_\eps + G\,\xi_\eps +  (v_\eps + \eps^{-1/2} \uu)G'\eta_\eps
 + {1\over 2} (v_\eps + \eps^{-1/2} \uu)^2G''\sigma_\eps \\
 &- \eps^{-3/2} G'G \<XiIXic> - \eps^{-1/2} G(G')^2 \<XiIXiIXic> - \eps^{-1/2}\scal{\nabla \u, (G')^2 \<XiIXiXc>} \\
 &- (v_\eps+\eps^{-1/2}\uu) (GG')' \<XiIXic> - {\eps^{-1/2}\over 2} G^2G'' \<XiIXi^2c>  +\hat R_\eps^{(3)}(v_\eps, \varsigma) \;,
\end{equs}
which, when combining with the definition of $\Psi$ given in \eqref{e:ubar1} and 
the fact that $\varsigma = \eps^\alpha \xi_\eps$, precisely coincides with \eqref{e:CLT}. Since its initial condition and
boundary condition coincide as well, this completes the proof of the claim.

In dimension $2$, a similar argument (but taking less terms into account) yields
\begin{equs}
\d_t v_\eps &= \Delta v_\eps + H_\eta' v_\eps + G\,\xi_\eps +  (v_\eps + \uu)G'\eta_\eps
 + {1\over 2} (v_\eps + \uu)^2G''\sigma_\eps \\
 &- \eps^{-1} G'G \<XiIXic> - G(G')^2 \<XiIXiIXic> - \scal{\nabla \u, (G')^2 \<XiIXiXc> + GG'' \<XiXIXic>} \\
 &- (v_\eps+\uu) (GG')' \<XiIXic> - {1\over 2} G^2G'' \<XiIXi^2c>  - {1\over 2} G^2G'' \<IXi^2c>_{\beps}\, \sigma_\eps
  +\hat R_\eps^{(2)}(v_\eps, \varsigma)  + \tilde R_\eps^{(2)}(\bar \varsigma)\;.
\end{equs}
Again, the term $\tilde R_\eps^{(2)}$ precisely cancels the term proportional to 
$\<IXi^2c>_{\beps}\, \sigma_\eps$, so that this again coincides with \eqref{e:CLT}.
In dimension $1$, an even simpler argument shows that
\begin{equs}
\d_t v_\eps &= \Delta v_\eps + H_\eta' v_\eps + G\,\xi_\eps +  v_\eps G'\eta_\eps
 + {1\over 2} v_\eps^2G''\sigma_\eps \\
 &- \eps^{-1/2} G'G \<XiIXic> - v_\eps\, (GG')' \<XiIXic>
  +\hat R_\eps^{(1)}(v_\eps, \varsigma)\;,
\end{equs}
which again coincides with \eqref{e:CLT}, noting that in this case one has $\uu= 0$.
\end{proof}

\section{Convergence of models}
\label{sec:model}

In order to show convergence of the models, we apply the general result of
\cite[Thm~2.31]{Ajay}. This result shows that if one considers the ``BPHZ lifts''
of a sequence of smooth and stationary stochastic processes $\xi_n$ as given in \cite[Thm~6.17]{Lorenzo}
then, provided that one has uniform bounds of a suitable ``norm'' of $\xi_n$ and under a few 
relatively weak additional algebraic assumptions, the resulting sequence of models converges to a 
limit, provided that $\xi_n \to \xi$ weakly in probability.

\subsection{Cumulant homogeneity assignments}

In this section, we define 
\begin{equ}[e:defeta]
\eta^{(\alpha)}_\eps(t,x) = \eps^{-\alpha} \eta (\eps^{-2}t,\eps^{-1}x)\;,
\end{equ}
and we often use $z = (t,x)$ for space-time coordinates. The exponents $\alpha$ will
always be chosen in $\big(\underline c,{d+2\over 2}\big]$. 
Our aim is to obtain a suitable bound independent of $\eps$ for joint cumulants of
the form
\begin{equ}
\kappa_p\bigl(\eta^{(\alpha_1)}_\eps(z_1),\ldots,\eta^{(\alpha_p)}_\eps(z_p)\bigr)\;.
\end{equ}
Given a finite collection of at least two space-time points 
$z = \{z_a\}_{a \in A}$, we again consider the corresponding
labelled binary tree $\tree_z = (T,\scale)$ as in Section~\ref{sec:coal}, with the leaves of $T$ identified
with the index set $A$. Recall that the nodes $V_T$ of $T$ are given by subsets of $A$,
with inner nodes $\mathring V_T$ given by subsets with at least two elements and the root of $T$
given by $A$ itself.

Recall from \cite[Def.~A.14]{Ajay} the following definition.
\begin{definition}\label{def:consistent}
A ``consistent cumulant homogeneity'' consists, for each 
finite index set $A$, each binary tree $T$ over $A$ as above, and each choice of 
indices $\alpha \colon A \to [\underline c, {d+2\over 2}]$, a function
$\c_T^{(\alpha)}\colon \mathring V_T\to \R_+$ satisfying furthermore
\begin{claim}
\item For every $B\subset A$, $\sum_{v\in \mathring V_T: v \cap B \neq \emptyset} \c_T^{(\alpha)}(v) \ge \sum_{a\in B} \alpha_a$.
\item For every $u \in \mathring V_T$, $\sum_{v\in \mathring V_T: v \subset u} \c_T^{(\alpha)}(v) \le \sum_{a \in u} \alpha_a$.
\item If $|A| \ge 3$, then for every $u \in \mathring V_T$ with $|u| \le 3$, 
one has $\sum_{v\in \mathring V_T: v \subset u}\c_T^{(\alpha)}(v) < (d+2)(|u|-1)$.
\end{claim}
\end{definition}

\begin{remark}
Applying the first two conditions with $B = A$ and $u = A$ respectively implies in particular that
$\sum_{v\in \mathring V_T} \c_T^{(\alpha)}(v) = \sum_{a\in A} \alpha_a$.
\end{remark}

We will now display a consistent cumulant homogeneity such that, for every 
finite set $A$, space-time points $\{z_a\}_{a\in A}$, and choices $\alpha\colon A\to [\underline c, {d+2\over 2}]$,
We have the bound
\begin{equ}[e:boundCumScale]
|\kappa_A\bigl(\bigl\{\eta^{(\alpha_a)}_\eps(z_a)\bigr\}_{a\in A}\bigr)|
\lesssim \prod_{u \in \mathring V_T} 2^{\c_T^{(\alpha)}(u) \scale(u)}\;,\qquad (T,\scale) = \tree_z\;.
\end{equ}
We claim that one possible choice is obtained by setting
\begin{equ}[e:defCum]
\c_T^{(\alpha)}(u) = \sum_{a\in A} \alpha_a 2^{- \DD(u,a)}\;,
\end{equ}
where
\begin{equ}[e:defDD]
\DD(u,a) = 
\left\{\begin{array}{cl}
	|\{v \in \mathring V_T\,:\, a \in v \subset u\}| & \text{if $a \in u$ and $u \neq A$,}\\
	|\{v \in \mathring V_T\,:\, a \in v\}| - 1 & \text{if $u=A$,}\\
	+\infty & \text{if $a \not \in u$.} 
\end{array}\right.
\end{equ}

\begin{proposition}\label{prop:cumulant}
The choice \eqref{e:defCum} is a consistent cumulant homogeneity.
\end{proposition}

\begin{proof}
The first two conditions of Definition~\ref{def:consistent} follow immediately from the
structure of the formula \eqref{e:defCum}, in particular the facts 
that $2^{- \DD(u,a)}$ is positive, vanishes for $a \not \in u$, and is such that 
$\sum_{u \in \mathring V_T} 2^{- \DD(u,a)} = 1$.

Regarding the last condition, the case $|u| = 3$ follows from the fact that the 
second condition holds and $\alpha_a \le {d+2\over 2}$. The case $|u|=2$
follows from the condition $|A| \ge 3$ which guarantees that the corresponding sum
is bounded by ${d+2\over 2}$ since $\DD(u,a) = 1$ in this case.
\end{proof}

\begin{lemma}\label{lem:range}
Setting $\underline \alpha = \inf_{a\in A}\alpha_a$ and $\overline \alpha = \sup_{a\in A}\alpha_a$,
one has $\c_T^{(\alpha)}(u) \in [\underline \alpha,\overline \alpha]$
for $u \in \mathring V_T \setminus \{A\}$ and $\c_T^{(\alpha)}(A) \in [2\underline \alpha,2\overline \alpha]$.
\end{lemma}

\begin{proof}
By convexity of \eqref{e:defCum}, it suffices to consider the case where $\alpha_a = 1$ for all $a$.
We proceed by induction on the size of $A$. When $|A| = 2$, one has $\mathring V_T = \{A\}$,
so that $\DD(A,a) = 0$ and therefore $\c_T^{(\alpha)}(A) = 2$ as claimed.

Assume now that $|A| > 2$ and fix a binary tree $T$ over $A$. Write $A_1$ and $A_2$ for the children of 
$A$ in $\mathring V_T$, so that $A = A_1 \sqcup A_2$. 
We distinguish two cases. In the first case, $|A_1| \wedge |A_2| \ge 2$, so that the tree $T$
can naturally be thought of as two trees $T_1$, $T_2$ over $A_1$, $A_2$ joined by their roots.
By \eqref{e:defDD} (in particular the fact that there is an additional $-1$ at the root), we then have
\begin{equ}
\c_T^{(\alpha)}(A_1) = {1\over 2}\c_{T_1}^{(\alpha)}(A_1)\;,\quad
\c_T^{(\alpha)}(A_2) = {1\over 2}\c_{T_2}^{(\alpha)}(A_2)\;,\quad
\c_T^{(\alpha)}(A) = {1\over 2}\bigl(\c_{T_1}^{(\alpha)}(A_1) + \c_{T_2}^{(\alpha)}(A_2)\bigr)\;,
\end{equ}
while $\c_T^{(\alpha)}(u) = \c_{T_i}^{(\alpha)}(u)$ for all other $u \in \mathring V_T$, with
$i \in \{1,2\}$ depending on whether $u \subset A_1$ or $u \subset A_2$.
We conclude by using the induction hypothesis, which implies that $\c_{T_1}^{(\alpha)}(A_1) = \c_{T_2}^{(\alpha)}(A_2) = 2$.

In the second case, we have $|A_1| = 1$ and $|A_2| \ge 2$ (or vice-versa), the case $|A|=2$ having already
been dealt with. In this case, the tree $T$ consists of a subtree $T_2$ over $A_2$ as before, 
with an additional root vertex $A$ and single extra leaf. In this case, we have
\begin{equ}
\c_T^{(\alpha)}(A_2) = {1\over 2}\c_{T_2}^{(\alpha)}(A_2)\;,\quad
\c_T^{(\alpha)}(A) = 1 + {1\over 2}\c_{T_2}^{(\alpha)}(A_2)\;,
\end{equ}
whence we conclude as before.
\end{proof}

\begin{proposition}\label{prop:conditionAjay}
Under Assumption~\ref{ass:kappa}, the bound \eqref{e:boundCumScale} holds for the choice \eqref{e:defCum}.
\end{proposition}

\begin{proof}
It follows from Assumption~\ref{ass:kappa} and \eqref{e:defeta} that 
\begin{equ}[e:bbasic]
|\kappa_A\bigl(\bigl\{\eta^{(\alpha_a)}_\eps(z_a)\bigr\}_{a\in A}\bigr)|
\lesssim \eps^{c_\eps(\scale(A)) - \sum_{a\in A} \alpha_a} 2^{c_\eps(\scale(A)) \scale(A)} \prod_{u \in \mathring V_T} \eps^{c_\eps(\scale(u))} 2^{c_\eps(\scale(u)) \scale(u)}\;,
\end{equ}
where 
\begin{equ}[e:boundceps]
c_\eps(n) = 
\left\{\begin{array}{cl}
	\underline c & \text{if $n \ge \log_2{1\over\eps}$,} \\
	\overline c & \text{otherwise.}
\end{array}\right.
\end{equ}
Let now $\hat c\colon \mathring V_T \to [\underline c, \overline c]$ be any map such that 
$\hat c(A) + \sum_{u \in \mathring V_T} \hat c(u) = \sum_{a\in A} \alpha_a$, and rewrite \eqref{e:bbasic}
as
\begin{equ}[e:better]
|\kappa_A\bigl(\bigl\{\eta^{(\alpha_a)}_\eps(z_a)\bigr\}_{a\in A}\bigr)|
\lesssim \eps^{c_\eps(\scale(A)) - \hat c(A)} 2^{c_\eps(\scale(A)) \scale(A)} \prod_{u \in \mathring V_T} \eps^{c_\eps(\scale(u))- \hat c(u)} 2^{c_\eps(\scale(u)) \scale(u)}\;.
\end{equ}
We now note that \eqref{e:boundceps} implies that 
\begin{equ}
\eps^{c_\eps(n)- \hat c} 2^{c_\eps(n) n} = (\eps 2^n)^{c_\eps(n)- \hat c} 2^{\hat c n} \le 2^{\hat c n}\;, 
\end{equ}
for every $\hat c \in [\underline c, \overline c]$, uniformly over $n \in \Z$ and $\eps \in (0,1]$.
Inserting this into \eqref{e:better} immediately yields that, uniformly in $\eps$, one has
\begin{equ}[e:good]
|\kappa_A\bigl(\bigl\{\eta^{(\alpha_a)}_\eps(z_a)\bigr\}_{a\in A}\bigr)|
\lesssim   2^{\hat c(A) \scale(A)} \prod_{u \in \mathring V_T} 2^{\hat c(u) \scale(u)}\;.
\end{equ}
Since the map $\c_T^{(\alpha)}$ is of the desired type by Lemma~\ref{lem:range} (modulo the additional
factor $2$ at the root which is taken care of explicitly in \eqref{e:good}), the claim follows.
\end{proof}

\begin{corollary}\label{cor:tight}
For any $\alpha \in \big(\underline c,{d+2\over 2}\big)$, one has $\eta_\eps^{(\alpha)} \to 0$ in probability
in $\CC^\beta$ for every $\beta < -\alpha$.
In particular, $\eta_\eps \to 0$ in probability in $\CC^\beta$ for every $\beta < -1$
and $\zeta_\eps \to 0$ in probability in $\CC^\beta$ for every $\beta < {d-2\over 2}$.
The same holds for $\eta_\eps^{(\alpha)} \one_A$ for any fixed Borel set $A$.
\end{corollary}

\begin{proof}
The first statement is an immediate consequence of \cite[Thm~2.31]{Ajay}.
The fact that we can multiply $\eta_\eps^{(\alpha)}$ by an arbitrary indicator function 
follows from the fact that these bounds do not involve the derivative of the test function
in this case.
\end{proof}

\subsection{Power-counting conditions}

By \eqref{e:defCum}--\eqref{e:defDD}, the quantity $\mcb{h}_{\c,D}(A)$ defined in \cite[Def.~A.24]{Ajay} can be
estimated by 
\begin{equ}[e:boundh]
\mcb{h}_{\c,D}(A) = \sum_{a \in D} 2^{-\DD(D,a)} \alpha_a \ge 2^{-(|D|-1)}\sum_{a \in D} \alpha_a\;.
\end{equ}
This is because for $D \in \mathring V_T$ with $D \neq A$, one always has
$\DD(D,a) \le |D|-1$. Furthermore, by \cite[Rem.~2.28]{Ajay}, one has $\mcb{j}_D(A) \ge {d+2\over 2} -\kappa$
for all typed sets $A$ and $D$.
These observations allow us to conclude the following.

\begin{lemma}\label{lem:regular}
All the decorated trees generated by the rule $R_{\invariant}$ of Definition~\ref{def:ruleInvariant} 
are $\c$-super regular in the terminology of \cite[Def.~A.27]{Ajay}.
\end{lemma}

\begin{remark}
Since our scaling and degree assignment are fixed throughout and since we consider all cumulants, 
i.e.\ we choose $\mfL_{\CCum}$ to contain all possible cumulants, we omit these from our notation.   
\end{remark}

\begin{proof}
We reduce ourselves to considering symbols of negative degree since the claim then follows
automatically for the remaining ones. These symbols are listed in \eqref{e:listBasis}.
Note also that the definition of $\c$-super regularity is non-trivial only for trees that 
contain at least three noises, so that it suffices to consider the symbols
\begin{equ}[e:simpleList]
\<Xi1IXi1IXi>, \<Xi1IXi2IXi^2>, \<Xi2IXi^2>, \<Xi2IXiIXi1IXi>, \<Xi2IXiIXiX>, \<Xi2IXiIXi11>, \<Xi21IXi1IXi>, \<IXiIXi1IXi>, \<Xi1IIXi^2>\;.
\end{equ}
By \eqref{e:boundh}, it is sufficient to verify that for every subtree $\tau$ containing 
$k\ge 2$ instances of a noise one has the bound
\begin{equ}
\deg \tau + \Bigl({d+2\over 2} \wedge 2^{-(k-1)}\sum_{i=1}^k |\deg \Xi_i|\Bigr) - \kappa \ge 0\;,
\end{equ}
where $\Xi_i$ denotes the $i$th noise appearing in $\tau$. This can be seen simply by inspection
of the list \eqref{e:simpleList}.
\end{proof}

\subsection{Special bounds}

For some of the symbols in our regularity structure, we will bounds that are 
stronger than what is suggested by the degrees of the symbols in question. 
In this statement, $\phi$ denotes an arbitrary measurable function with 
\begin{equ}[e:constraintphi]
\sup_z|\phi(z)| \le 1\;,\qquad
\supp \phi \subset B(0,1)\;.
\end{equ}

\begin{proposition}\label{prop:boundsUseful}
Let $\delta>0$ be as in Section~\ref{sec:ass}.
For $\alpha \in [-\delta,1]$ and $\beta = -\f12-\alpha$, we have the bounds
\begin{equ}
\big(\hat \PPi_\eps \<Xi2IXi^2>\bigr)(\phi^\lambda_z)
\lesssim \eps^{\alpha} \lambda^{\beta}\;,
\qquad
\big(\hat \PPi_\eps \<Xi2IXi>\bigr)(\phi^\lambda_z)
\lesssim \eps^{\f12 + \alpha} \lambda^{\beta}\;,
\qquad
\big(\hat \PPi_\eps \<Xi2>\bigr)(\phi^\lambda_z)
\lesssim \eps^{1+\alpha} \lambda^{\beta}\;,
\end{equ}
where we write $X \lesssim Y$ as a shorthand for the existence, for every $p \ge 1$,
of a constant $C$  such that $\E |X|^p \le C Y^p$, uniformly over all 
$\lambda \le 1$, $\eps \le 1$ and $\phi$ as in \eqref{e:constraintphi}.
\end{proposition}

\begin{proof}
Recall that, by Proposition~\ref{prop:conditionAjay}, we can view any $\eta_\eps^{(\theta)}$ 
as in \eqref{e:defeta} for $\theta \in [{1\over 2}-\delta, {5\over 2}]$ as a ``noise'' of 
regularity $-\theta$ whose ``norm'', as measured by \cite[Def.~A.18]{Ajay} with respect to the cumulant homogeneity
just described remains uniformly bounded as $\eps \to 0$.

In particular, we can write
\begin{equ}
\PPi_\eps \<Xi2IXi^2> = \eps^\alpha (K*\eta_\eps^{(2)})^2 \eta_\eps^{({1\over 2}+\alpha)}\;,
\end{equ} 
and we can apply \cite[Thm~2.31]{Ajay}, showing that the BPHZ renormalisation of this term 
satisfies the required bounds. Recall that $\hat \PPi_\eps \<Xi2IXi^2>$ doesn't quite agree with
the BPHZ renormalisation, but the error between the two is given by 
$2\eps^{3/2} \delta g_\eps(\<Xi1IXi>) \PPi_\eps \<IXi>$ with $\delta g_\eps$ as in \eqref{e:deltag1},
which is easily seen to satisfy the required bounds.

The other two terms can be dealt with similarly by writing
\begin{equ}
\PPi_\eps \<Xi2IXi> = \eps^{1\over 2+\alpha} (K*\eta_\eps^{(2)}) \eta_\eps^{({1\over 2}+\alpha)}\;,\qquad
\PPi_\eps \<Xi2> = \eps^{1+\alpha} \eta_\eps^{({1\over 2}+\alpha)}\;,
\end{equ}
thus concluding the proof.
\end{proof}

\begin{corollary}\label{prop:boundXieps}
For every $\kappa > 0$, every $\alpha \in [0,1)$, and every Borel set $A \subset \R^{d+1}$, 
one has the bound $\big(\E \|\one_A \xi_\eps\|_{\CC^w}^p\big)^{1/p} \lesssim \eps^{-\alpha}$
with $w = \alpha - {5\over 2} - \kappa$.
\end{corollary}

\begin{proof}
This follows from the third bound of Proposition~\ref{prop:boundsUseful} by Kolmogorov's criterion,
using the fact that $\hat \PPi_\eps \<Xi2> = \eps^3 \xi_\eps$ and that we can move the multiplication
by $\one_A$ onto the test function.
\end{proof}

\subsection{Tightness and convergence for the noise}

In this section, we show that the convergence announced in Theorem~\ref{theo:model}
holds. As usual, convergence is obtained by
first showing tightness and then identification of the limiting distribution.
More precisely, we prove the following.

\begin{theorem}\label{theo:main1}
One has $\xi_\eps \to \xi$ weakly in $\CC^\alpha$ for every $\alpha < -{d+2\over 2}$.
\end{theorem}

It turns out that this statement is a relatively straightforward consequence of the
following proposition.
Writing 
$\kappa_p(X)$ for the $p$th cumulant of a real-valued random variable $X$, one has the following.

\begin{proposition}\label{prop:convGauss}
For $p \ge 2$, we have the bound
\begin{equ}
|\kappa_{p}\big(\scal{\eta,\phi_0^\lambda}\big)| \lesssim \big(\lambda^{-\overline c p} + \lambda^{-(d+2)(p-1)}\big) \wedge \lambda^{-\underline c p}\;,
\end{equ}
uniformly over all $\lambda \in \R_+$ and all $\phi$ as in \eqref{e:constraintphi}.
\end{proposition}

Before we turn to the proof of Proposition~\ref{prop:convGauss}, let us show how to deduce
Theorem~\ref{theo:main1} from it. First, we have the following corollary.

\begin{corollary}\label{cor:tightSimple}
For any $\kappa \le {1\over 2}$ and $p \ge 2$, one has the bounds
\begin{equ}
\E \bigl|\scal{\xi_\eps,\phi_0^\lambda}\bigr|^{p}
\lesssim \lambda^{-{d+2\over 2}p}\;,
\end{equ}
uniformly over $\lambda \le 1$, $\eps \le 1$ and $\phi$ as in \eqref{e:constraintphi}.
\end{corollary}

\begin{proof}
By simple rescaling, it is straightforward to see that the required bound is
equivalent to the bound
\begin{equ}[e:rescaled]
\E \bigl|\scal{\eta,\phi_0^{\bar \lambda}}\bigr|^{p}
\lesssim {\bar\lambda}^{-{d+2\over 2}p}\;,
\end{equ}
uniformly over $\bar \lambda \le 1/\eps$ and $\eps \le 1$.

For even integers $p \ge 2$, Proposition~\ref{prop:convGauss} implies that 
$|\kappa_{p}\big(\scal{\eta,\phi_0^\lambda}\big)| \lesssim \lambda^{-{d+2\over 2} p} \wedge \lambda^{-\underline c p}$, so that.
\begin{equ}[e:goodbound]
\E |\scal{\eta,\phi_0^\lambda}|^{p} \lesssim \lambda^{-{d+2\over 2} p} \wedge \lambda^{-\underline c p}\;,
\end{equ}
It remains to observe that \eqref{e:goodbound} implies that
$\E |\scal{\eta,\phi_0^{\bar \lambda}}|^{p} \lesssim {\bar \lambda}^{-c p}$,
uniformly over all $\bar \lambda > 0$, for any $c \in [\underline c, {d+2\over 2}]$.
Since \eqref{e:rescaled} is of this form and our assumptions guarantee that the values of $c$ appearing
there fall into the correct interval, this concludes the proof.
\end{proof}

\begin{proof}[of Theorem~\ref{theo:main1}]
Using \cite[Eq.~10.4]{regularity} (which is nothing but an analogue of Kolmogorov's continuity criterion) and the compactness of the
embeddings $\CC^\alpha \subset \CC^\beta$ for $\alpha > \beta$ (over any bounded domain), 
it follows from Corollary~\ref{cor:tightSimple} that the laws $\xi_\eps$ are tight in 
$\CC^\alpha$ for every $\alpha < -{d+2\over 2}$.

It remains to show that every limit $\xi$ of a convergent subsequence of $\xi_\eps$ is 
space-time white noise. It thus suffices to show that, for every $\phi \in \CC_0^\infty$ with
$\int \phi^2(z)\,dz = 1$, $\xi(\phi)$ is centred normal with variance $1$. 
Since all moments of $\scal{\xi_\eps,\phi}$ remain bounded as $\eps \to 0$ by Corollary~\ref{cor:tightSimple}, 
one has
\begin{equs}
\kappa_p \bigl(\xi(\phi)\bigr)
&= \lim_{\eps \to 0} \kappa_p \bigl(\scal{\xi_\eps,\phi}\bigr)
= \lim_{\eps \to 0} \eps^{-{(d+2)p\over 2}} \kappa_p \bigl(\scal{\eta,\phi_0^{1/\eps}}\bigr) \\
&\lesssim \lim_{\eps \to 0} \eps^{-{(d+2)p\over 2}} \bigl(\eps^{\overline c p} + \eps^{(d+2)( p-1)}\bigr)\;.
\end{equs}
For $p \ge 3$, this vanishes, thus showing that $\xi(\phi)$ is Gaussian.
It clearly has zero mean since this is already the case for $\scal{\xi_\eps,\phi}$.
Its variance is given by
\begin{equ}
\E |\scal{\xi_\eps,\phi}|^2 = \eps^{-(d+2)}\int \kappa_2(z,\bar z)\phi_0^{1/\eps}(z)\phi_0^{1/\eps}(\bar z)\,dz\,d\bar z\;.
\end{equ}
Note now that, for every $\eps > 0$, every $\delta \in [0,1]$ and any $z, \bar z \in \R^{d+1}$, one has the bound
\begin{equ}
|\phi_0^{1/\eps}(z) - \phi_0^{1/\eps}(\bar z)|
\lesssim \eps^{d+2} \bigl(\eps\|z-\bar z\|\bigr)^{\delta}\;.
\end{equ}
Recalling that $\int \phi^2 =1$ and that $\int \kappa_2(0,z)\,dz = 1$ by \eqref{e:normalisation}, we thus obtain
\begin{equs}
\bigl|\E |\scal{\xi_\eps,\phi}|^2 - 1\bigr|
&= \Bigl| \eps^{-(d+2)}\int \kappa_2(z,\bar z)\phi_0^{1/\eps}(z)\bigl(\phi_0^{1/\eps}(\bar z)- \phi_0^{1/\eps}(z)\bigr)\,dz\,d\bar z \Bigr|\\
&\lesssim \eps^{\delta}  \int \Bigl|\kappa_2(z,\bar z)\phi_0^{1/\eps}(z)\Bigr|\, \|z-\bar z\|^{\delta} \,dz\,d\bar z \\
&\lesssim \eps^{\delta}  \int \bigl(\|z'\|^{-2\overline c} \wedge \|z'\|^{-2\underline c} \bigr)\|z'\|^{\delta} \,dz'  \lesssim \eps^{\delta}\;.
\end{equs}
Here, we used the fact that the integral of $\phi_0^{1/\eps}$ is independent of $\eps$,
as well as our assumptions \eqref{e:normalisation} and  \eqref{e:boundkappa} on the covariance function of $\eta$. In particular,
we choose $\delta$ as in the definition of $\overline c$, see Section~\ref{sec:ass}, 
whence $2\overline c - \delta > d+2$ which guarantees integrability at infinity. (Integrability at $0$ is guaranteed
by $2\underline c - \delta = d - \delta < d+2$.)
\end{proof} 

\begin{proof}[of Proposition~\ref{prop:convGauss}]
We expand the expression for the cumulant as
\begin{equ}[e:exprcumul]
\kappa_{p}\big(\scal{\eta,\phi_0^\lambda}\big) = \int \kappa_p(z_1,\ldots,z_p) \,\phi_0^\lambda(z_1)\ldots \phi_0^\lambda(z_p)\,dz_1\ldots dz_p \;.
\end{equ}
In order to bound this integral, we use a simplified version of the 
type of multiscale analysis used in \cite{HS,Jeremy}. Let us recall how this works.

We now write $\tree = (T,\scale)$ for a generic binary tree, together with a scale assignment
as above (i.e.\ we enforce the fact that $\scale$ is monotone) and $\trees$ for the set of all such $\tree$.
Given $\tree \in \trees$, we write $V_\tree$ / $E_\tree$ for the vertex / edge set of the corresponding tree and
$\scale_\tree$ for the scale assignment. In particular, $\scale_\tree(\Omega_p)$ denotes the scale assignment
for the root, which controls the diameter of the set $\{z_1,\ldots,z_p\}$ in $\R^{d+1}$.
The number $p$ will always be considered fixed, so we do not include it explicitly in our notation.
Denoting by $\CT$ the map $\CT \colon (z_1,\ldots,z_p) \mapsto \tree \in \trees$ defined at the start of
Section~\ref{sec:ass},
we set $D_\tree = \CT^{-1}(\tree)$ for the set of all configurations of points $z \in (\R^{d+1})^p$ giving rise to
a given combinatorial data. Then, 
it was shown in \cite[Lem.~A.13]{Jeremy} that, for every bounded Borel set $U \in \R^{d+1}$ one has the bound
\begin{equ}
|D_\tree \cap \{z\,:\, z_1 \in U\}| \lesssim |U|\prod_{A \in \mathring V_\tree} 2^{-(d+2) \scale_\tree(A)}\;,
\end{equ}
where $|\cdot|$ denotes Lebesgue measure. Furthermore, by construction, $\bigcup_{\tree \in \trees} D_\tree$ is
of full Lebesgue measure.

Without loss of generality, we can restrict ourselves to the case
where the support of $\phi$ has diameter bounded by $1$ in the parabolic distance.
With all of this notation at hand, we then bound \eqref{e:exprcumul} by
\begin{equ}
|\kappa_{p}\big(\scal{\eta,\phi_0^\lambda}\big)|
\lesssim \lambda^{-(d+2)p} \sum_{\tree \in \trees} |D_\tree \cap \{z\,:\, z_i \in \supp \phi_0^\lambda \;\forall i\}|\, \sup_{z \in D_\tree} |\kappa_p(z_1,\ldots,z_p)|\;.
\end{equ}
We simplify this expression as follows. First, we note that 
\begin{equ}
D_\tree \cap \{z\,:\, z_i \in \supp \phi_0^\lambda \;\forall i\} = \emptyset
\end{equ}
as soon as $2^{- \scale(\Omega_p)} \ge \lambda$, since the support of $\phi_0^\lambda$ is bounded
by $\lambda$, so that we can restrict the sum above to those $\tree$ satisfying $2^{- \scale(\Omega_p)} \le \lambda$.
Furthermore, one has 
\begin{equ}
|D_\tree \cap \{z\,:\, z_i \in \supp \phi_0^\lambda \;\forall i\}|
\le 
|D_\tree \cap \{z\,:\, z_1 \in \supp \phi_0^\lambda\}|
\lesssim \lambda^{d+2} \prod_{A \in \mathring V_\tree} 2^{-(d+2) \scale_\tree(A)}\;.
\end{equ}
Combining this with Assumption~\ref{ass:kappa}, we conclude that
\begin{equ}
|\kappa_{p}\big(\scal{\eta,\phi_0^\lambda}\big)|
\lesssim \lambda^{(d+2)(1-p)} \sum_{\tree \in \trees} \one_{2^{- \scale(\Omega_p)} \le \lambda}\,
2^{c(\Omega_p)\, \scale_\tree(\Omega_p)} \prod_{A \in \mathring V_\tree} 2^{(c(A)-(d+2)) \scale_\tree(A)}\;.
\end{equ}

We treat separately the cases $\lambda \le 1$ and $\lambda \ge 1$.
In the former case, we can apply \cite[Lem.~A.10]{Jeremy} with the distinguished vertex $\nu_\star$ appearing there
equal to the root $\Omega_p$. The first condition appearing there is then satisfied by the fact that 
$2\underline c - (d+2) < 0$ by assumption, while the second condition is empty and therefore trivially satisfied.
Note that the shape $T$ of the tree is fixed in \cite[Lem.~A.10]{Jeremy}, while we also sum over all possible shapes,
but since there are finitely many of them for any fixed $p$, this just yields an additional prefactor.
We thus obtain the bound
\begin{equs}
|\kappa_{p}\big(\scal{\eta,\phi_0^\lambda}\big)|
&\lesssim \lambda^{(d+2)(1-p)} \sum_{\tree \in \trees} \one_{2^{- \scale(\Omega_p)} \le \lambda}\,
2^{\underline c\, \scale_\tree(\Omega_p)} \prod_{A \in \mathring V_\tree} 2^{(\underline c-(d+2)) \scale_\tree(A)}\\
&\lesssim \lambda^{(d+2)(1-p) - \underline c}  \prod_{A \in \mathring V_\tree} \lambda^{d+2-\underline c}\;.
\end{equs}
Since $\mathring V_\tree$ contains exactly $p-1$ elements (the tree $T$ is binary and has $p$ leaves), 
we finally obtain
\begin{equ}
|\kappa_{p}\big(\scal{\eta,\phi_0^\lambda}\big)| \lesssim \lambda^{-p\underline c}\;,
\end{equ}
as required.

The case $\lambda > 1$ is split into two subcases. If $p \ge 3$, 
we use the fact that $|\kappa_p(z)| \lesssim \bar \kappa_{p,\overline c}(z)$ uniformly over all $z$,
so that, using \cite[Lem.~A.10]{Jeremy} as above, we have
\begin{equ}
|\kappa_{p}\big(\scal{\eta,\phi_0^\lambda}\big)|
\lesssim \lambda^{(d+2)(1-p)} \sum_{\tree \in \trees} \one_{2^{- \scale(\Omega_p)} \le \lambda}\,
2^{\overline c\, \scale_\tree(\Omega_p)} \prod_{A \in \mathring V_\tree} 2^{(\overline c-(d+2)) \scale_\tree(A)}
\lesssim \lambda^{-p\overline c}\;.
\end{equ}
Note that, in order to be able to apply that result, we need to verify that, for every subtree
$\hat \tree$ of $\tree$ spanned by some subset of its leaves satisfies 
$\sum_{A \in \mathring V_{\hat \tree}} \alpha(A) < 0$, where $\alpha(A) = 2\overline c-(d+2)$ if 
$A = \Omega_p$ and $\alpha(A) = \overline c-(d+2)$ otherwise.
This is of course trivially satisfied as soon as $\hat \tree \neq \tree$ since $\overline c < d+2$ by assumption. 
The exponent at the root however is given by $2\overline c-(d+2)$, which
is positive, but since $p \ge 3$, the sum of all exponents is given by
$2\overline c-(d+2) + (p-2)(\overline c - (d+2)) = p\overline c - (d+2)(p-1)$, which is indeed negative for $p \ge 3$ and $\delta < {d+2\over6}$,which we assume w.l.o.g..

It remains to consider the case $p = 2$ and $\lambda > 1$. In this case, the above computation reduces to
\begin{equ}
|\kappa_{2}\big(\scal{\eta,\phi_0^\lambda}\big)|
\lesssim \lambda^{-(d+2)}  \Bigl(\sum_{1 \le 2^{-n} \le \lambda} 2^{(2\overline c - (d+2))n} + \sum_{2^{-n} < 1} 2^{(2\underline c - (d+2))n}\Bigr)
\lesssim \lambda^{-(d+2)}\;,
\end{equ}
as required, thus concluding the proof.
\end{proof}

We also use the following convergence results which do not strictly speaking follow
from Theorem~\ref{theo:main1} since multiplication by an indicator function (even that of a hypercube)
is not a continuous operation on $\CC^{\beta}$ for $\beta < 0$.

\begin{lemma}\label{lem:convergeRestrict}
Let $\xi_{\eps,D} = \xi_\eps \one_{\R \times D}$ and $\xi_{\eps,D}^+ = \xi_\eps \one_{\R_+ \times D}$.
Then, $\xi_{\eps}$, $\xi_{\eps,D}$,  and $\xi_{\eps,D}^+$ jointly weakly converge to limits
$\xi$, $\xi_D$ and $\xi_D^+$.
\end{lemma}

\begin{proof}
The proof is identical to that of Theorem~\ref{theo:main1}, using the fact that 
all bounds we used are uniform over test functions as in \eqref{e:constraintphi}, so that
multiplying them by the indicator function of some domain changes nothing.
\end{proof}

\subsection{Boundary term}

Recall that we have set
\begin{equ}[e:defhatPi]
\hat \PPi_\eps\<Xi1IXib> = \PPi_\eps\<Xi1IXib> - \sum_{i=1}^d \eps^{1-{d\over 2}} \bigl(c_{i,0} \delta_{\d_{i,0} D}
+c_{i,1} \delta_{\d_{i,1} D}\big)\;,
\end{equ}
where $\d_{i,0} D = \{x \in D \,:\, x_i = 0\}$ and $\d_{i,1} D = \{x \in D \,:\, x_i = 1\}$ and
the constants are given by 
\begin{equ}[e:intQ]
c_{i,j} = \int_{\R_+} Q_{i,j}(s)\,ds\;,
\end{equ}
where the function $Q_{i,j}$ is defined as follows. For $i=1,\ldots,d$, write $\iota_i\colon \R\times \R^{d} \to \R^{d+1}$
for the map given by 
\begin{equ}
\iota_i^{-1}(t,x) = (x_i,v)\;,\quad v = (t,x^{(i)})\;,
\end{equ}
where $x^{(i)} \in \R^{d-1}$ denotes the vector obtained from $x$ by deleting the $i$th coordinate.
With this notation, we then set
\begin{equs}
Q_{i,0}(s)&=\int_{\R_+\times \R^d}(P\circ \iota_i)\left(s+\beta,v\right)\left[(\kappa_2\circ \iota_i)\left(s-\beta,v\right)-
(\kappa_2\circ \iota_i)\left(s+\beta,v\right)\right]d\beta dv\;,\\
Q_{i,1}(s)&=\int_{\R_+\times \R^d}(P\circ \iota_i)\left(s+\beta,v\right)\left[(\kappa_2\circ \iota_i)\left(\beta-s,v\right)-
(\kappa_2\circ \iota_i)\left(-s-\beta,v\right)\right]d\beta dv\;.
\end{equs}

\begin{remark}
We will show in Lemma~\ref{lem:boundQ} below that both the integrands in the definition of $Q_{i,j}$ and the functions $Q_{i,j}$ 
themselves are integrable, so that these expressions are all finite. 
\end{remark}

\begin{remark}
The formula given above is valid for the case of Neumann boundary conditions. In the case
of Dirichlet boundary conditions, a similar formula holds, but the precise values of the
constants do not matter since they do not affect the solutions.
\end{remark}

Note that although $P\circ \iota_i\neq P$, it does not depend on $i$, while $\kappa_2\circ \iota_i$ does depend
on $i$ in general, since we do not assume that the driving noise is isotropic.
We henceforth write $\hat P = P\circ \iota_i$. One of the main results of this section is that 
the renormalised model on $\<Xi1IXib>$ vanishes in a suitable sense as $\eps \to 0$.
We first provide a bound on its expectation, which requires the bulk of the work.
For the formulation of this result, we write $\CB_0^1$ for the set of all test functions
$\phi \in \CC^1_0$  with support contained in the parabolic ball of radius $1$ and 
such that $\max\{\|\phi\|_\infty, \|D\phi\|_\infty\} \le 1$.

\begin{proposition}\label{prop:correctionTerm}
With $\hat \PPi_\eps\<Xi1IXib>$ defined as in \eqref{e:defhatPi}, one has for $d \in \{2,3\}$ and any $\kappa > 0$ small enough,
\begin{equ}[e:actualBound]
\big|\E \bigl(\hat \PPi_\eps\<Xi1IXib>\bigr)(\phi_z^\lambda)\big| \lesssim \eps^{\kappa} \lambda^{-{d\over 2}-\kappa}\;,
\end{equ}
uniformly over $\eps, \lambda \le 1$, $z \in \R^{d+1}$, and $\phi \in \CB_0^1$.
\end{proposition}

Before we turn to the proof, we introduce a number of notations and preliminary bounds.
Write $\CG$ for 
the reflection group generated by ``elementary'' reflections across the $2d$ planes containing the faces of $D$.
The group $\CG$ is naturally identified with $\Z^d$ (as a set, not as a group!) since for each 
$k \in \Z^d$ there exists exactly one
element $R_k \in \CG$ mapping $k+D$ into $D$. 
We also write $\CG \ni R \mapsto (-1)^R \in \{-1,1\}$
for the group homomorphism mapping the elementary reflections to $-1$.
We will write $\Lambda \colon \R^{d+1} \to \R \times D$
for the map such that $\Lambda\restr \R \times (k+D) = \id \times \CR_k$
and $S \colon \R^{d+1} \to \{-1,1\}$ by $S \restr \R \times (k+D) = (-1)^{R_k}$.

With these notations, it follows from Proposition~\ref{prop:NeumannBC} that the truncated Neumann and 
Dirichlet heat kernels are such that for $z \in \R\times D$ one
has the identities 
\begin{equs}[e:reprHeat]
\int_{\R\times D}  K_\Neu(z, z')f(z')\,dz' &= \int_{\R^{d+1}} K(z-z') f(\Lambda(z'))\,dz'\;,\\
\int_{\R\times D}  K_\Dir(z, z')f(z')\,dz' &= \int_{\R^{d+1}}  S(z') K(z-z') f(\Lambda(z'))\,dz'\;.
\end{equs}
(Note that $S$ is ill-defined on the measure zero set consisting of the reflection planes, but since it always
appears in an integral this does not matter.)

\begin{lemma}\label{lem:boundSimple}
Let $U \subset \R^{d+1}$ and let $\Phi \colon U \to \R^{d+1}$ be a diffeomorphism between $U$
and its image such that $\|D\Phi^{-1}\|$ is bounded uniformly over $\Phi(U)$.
Assume furthermore that $z$ is such that $\|z-z'\| \ge \lambda$ for all $z' \in U$.
Then,  
\begin{equ}
\int_{U} K(z-z')\,|\kappa^{(\eps)}_2(z,\Phi(z'))| dz'
\lesssim 
 \eps^{{d\over 2}+\kappa} \lambda^{-{d\over 2}-\kappa}\;.
\end{equ}
for all $\kappa \in [0,{d\over 2}]$.
\end{lemma}

\begin{proof}
We write $U = U_1 \sqcup U_2$ where
\begin{equ}
U_1 = \{z'\in U\,:\, \|z-z'\| \ge \|z-\Phi(z')\|\}\;.
\end{equ}
Since $\|z-z'\| \ge \lambda$ for $z' \in U$ by assumption,
we have the bound $K(z-z') \lesssim \lambda^{-d}$. On $U_1$, it follows from the definition that 
$K(z-z') \lesssim  \|z- \Phi(z')\|^{-d}$.
As a consequence, we obtain the bound
\begin{equs}
\int_{U_1} &K(z-z')\,|\kappa^{(\eps)}_2(z,\Phi(z'))| dz'
\lesssim 
 \int_{\Phi(U_1)} \bigl(\|z-z'\|^{-d} \wedge \lambda^{-d} \bigr) |\kappa^{(\eps)}_2(z,z')|\,dz' \\
&\lesssim \int_{\R^{d+1}} \bigl(\|z'\|^{-d} \wedge (\lambda/\eps)^{-d} \bigr) \rho^2(\|z'\|)\,dz' \lesssim 1 \wedge \eps^d \lambda^{-d} \\
&\le \eps^{{d\over 2}+\kappa} \lambda^{-{d\over 2}-\kappa}\;,
\end{equs}
as claimed. On $U_2$ on the other hand, we use the fact that $\rho$ is a decreasing function, so that 
\begin{equs}
\int_{U_2} K(z-z')&\,|\kappa^{(\eps)}_2(z,\Phi(z'))| dz' \\
&\lesssim \eps^{-2} \int_{U_2} \bigl(\|z-z'\|^{-d} \wedge \lambda^{-d} \bigr) \rho^2(\|S_\eps(z-\Phi(z'))\|)\,dz'\\
&\lesssim \eps^{-2} \int_{\R^{d+1}} \bigl(\|z-z'\|^{-d} \wedge \lambda^{-d} \bigr) \rho^2(\|S_\eps(z-z')\|)\,dz'\;,
\end{equs}
which is then bounded exactly as above.
\end{proof}

\begin{lemma}\label{lem:boundLargeDistance}
For any fixed $c > 0$, one has
\begin{equ}
\int_{\|z\| \ge c} |\kappa^{(\eps)}_2(0,z)|\, dz \lesssim  \eps^{d+2\delta}\;,
\end{equ}
uniformly over $\eps \in (0,1]$.
\end{lemma}

\begin{proof}
We can assume that $\eps < c$, so that 
\begin{equ}
|\kappa^{(\eps)}_2(0,z)| \lesssim \eps^{d+2+2\delta} \|z\|^{-d-2-2\delta}
\end{equ}
by Assumption~\ref{ass:kappa}, and the bound follows at once.
\end{proof}

\begin{proof}[of Proposition~\ref{prop:correctionTerm}]
We now consider the Neumann case, the Dirichlet case follows from a virtually identical calculation.
We start by bounding the expectation of $\bigl(\hat \PPi_\eps\<Xi1IXib>\bigr)(\phi_z^\lambda)$.
By the reflection principle \eqref{e:reprHeat}, the correction due to the Neumann boundary conditions is given by
\begin{equs}
\int_{\R\times D} K_{\Neu}&(z,z')\xi_\eps(z')dz'-\int_{\R^{d+1}}K(z-z')\xi_\eps(z')dz'\\ 
&= \int_{\R^{d+1}} K(z-z')\bigl(\xi_\eps(\Lambda z') - \xi_\eps(z')\bigr) dz'\;.
\end{equs}
It then follows from the definitions that 
\begin{equ}[e:expectPPi]
\E \bigl(\hat\PPi_\eps\<Xi1IXib>\bigr)(\psi)
= \eps^{-{d\over 2}}\int_{\R \times D} \psi(z) \int_{\R^{d+1}} K(z-z')\bigl(\kappa^{(\eps)}_2(z,\Lambda z') - \kappa^{(\eps)}_2(z,z')\bigr) dz'\,dz\;.
\end{equ}
Writing $\lambda = \lambda_\psi$ for the diameter of $\supp \psi$, we aim to give a bound of the form
\begin{equ}[e:aimExpectation]
\big|\E \bigl(\hat \PPi_\eps\<Xi1IXib>\bigr)(\psi)\big| \lesssim \eps^{\kappa} \lambda^{-{d\over 2}-\kappa} \bigl(\lambda^d\|\psi\|_\infty + \lambda^{d+1}\|D_x\psi\|_{\infty}\bigr)\;,
\end{equ}
for any $\kappa > 0$ small enough, uniformly over $\eps, \lambda \le 1$.
We restrict ourselves to the case
$\lambda_\psi \le {1\over 8}$ and we set 
\begin{equ}
W_\psi = \{z\,:\, \exists z' \in \supp \psi \quad \|z-z'\| \le \lambda_\psi\}\;.
\end{equ}
We furthermore assume for the moment that 
\begin{equ}[e:propW]
W_\psi \subset \R \times [-3/4,3/4]^d
\end{equ}
and that, whenever $W_\psi \cap \d D = \emptyset$, one has $W_\psi \subset D$.
We will see later that it is always possible to reduce oneself to this case.
Finally, for $x \in \R^{d}$, we write $|x|_0$ for the number of vanishing coordinates 
of $x$ and we set $d_\psi = \sup\{|x|_0\,:\, (t,x) \in W_\psi\}$.
We treat the different cases separately.

\smallskip\noindent\textbf{The case $d_\psi = 0$.}
In this case $W_\psi \cap \d D = \emptyset$, so that $\supp \psi \cap \d D = \emptyset$ and in particular 
$\bigl(\hat \PPi_\eps\<Xi1IXib>\bigr)(\psi)
= \bigl(\PPi_\eps\<Xi1IXib>\bigr)(\psi)$. We then have
\begin{equ}[e:expectationInterior]
|\E\bigl(\hat \PPi_\eps\<Xi1IXib>\bigr)(\psi)|
\le \eps^{-{d\over 2}}\int_{\R\times D} |\psi(z)| \int_{\R \times D^c} K(z-z')\bigl(|\kappa^{(\eps)}_2(z,\Lambda z')| + |\kappa^{(\eps)}_2(z,z')|\bigr) dz'\,dz\;.
\end{equ}
We break the inner integral into a finite sum of integrals over $\R \times (k+D)$, since $K$ has compact support and
$z \in \R \times D$.
Since we can restrict $z$ to the support of $\psi$, 
we have $|x-x'| \ge \lambda$ for all $x' \in k+D$ with $k \neq 0$ by the definition of $W_\psi$.
We can then apply Lemma~\ref{lem:boundSimple} with $U = \R\times (k+D)$ and $\Phi = \Lambda$
for the first term, while $\Phi = \id$ for the second term. 
It follows that $|\E\bigl(\hat \PPi_\eps\<Xi1IXib>\bigr)(\psi)|$
is bounded by
\begin{equ}[e:boundIntPsi]
\eps^{\kappa} \lambda^{-{d\over 2}-\kappa} \int |\psi(z)|\,dz\;,
\end{equ}
which is indeed bounded by the right hand side of \eqref{e:aimExpectation}.

\smallskip\noindent\textbf{The case $d_\psi = 1$.}
In this case, the support of $\psi$ is located near one of the faces of $D$ (say $\d_{i,0}D$), but 
its distance to the other faces is at least $\lambda$. Write $\fR$ for the element of $G$ which corresponds to
reflection around the plane containing $\d_{i,0}D$ and $\pi_i \colon \R^{d+1} \to \R^{d+1}$ for the
orthogonal projection onto that plane. We also write $E_i = \{(t,x) \,:\, x_i < 0\}$.

We first note that 
\begin{equ}[e:expectPPiGood]
\E \bigl(\PPi_\eps\<Xi1IXib>\bigr)(\psi) = 
\eps^{-{d\over 2}}\int_{\R \times D} \psi(z) \int_{E_i} K(z-z')\bigl(\kappa^{(\eps)}_2(z,\fR z') - \kappa^{(\eps)}_2(z,z')\bigr) dz'\,dz
+ R\;,
\end{equ}
where $|R|$ is bounded by \eqref{e:boundIntPsi}. Indeed,  the integrands in \eqref{e:expectPPi} and \eqref{e:expectPPiGood}
vanish on $D$ and coincide on $\fR D$. Their integral over the complement of these two regions is then bounded
exactly as before by applying  Lemma~\ref{lem:boundSimple} to finitely many translates of $D$.
By Lemma~\ref{lem:boundLargeDistance}, we can furthermore replace $K$ by $P$ so that 
\begin{equ}[e:expectPPiGood2]
\E \bigl(\PPi_\eps\<Xi1IXib>\bigr)(\psi) = 
\eps^{-{d\over 2}}\int_{\R \times D} \psi(z) \int_{E_i} P(z-z')\bigl(\kappa^{(\eps)}_2(z,\fR z') - \kappa^{(\eps)}_2(z,z')\bigr) dz'\,dz
+ \tilde R\;,
\end{equ}
where $|\tilde R|$ is bounded by \eqref{e:boundIntPsi}. The reason for this is that 
 $P-K$ is supported outside of an annulus of radius $1$ and $\psi$ is supported at a distance of at most 
$1/4$ from the reflection plane of $\fR$, and one has $\|z - \fR z'\| \ge 1/2$ for all $z,z'$ with $z \in \supp\psi$
and $\|z - z'\| \ge 1$.

On the other hand, we claim that 
\begin{equ}
\eps^{1-{d\over 2}} c_{i,0} \delta_{\d_{i,0}D} (\psi) 
= \eps^{-{d\over 2}}\int_{\R \times D} \psi(\pi_i z) \int_{E_i} P(z-z')\bigl(\kappa^{(\eps)}_2(z,\fR z') - \kappa^{(\eps)}_2(z,z')\bigr) dz'\,dz\;.
\end{equ}
To see that this is the case, we first perform the change of variables $z = \iota_i(s,v)$ and similarly for $z'$,
and we note that $\pi_i \iota_i (s,v) = \iota_i(0,v)$, so that the right hand side is given by
\begin{equs}
\eps^{-{d\over 2}}&\int_{\R_+ \times \R^{d}} \psi(\iota_i(0,v)) \int_{\R_- \times \R^d} P(\iota_i(s-s',v-v'))\\
&\quad \bigl(\kappa^{(\eps)}_2(\iota_i(s+s',v-v')) - \kappa^{(\eps)}_2(\iota_i(s-s',v-v'))\bigr)\, d(s',v')\,d(s,v) \\
&= \eps^{-{d\over 2}}\int_{\R^{d}} \psi(\iota_i(0,v))\,dv \int_{\R_+} \int_{\R_- \times \R^d} P(\iota_i(s-s',v'))\\
&\quad\qquad \bigl(\kappa^{(\eps)}_2(\iota_i(s+s',v')) - \kappa^{(\eps)}_2(\iota_i(s-s',v'))\bigr)\, d(s',v')\,ds\\
&= \eps^{1-{d\over 2}}\int_{\R^{d}} \psi(\iota_i(0,v))\,dv \int_{\R_+} \int_{\R_- \times \R^d} P(\iota_i(s-s',v'))\\
&\quad\qquad \bigl(\kappa_2(\iota_i(s+s',v')) - \kappa_2(\iota_i(s-s',v'))\bigr)\, d(s',v')\,ds
\end{equs}
Performing the substitution $s' \mapsto -\beta$ and comparing to the definition of $Q_{i,0}$, we conclude that 
\begin{equ}[e:linkToQi]
\E \bigl(\hat \PPi_\eps\<Xi1IXib>\bigr)(\psi) = 
\eps^{-{d\over 2}}\int_{\R_+} Q_{i,0}\left(\frac{s}{\eps}\right)\int_{\R^d}\big((\psi\circ \iota_i)(s,u)-\psi\circ \iota_i)(0,u)\big)\,du\, ds + \tilde R\;.
\end{equ}
By Lemma~\ref{lem:boundQ}, the function $Q_{i,0}$ satisfies the bounds
\begin{equ}
\sup_{s\in \R_+}\left| Q_{i,0}\left(s\right) \right|<\infty\;,\quad \int_{\R_+} |Q_{i,0}\left(s\right)| ds <\infty
\;,\quad \int_{\R_+} |Q_{i,0}\left(s\right)|s\, ds <\infty\;.
\end{equ}
Note now that the integral over $u$ is restricted to the projection of the support of $\psi$ which is 
of volume at most $\lambda^{d+1}$ (since $u$ consists of $d-1$ spatial variables and one time variable).
Bounding $\psi$ by its supremum, it follows that
\begin{equ}
\E \bigl(\hat \PPi_\eps\<Xi1IXib>\bigr)(\psi) \lesssim
\eps^{1-{d\over 2}} \lambda^{d+1} \|\psi\|_\infty \int_{\R_+} |Q_{i,0}(s)|\,ds
+|R|\;.
\end{equ}
On the other hand, we can bound $|(\psi\circ \iota_i)(s,u)-\psi\circ \iota_i)(0,u)|$ by $|s|\,\|D_x\psi\|_\infty$,
which similarly yields
\begin{equ}
\E \bigl(\hat \PPi_\eps\<Xi1IXib>\bigr)(\psi) \lesssim
\eps^{2-{d\over 2}} \lambda^{d+1} \|D_x\psi\|_\infty \int_{\R_+} |Q_{i,0}(s)|s\,ds
+|R|\;.
\end{equ}
Combining these and choosing $\psi = \phi_z^\lambda$ so that $\|\psi\|_\infty\lesssim \lambda^{-d-2}$
and $\|D_x\psi\|_\infty\lesssim \lambda^{-d-3}$ yields a bound of the order
$\eps^{-{d\over 2}}(\eps \lambda^{-1} \wedge \eps^2 \lambda^{-2})$, which does imply \eqref{e:actualBound}
since $-{d\over 2} - \kappa \in (-2,-1)$.

\smallskip\noindent\textbf{The case $d_\psi = 2$.}
We claim that this case can be obtained as a consequence of the bounds for the cases $d_\psi = 0$
and $d_\psi = 1$. We consider the case of dimension $d=2$ (but the computation below is done in a way that keeps
track of dimension and applies to $d=3$ as well) so that by \eqref{e:propW}
\begin{equ}
W_\psi \cap \{(t,x)\,:\, x_1 = x_2 = 0\} \neq \emptyset\;.
\end{equ}
We then fix a smooth function $\chi\colon \R \to \R_+$
such that $\supp \chi \subset [0,2]$ and such that $\sum_{k \in \Z} \chi_k(x)=1$, where $\chi_k(x) = \chi(x-k)$.
For some fixed constant $c > 1$ and integers $n$, $k_i$ and $\ell$, we then set
\begin{equ}
\chi_{n,k,\ell}(z) = \chi_n(-\log |x|/\log 2)\chi_{k_1}(c 2^n x_1)\chi_{k_2}(c 2^n x_2)\chi_\ell(c^2 2^{2n} t)\;.
\end{equ}
By choosing $c$ sufficiently large, we can guarantee that, for every $n$, $k$ and $\ell$, the
function $\chi_{n,k,\ell}$ is such that $d_{\chi_{n,k,\ell}} \in \{0,1\}$. Furthermore, there are
only finitely many values of $k$ (independently of $n$ and $\ell$)
for which $\chi_{n,k,\ell} \neq 0$. This is because $\chi_{n,k,\ell}(S_{2^{n}}z)$ is independent of $n$.
Fix now a test function of the form $\psi = \phi_{z_0}^\lambda$ and write
\begin{equ}
\psi_{n,k,\ell}(z) = \phi_{z_0}^\lambda(z) \chi_{n,k,\ell}(z)\;.
\end{equ}
By construction, one has $\psi_{n,k,\ell} = 0$ for $n$ such that $2^{-n} \ge 2\lambda$,
so that 
\begin{equ}
\lambda_{\psi_{n,k,\ell}} \le 2^{-n}\;,\qquad \|D^m\psi_{n,k,\ell}\|_\infty \lesssim \lambda^{-d-2} \, 2^{n|m|}\;.
\end{equ}
Applying the bounds we already obtained for $d_\psi \in \{0,1\}$, we conclude that
\begin{equ}
\big|\E \bigl(\hat \PPi_\eps\<Xi1IXib>\bigr)(\psi_{n,k,\ell}) \big| \lesssim
\eps^{\kappa} 2^{n({d\over 2}+\kappa)} \lambda^{-d-2}2^{-n(d+2)}\;.
\end{equ}
For any given $n$, the number of values for $k$ and $\ell$ leading to non-vanishing $\lambda_{\psi_{n,k,\ell}}$
is of the order of $(\lambda 2^n)^d$, so that we eventually obtain the bound
\begin{equs}
\big|\E \bigl(\hat \PPi_\eps\<Xi1IXib>\bigr)(\psi) \big| &\lesssim
\eps^{\kappa} \sum_{2^{-n} \ge 2\lambda} (\lambda 2^n)^d 2^{n({d\over 2}+\kappa)} \lambda^{-d-2}2^{-n(d+2)}\\
&= \eps^{\kappa} \lambda^{-2}\sum_{2^{-n} \ge 2\lambda} 2^{n({d\over 2}+\kappa-2)} 
\lesssim \eps^\kappa \lambda^{-{d\over 2}-\kappa}\;,
\end{equs}
as claimed. The case of dimension $d=3$ is identical, the only difference being that $\ell$ now has two components.
Note that this calculation breaks in $d=4$ where the sum over $n$ diverges. This suggests that in this
case one would have to add to $\hat \PPi_\eps\<Xi1IXib>$ an additional correction term that charges 
faces of codimension $2$.

\smallskip\noindent\textbf{The case $d_\psi = 3$.}
This is relevant only for $d=3$, we shall however keep track of $d$ in our calculation to illustrate how this would
behave in higher dimensions. We then proceed in the same way as for the case $d_\psi=2$, making use this time of the fact that 
we already have the required bound for all test functions with $d_\psi < 3$. This time, we have
\begin{equ}
W_\psi \cap \{(t,x)\,:\, x_1 = x_2 = x_3 = 0\} \neq \emptyset\;,
\end{equ}
and we set similarly to above
\begin{equ}
\chi_{n,k,\ell}(z) = \chi_n(-\log |x|/\log 2)\chi_{k_1}(c 2^n x_1)\chi_{k_2}(c 2^n x_2)\chi_{k_3}(c 2^n x_3)\chi_\ell(c^2 2^{2n} t)\;.
\end{equ}
This time, for any given $n$, the number of values for $k$ and $\ell$ leading to non-vanishing $\lambda_{\psi_{n,k,\ell}}$
is of the order of $(\lambda 2^n)^{d-1}$, which then yields similarly to above
\begin{equs}
\big|\E \bigl(\hat \PPi_\eps\<Xi1IXib>\bigr)(\psi) \big| &\lesssim
\eps^{\kappa} \sum_{2^{-n} \ge 2\lambda} (\lambda 2^n)^{d-1} 2^{n({d\over 2}+\kappa)} \lambda^{-d-2}2^{-n(d+2)}\\
&= \eps^{\kappa} \lambda^{-3}\sum_{2^{-n} \ge 2\lambda} 2^{n({d\over 2}+\kappa-3)} 
\lesssim \eps^\kappa \lambda^{-{d\over 2}-\kappa}\;.
\end{equs}
Note that this time the series actually converges for all $d < 6$. 

To conclude, we justify the assumption \eqref{e:propW} and the fact that, 
for $W_\psi \cap \d D = \emptyset$, one has $W_\psi \subset D$.
Indeed, by our assumption on $\lambda_\psi$, it is always possible
to enforce this by applying a finite number of reflections around
the planes $\{x\,:\, x_i = 1/2\}$. If $\supp \psi$ intersects $\d_{i,1} D$ for example, then we reflect
around $x_i = {1\over 2}$ to have $\supp \psi$ intersect $\d_{i,0} D$ instead. The only effect
of this reflection is that, in order to obtain the same answer, we only need to reflect the noise $\eta$
around that plane. The effect of this operation on its covariance function is to change the sign of the
$i$th spatial coordinate of its argument, which is why how we obtain $Q_{i,1}$ rather than $Q_{i,0}$ in
\eqref{e:linkToQi}.
\end{proof}

We now have the main ingredients in place to prove the main result of this section. 

\begin{theorem}\label{theo:correctionTermLp}
With $\hat \PPi_\eps\<Xi1IXib>$ defined as in \eqref{e:defhatPi}, one has for $d \in \{2,3\}$ and any $\kappa > 0$ small enough,
\begin{equ}[e:actualBound2]
\big\| \bigl(\hat \PPi_\eps\<Xi1IXib>\bigr)(\phi_z^\lambda)\big\|_{L^p} \lesssim \eps^{\kappa} \lambda^{-{d\over 2}-\kappa}\;,
\end{equ}
uniformly over $\eps, \lambda \le 1$, $z \in \R^{d+1}$, and $\phi \in \CB_0^1$.
\end{theorem}

\begin{proof}
Writing $\psi = \phi_z^\lambda$, the triangle inequality yields
\begin{equ}
\big\| \bigl(\hat \PPi_\eps\<Xi1IXib>\bigr)(\psi)\big\|_{L^p}
\le 
\big\| \bigl(\hat \PPi_\eps\<Xi1IXib>\bigr)(\psi)- \E \bigl(\hat \PPi_\eps\<Xi1IXib>\bigr)(\psi)\big\|_{L^p}
+ \big|\E \bigl(\hat \PPi_\eps\<Xi1IXib>\bigr)(\psi)\big|\;,
\end{equ}
and we already obtained the required bound on the second term in Proposition~\ref{prop:correctionTerm}, 
so we focus on the first one. Furthermore, $\hat \PPi_\eps\<Xi1IXib>$ differs from $\PPi_\eps\<Xi1IXib>$ by
a deterministic quantity, so that we only need to bound
\begin{equ}
\big\| \bigl( \PPi_\eps\<Xi1IXib>\bigr)(\psi)- \E \bigl( \PPi_\eps\<Xi1IXib>\bigr)(\psi)\big\|_{L^p}\;.
\end{equ}
By \eqref{e:defPPiepscorr} combined with \eqref{e:reprHeat}, this random variable equals
\begin{equ}[e:PPi]
\eps^{-{d\over 2}}\int_{\R \times D} \psi(z) \int_{\R^{d+1}} K(z-z')\bigl(\Wick{\eta_\eps(z)\eta_\eps(\Lambda z')}
- \Wick{\eta_\eps(z)\eta_\eps(z')}\bigr) dz'\,dz\;.
\end{equ}
This time, we will not need to exploit the cancellation between these two terms on $\R\times D$, so we simply bound both
terms separately. The second term equals 
$\bigl( \hat\PPi_\eps\<XiS1IXi>\bigr)(\psi)$, which is  bounded  by the right hand side 
of \eqref{e:actualBound2} by Theorem~\ref{theo:model}.

For the first term, we use the fact that $K$ is compactly supported, so that it can be
bounded by a finite sum of terms of the type
\begin{equ}
\eps^{-{d\over 2}}\int_{\R \times D} \psi(z) \int_{\R \times D} K(z-R z')\,\Wick{\eta_\eps(z)\eta_\eps(z')}\,dz'\,dz\;,
\end{equ}
with $R \in \CG$.

The expectation of the $p$th power of this expression is given by a multiple integral with 
the integrand given by a sum of terms, each of which is a product of heat kernels and of cumulants.
At this stage, we note that the bound in \cite{Ajay} does not exploit any further cancellations, so we
can put absolute values everywhere, bound $K(z-Rz')$ by $\|z-z'\|^{-d}$, and use 
the bounds from that paper.
\end{proof}

\subsection{Bounds on the function \texorpdfstring{$Q$}{Q}}

It remains to prove
\begin{lemma}\label{lem:boundQ}
Under Assumption~\ref{ass:kappa},
\[
	\sup_{s\in\R_+}\left| Q_{i,0}\left(s\right) \right|+\int_{\R_+} |Q_{i,0}\left(s\right)|(1+s)\, ds < \infty\;,
\]
and similarly for $Q_{i,1}$.
\end{lemma}
\begin{proof} 
We fix $i$ and we simply write $Q$ instead of $Q_{i,0}$.
In all the estimates below we will use repeatedly the inequality 
$|(\kappa_2\circ\iota_i)(s-\beta,v)|\vee|(\kappa_2\circ\iota_i)(s+\beta,v)|\lesssim t^{-\underline{c}}\wedge t^{-\overline{c}}$.

We first estimate $\sup_{s\in\R_+}\left| Q\left(s\right) \right|$.
Here $\beta\in\R_+$, $v=(t,x)$, with $x\in\R^{d-1}$. 
\begin{equs}
\left| Q\left(s\right) \right|&\lesssim\int_0^\infty\int_0^\infty\int_{\R^{d-1}} \hat P(s+\beta,v)t^{-\underline{c}}\wedge t^{-\overline{c}}d\beta\, dv\\
&\lesssim \int_0^\infty\int_0^\infty\frac{e^{-(s+\beta)^2/(4t)}}{\sqrt{t}}t^{-\underline{c}}\wedge t^{-\overline{c}}d\beta \,dt\\
&\lesssim \int_s^\infty\int_0^\infty\frac{e^{-1/t}}{\sqrt{t}}|a|\,\left(\big(a^2t\big)^{-\underline{c}}\wedge \big(a^2t\big)^{-\overline{c}}\right)\,da\, dt
\\
&\lesssim\int_0^\infty\frac{e^{-1/t}}{t^{1-\delta}}dt\int_0^{t^{-1/2}}a^{2\delta}da
+\int_0^\infty \frac{e^{-1/t}}{t^{\overline{c}+1/2}}dt\int_{t^{-1/2}}^\infty \frac{da}{a^{2\overline{c}-1}}\\
&\lesssim\int_0^\infty\frac{e^{-1/t}}{t^{3/2}} dt +\int_0^\infty \frac{e^{-1/t}}{t^{3/2}}dt <\infty\, ,
\end{equs}
and the bound does not depend upon $s$. Here, to go from the second to the third line, we 
set $a = s+\beta$ and we performed the substitution $t \mapsto a^2 t/4$.

It remains to estimate $\int_{\R_+} |Q\left(s\right)|(1+s) ds$. Again $v=(t,x)$, with $x\in\R^{d-1}$.
\begin{equs}
\int_{\R_+} |Q\left(s\right)| &(1+s)ds \lesssim\int_{0}^\infty\int_0^\infty\int_{\R^d}(1+s)\hat P(s+\beta,v)t^{-\underline{c}}\wedge t^{-\overline{c}}dv\,d\beta\, ds\\
&\lesssim \int_0^\infty\int_{\R^d}(a+a^2)\hat P(a,v)t^{-\underline{c}}\wedge t^{-\overline{c}}dv\,da\\
&\lesssim \int_0^\infty\int_0^\infty (a^2+a^3)\frac{e^{-1/t}}{\sqrt{t}}\left(a^2t\right)^{-\underline{c}}\wedge \left(a^2t\right)^{-\overline{c}}dt\,da\\
&\lesssim\int_0^\infty\frac{e^{-1/t}}{t^{1-\delta}}dt\int_0^{t^{-1/2}}(a^{1+2\delta}+a^{2+2\delta})da\\&\qquad
+\int_0^\infty \frac{e^{-1/t}}{t^{\overline{c}+1/2}}dt\int_{t^{-1/2}}^\infty (a^{2-2\overline{c}}+a^{3-2\overline{c}})da\\
&\lesssim\int_0^\infty\left(\frac{1}{t^{2}}+\frac{1}{t^{5/2}}\right)e^{-1/t} dt +\int_0^\infty\left(\frac{1}{t^{2}}+\frac{1}{t^{5/2}}\right)e^{-1/t} dt <\infty\,,
\end{equs}
thus concluding the proof.
\end{proof}

\section{Auxiliary results}

In this section, we collect a number of results that are more or less straightforward
consequences of known results, specialised to our setting. 
Throughout this section, we assume that we are working with 
the regularity structure defined in Section~\ref{sec:structure} and that $\xi_\eps$ is
defined as in \eqref{e:defveps} and satisfies Assumption~\ref{ass:kappa}.

We will write $\CC^\gamma$ as a shorthand for $\CD^\gamma(\bar T)$, where
$\bar T$ is the sector spanned by the Taylor polynomials, and similarly
for $\CC^{\gamma,\eta}$, etc. Note that for $\gamma \not\in\N$ this is consistent with the usual
definition of $\CC^\gamma$.

\begin{proposition}\label{prop:commuteSpecialCase}
Let $\zeta^+, \zeta \in \CC^{-{5\over 2}-\kappa}$ be such that $\zeta^+(\phi) = 0$ 
for $\phi$ supported in $\{t < 0\}$ and $\zeta^+(\phi) = \zeta(\phi)$ 
for $\phi$ supported in $\{t > 0\}$ and let $\PPi$ be an admissible model with
$\PPi \<XiS> = \zeta$. 
Write $\CK \one_+\<XiS> = \CK^{\hat \zeta} \one_+\<XiS>$ for $\CK^{\hat \zeta}$ defined as in \cite[Sec.~4.5]{Mate}.
Define $\CK_\Neu \one_+\<XiS>$, $\CK_i \one_+\<XiS>$ and $\CK_{\Neu,i} \one_+\<XiS>$ analogously
 and, given $\Phi$, let $V^{(0)}$ be given by \eqref{e:defV0} for some
 $\u \in \CC^{3-\kappa}$.
 
Then, setting $\bar \gamma = {3\over 2}-\kappa$ and $\bar \eta = -{1\over 2}-\kappa$,
there exists a choice $\Phi \in \CC^{\bar \gamma,\bar \eta}$ such that 
$\CR V^{(0)} = K_\Neu \bigl(G(\u)\zeta^+\bigr)$ and $V^{(0)}$ belongs to $\CD^{\bar \gamma,\bar \eta}$. 
In particular, if $\zeta^+ \in \CC^\alpha$ for some
$\alpha > -1$ is supported in $\R_+ \times D$ then, for $t \in [0,1]$,  $\CR V^{(0)}$
coincides with the solution of $\d_t v = \Delta v + G(\u) \zeta^+$ with vanishing
initial condition, endowed
with Neumann (respect.\ Dirichlet) boundary conditions.

If furthermore $\zeta_n \to \zeta$ in $\CC^{-{5\over 2}-\kappa}$ and $\PPi_n \to \PPi$
as admissible models, then
one has $\$V^{(0)}_n ; V^{(0)}\$_{\bar \gamma,\bar \eta} \to 0$.
\end{proposition}

\begin{remark}
In principle, the model $\PPi$ does contain non-trivial information through its action
on $\<dIXiS>$. This is because the kernel $K_\d$
is not $2$-regularising (condition~5.5 in \cite[Ass.~5.1]{regularity} fails to be satisfied), so 
that the extension theorem \cite[Thm~5.14]{regularity} cannot be applied here.
\end{remark}

\begin{proof}
We aim to apply Corollary~\ref{cor:commutationSingular}.
Let $\gamma = {1\over 2}-2\kappa$, $\eta = -{5\over 2} -\kappa$ (so that in particular
$\gamma - \eta = 3-\kappa$), and let 
$\CB = \CC^{\gamma-\eta}$ (on which $\CC^{\gamma-\eta}$ then acts canonically by multiplication)
with the injection $\iota g = (\CL_{\gamma-\eta}g)\,\one_+ \<XiS> \in \CD^{\gamma,\eta}$.
(This is actually independent of the model $\PPi$.)
Given $g \in \CB$, we set
\begin{equ}
\hat \CR g = g \,\zeta^+\;,
\end{equ}
which is consistent with the reconstruction operator by our assumption
on $\zeta^+$. We are therefore in the setting of Corollary~\ref{cor:commutationSingular} provided
that we set $K_0 = K_\Neu$.

This guarantees that we can find $\Phi \in \CC^{\gamma+2,\eta+2}$ with the 
desired properties. 
The continuity as a function of $\zeta^+$ and the model $\PPi$
follows from the corresponding continuity statement in Corollary~\ref{cor:commutationSingular}.
\end{proof}

\begin{proposition}\label{prop:specialSymbol}
For every $g \in \CC^{2-\kappa}$ one can find $\Phi$ taking values in the Taylor polynomials 
such that, setting
\begin{equ}
V = \one_+^D\,g (\hat\PPi_\eps \<IXib>) \<IXi1> + \Phi\;,
\end{equ}
one has $V \in \CD^{2-2\kappa,\bar w}$ with
$\bar w = \big({1\over 2}-2\kappa,{1\over 2}-2\kappa,0\big)$, and $\CR_\eps V = K \big(\one_+g \hat \PPi_\eps \<Xi1IXib>\big)$.
\end{proposition}

\begin{proof}
We make use of \cite[Lem.~4.12]{Mate} and Corollary~\ref{cor:commutationSingular}. For this, similarly to above,
we set $\CB = \CC^{2-\kappa}$ and, for $g \in \CB$, we set
\begin{equ}
\iota g \eqdef \one_+^D \CL_{2-\kappa}\bigl(g \hat \PPi_\eps \<IXib>\bigr)\<Xi1>\;,\qquad
\hat \CR g \eqdef \one_+\,g \hat \PPi_\eps\<Xi1IXib>\;.
\end{equ}
Note that as a consequence of Proposition~\ref{prop:diffKernel}
one has $\hat \PPi_\eps \<IXib> \in \CC^{\tilde \gamma,\tilde w}$ 
for any $\tilde \gamma > 0$ and for $\tilde w = \big(-{1\over 2}-\kappa\big)_3$,
where $(\eta)_3 \eqdef (\eta,\eta,\eta)$.
Since $\deg \<Xi1> = -1-\kappa$, it follows from \cite[Lem.~4.3]{Mate} that
$\iota g \in \CD^{\gamma, w}$ for $\gamma = 1-2\kappa$ and
for $w = \bigl(-{3\over 2}-2\kappa\bigr)_3$.
It then follows from \cite[Lem.~4.12]{Mate} that
one can find a modelled distribution $V \in \CD^{2-2\kappa,\bar w}$ 
with $\bar w = \big({1\over 2}-2\kappa,{1\over 2}-2\kappa,0\big)$ of the form
\begin{equ}
V = \one_+^D\bigl(g \hat \PPi_\eps \<IXib>\bigr)\<IXi1> +  \hat \Phi\;,
\end{equ}
with $\hat \Phi$ taking values in the Taylor polynomials and
such that \[\CR_\eps V = K \hat \CR g  = K \big(\one_+ g \hat \PPi_\eps \<Xi1IXib>\big)\;,\]
thus concluding the proof.
\end{proof}

\begin{proposition}\label{prop:specialSymbol2}
Equip the regularity structure $(\CT,\CG)$ with an admissible model $\PPi$
in the sense of Definition~\ref{def:admissible} and
let $G$ be a modelled distribution of the form
\begin{equ}
G = \bigl(\CL(g_1)\<IXi> + \CL(g_{2,i})\CK_i(\<Xi>) + \CL(g_3)\bigr)\;,
\end{equ}
for some functions $g_1, g_2 \in \CC^{2-\kappa}$ and $g_3 \in \CC^{{3\over2}-\kappa,w}$ with 
$w = \bigl(-{1\over 2}-\kappa,{1\over 2}-\kappa,{1\over 2}-\kappa\bigr)$. Then, there exists
a unique modelled distribution which we call $\CR(\<XiS1> G)$ such that 
$\CR(\<XiS1> G) = \tilde \CR(\<XiS1> G)$ on test functions whose support does not intersect $\R \times \d D$
and such that 
\begin{equ}
\big(\CR(\<XiS1> G)-\Pi_x(\<XiS1> G(x))\big)(\psi_x^\lambda)\lesssim \lambda^{-\f12 -2\kappa} 
\end{equ}
locally uniformly over $x \in \R \times D$ and uniformly over $\lambda \in (0,1]$.

Furthermore, there exists $V  \in \CD^{2-2\kappa,\bar w}$ with
$\bar w = \big({1\over 2}-2\kappa,{1\over 2}-2\kappa,0\big)$
taking values in the translation invariant sector, and
such that $\CR V = K \CR(\<XiS1> G)$. The function $V$ is of the form
\begin{equ}
V = \one_+ \CI(\<Xi1> G) + \Phi\;,
\end{equ}
with $\Phi$ taking values in the Taylor polynomials,
and the map $(g_1,g_2,g_3,\PPi) \mapsto V$ is uniformly Lipschitz continuous on bounded sets.
\end{proposition}

\begin{proof}
We use again the same strategy of proof as in Proposition~\ref{prop:specialSymbol}, but 
this time we take as our space $\CB$ the space of triples $g = (g_1,g_{2,i},g_3)$
as in the statement of the proposition and we set
\begin{equ}
\iota g = \one_+ \<Xi1> G\;,\qquad
\hat \CR g =  \CR \bigl(\<XiS1> G\bigr)\;,
\end{equ}
where $ \CR$ is the reconstruction operator given by Theorem~\ref{thm:reconstructDomain}. 
This time, we have $\deg \<Xi1IXi> = -{3\over 2}-2\kappa$ and 
$\deg \<Xi1> = -1-\kappa$, so that it follows from \cite[Lem.~4.3]{Mate} that   
$\iota g, \<XiS1> G \in \CD^{\gamma, w}$ for $\gamma = {1\over 2}-2\kappa$ and
 $w = \bigl(-{3\over 2}-2\kappa,-\f12-2\kappa,-\f32-2\kappa\bigr)$. 

This shows that Theorem~\ref{thm:reconstructDomain} can indeed be applied to this situation since
furthermore our admissible models
are such that $\hat \PPi_\eps(\<XiS1> \tau)(\phi) = 0$ as soon as $\phi$ is 
supported outside of $\R \times D$.
The remainder of the proof then follows from an application of \cite[Lem.~4.12]{Mate} in the same
way as in the proof of Proposition~\ref{prop:specialSymbol}.
\end{proof}

\begin{remark}
Note that the results of \cite{Mate} do not apply here since 
$\deg \<Xi1IXi>$ and $\deg \<Xi1>$ are both strictly below $-1$. Our saving grace
is that the coefficients are sufficiently well-behaved near the boundary of the domain.
\end{remark}

\appendix

\section{Extension of the kernel}
\label{app:extHeat}

We fix a function $K \colon \R^{d+1} \to \R$ which is smooth, non-anticipative 
(i.e.\ which is supported on positive times), even in the spatial variable, agrees with the heat kernel
on $[0,1] \times [-1,1]^d$, and is supported on $[0,2] \times [-2,2]^d$. 
As usual, we write $K = \sum_{n \ge 0} \tilde K_n$ with $\tilde K_n$ satisfying the conditions
of Definition~\ref{def:regK} with $\beta = 2$.
We furthermore write
\begin{equ}
K_i(x,y) = (y-x)_i K(y-x)\;.
\end{equ}

In order to implement integration against the heat kernel with Neumann boundary
conditions, we set
\begin{equ}[e:defKNeu]
K_\Neu(z,z') = 
\sum_{n \ge 0} \sum_{R \in \CG}  \phi(2^n \|z-z'\|)\, \tilde K_n (z-R(z'))\;,
\end{equ}
(recall that the reflection group $\CG$ was defined in Section~\ref{sec:decomp})
where $\psi \colon \R_+ \to \R_+$ is a smooth function such that $\phi(r) = 1$ for
$r \le 2$ and $\phi(r) = 0$ for $r \ge 3$. 
The kernel $K_\Dir$ is defined similarly.

Furthermore, it is obvious that Assumption~\ref{def:regK} is satisfied with $\beta = 2$, since
this is the case for $K$ itself.
Note that this would not be true if it weren't for the presence of the cutoff
$\phi$ in \eqref{e:defKNeu} since the kernel without cutoff has singularities at
$z = R(z')$ for  all reflections $R \in \CG$.

\begin{proposition}\label{prop:NeumannBC}
One has $K_\Neu(z,z') = \sum_{R \in \CG} K(z-R(z'))$ for $z,z' \in ([0,1]\times D)^2$.
\end{proposition}

\begin{proof}
Recall that $\tilde K_n$ is supported in the ball of radius $2^{-n}$ and note that
$\|z-R(z')\| \ge \|z-z'\|$ whenever both $z$ and $z'$ belong to $\R\times D$.
As a consequence, for every $z,z' \in ([0,1]\times D)^2$, every $n\ge 0$ and every $R \in \CG$,
one has either $\tilde K_n (z-R(z')) = 0$ or $\phi(2^n \|z-z'\|) = 1$, so that one can
replace $\phi$ by $1$ in \eqref{e:defKNeu}.
\end{proof}

In order to state the main result of this appendix, we write $\tilde D = \R_+ \times D$ as a shorthand.

\begin{proposition}\label{prop:diffKernel}
Let $K$, $K_\d$, $K_i$ and $K_{\d,i}$ be as above and let $\zeta, \zeta^c \in \CC^{\alpha}$
with $\alpha \le -2$ be such that $\zeta$ is supported on $\tilde D$ and $\zeta^c$ is supported on
its complement. 
Then, restricted to $\tilde D$, both $K \zeta^c$ and 
$K_{\d} \zeta$ belong to $\CC^{\gamma,w}$ for $w = \bigl(\alpha+2\bigr)_3$ and any $\gamma > 0$.
Furthermore, when restricted to $\tilde D^c$, $K \zeta$ belongs to $\CC^{\gamma,w}$.

The same statements hold when $\alpha \le -3$ with $K$ and $K_\d$ replaced by $K_i$ and
$K_{\d,i}$ respectively and $w = \bigl(\alpha+3\bigr)_3$.
\end{proposition}

\begin{proof}
Consider the case $z \not \in \tilde D$, so that
\begin{equ}
\bigl(D^k K \zeta\bigr)(z) = \sum_{n \ge 0} \scal{\zeta, D^k \tilde K_n(z-\cdot)}\;.
\end{equ}
We now note that $\scal{\zeta, D^k \tilde K_n(z-\cdot)} = 0$ whenever $2^{-n} \le d(z, \tilde D)$,
while in general 
\begin{equ}[e:boundKzeta]
|\scal{\zeta, D^k \tilde K_n(z-\cdot)}| \lesssim 2^{-n(\alpha+2-|k|)}\;,
\end{equ}
by definition of $\CC^\alpha$. It immediately follows that one has
\begin{equ}[e:wantedBound]
|\bigl(D^k K \zeta\bigr)(z)| \lesssim d(z, \tilde D)^{\alpha+2-|k|}\;,
\end{equ}
so that one has indeed $K \zeta \in \CC^{\gamma,(\alpha+2)_3}$ for every $\gamma > 0$.
The case of $K \zeta^c$, $K_i \zeta$ and $K_i \zeta^c$ is dealt with in exactly the same way.

Consider now the case $z\in \tilde D$ and define
\begin{equ}
\tilde K_n^R(z,z') = \phi(2^n |z-z'|)\, \tilde K_n (z-R(z'))\;.
\end{equ}
We then have the identity
\begin{equ}[e:sumKd]
\bigl(D^k K_\d \zeta\bigr)(z) = \sum_{n \ge 0} \sum_{R \in \CG \setminus\{\id\}} 
\scal{\zeta, D^k \tilde K_n^R(z,\cdot)}\;.
\end{equ}
It follows again from the definition of $\CC^\alpha$ that 
the bound \eqref{e:boundKzeta} holds with $\tilde K_n(z-\cdot)$ replaced by $\tilde K_n^R(z,\cdot)$.

We also note that $\tilde K_n^R(z,\cdot) = 0$ unless there exists a point $z'$ such that one has
on one hand $\|z-z'\| \lesssim 2^{-n}$ and on the other hand $\|z-R(z')\| \lesssim 2^{-n}$. 
In particular, there exists a constant $C$ such that, if $d(z, \d\tilde D) \ge C 2^{-n}$,
one has $\tilde K_n^R(z,\cdot) = 0$ unless $R = \id$ (which is excluded from the sum \eqref{e:sumKd}).
If on the other hand $d(z, \d\tilde D) \le C 2^{-n}$, then $\tilde K_n^R(z,\cdot)$ is only non-zero 
for at most eight different reflections $R$.
Combining these observations shows as before that the bound \eqref{e:wantedBound}
holds with $K\zeta$ replaced by $K_\d \zeta$. The proof for $K_\d$ replaced by $K_{\d,i}$ is virtually identical.
\end{proof}

\begin{corollary}\label{cor:boundaryTerm}
For $\alpha \in (-3,-2]$  and $\zeta \in \CC^\alpha$ supported on $\R_+ \times D$, one has
$K_{\d} \zeta \in \CC^{\alpha + 2}$.
\end{corollary}

\begin{proof}
This follows immediately from Proposition~\ref{prop:diffKernel}, combined with the fact that 
$\CC^{\gamma,(\eta)_3} \subset \CC^{\eta}$ whenever $\gamma \ge \eta > -1$.
\end{proof}

\begin{proposition}\label{prop:R+}
Let $\zeta \in \CC^{\alpha}$ with $\alpha \le -2$ be supported on $\R_- \times \R^d$. 
Then, restricted to $\R_+ \times \R^d$, $K_\Neu \zeta$ belongs to $\CC^{\gamma,w}$ for 
$w = \bigl(\alpha+2,\gamma, \alpha+2\bigr)$ and any $\gamma > 0$.
\end{proposition}

\begin{proof}
With the same notations as above, we have for $z = (t,x)$,
\begin{equ}
\bigl(D^k K_\Neu \zeta\bigr)(z) = \sum_{n \ge 0} \sum_{R \in \CG} 
\scal{\zeta, D^k \tilde K_n^R(z,\cdot)}\;.
\end{equ}
As before, the summand vanishes as soon as $2^{-n} \lesssim \sqrt{|t|}$
and, for all $z \in \R_+ \times \R^d$ and $n\ge 0$, only at most a fixed number of test functions 
$\tilde K_n^R(z,\cdot)$ are non-vanishing, so that we obtain for $t > 0$ the bound
\begin{equ}
\big|\bigl(D^k K_\Neu \zeta\bigr)(t,x)\big| \lesssim t^{{\alpha + 2\over 2}-k_0}\;.
\end{equ}
This then implies the required bound by the definition of the spaces $\CC^{\gamma,w}$.
\end{proof}

\section[Integration and multiplication almost commute]{Integration and multiplication by smooth functions almost commute}
\label{sec:almostComm}

We will use a form of the multilevel Schauder theorem of \cite[Sec.~5]{regularity} with slightly weaker
assumptions on the kernel $K$. In this section, we fix a space-time scaling $\s$ as in \cite{regularity}.
(In our case this would be the parabolic scaling $\s = (2,1,\ldots,1)$.)

\begin{definition}\label{def:regK}
A function $K\colon \R^d \times \R^d \to \R$ is said to be $\beta$-regularising if it
can be decomposed as
\begin{equ}[e:defKbar]
K(x,y) = \sum_{n \ge 0} \tilde K_n(x,y)\;,
\end{equ}
where the functions $\tilde K_n$ have the following properties:
\begin{claim}
\item For all $n \ge 0$, the map $\tilde K_n$ is supported in the set $\{(x,y)\,:\, \|x-y\|_\s \le 2^{-n}\}$.
\item For any two multiindices $k$ and $\ell$, there exists a constant $C$ such that
the bound
\begin{equ}[e:propK1]
\bigl|D_x^k D_y^\ell \tilde  K_n(x,y)\bigr| \le C 2^{(|\s| - \beta + |\ell|_\s + |k|_\s)n}\;,
\end{equ}
holds uniformly over all $n \ge 0$ and all $x,y \in \R^d$.
\end{claim}
\end{definition}

\begin{remark}
This is identical to \cite[Ass.~5.1]{regularity}, but with the last condition absent.
It turns out that the only place where the third condition of  \cite[Ass.~5.1]{regularity} is ever used in
\cite[Sec.~5]{regularity} is in the proof of
\cite[Lemma~5.19]{regularity}, which in turn is only used for the proof of the extension 
theorem, \cite[Thm~5.14]{regularity}. 
\end{remark}

%
%

We assume that we are given a regularity structure of the type studied in \cite{regularity,Lorenzo}, 
endowed with a family of integration maps $\CI_k$ for multiindices $k \in \N^d$. We also assume that 
we are given a model $(\Pi,\Gamma)$ which is admissible for the collection of kernels $K_k$ such that
\begin{equ}[e:admissible]
K_k(x,y) = (y-x)^k K_0(x,y)\;.
\end{equ}
We will furthermore assume that our regularity structure contains the polynomial structure on
$\R^d$ (for some fixed scaling) and that admissible models satisfy the 
usual identity $\Pi_x \X^k\tau = (\fat - x)^k \Pi_x \tau$ for every $\tau \in \CT$.
Assuming that $\CI_0$ is of order $\beta$, we assume that $\CI_k$ is of order $\beta + |k|_\s$,
which is compatible with \eqref{e:admissible} in the sense that if $K_0$ is $\beta$-regularising in
the sense of Definition~\ref{def:regK}, then $K_k$ is $(\beta + |k|_\s)$-regularising.

\begin{lemma}\label{lem:commute}
Let $K_0$ be a $\beta$-regularising kernel for some $\beta > 0$ and
let $(\Pi,\Gamma)$ be an admissible model for the collection of kernels $K_k$ 
given in \eqref{e:admissible}. Then, one has the identities
\begin{equs}
\Pi_x \CI_k(X^\ell \tau) &= \sum_{m \le \ell} \binom{\ell}{m} \Pi_x X^{\ell-m} \CI_{k+m}(\tau)\;,\label{e:combinatorial}\\
\Gamma_{xy}  \CI_k(X^\ell \tau) -  \CI_k(\Gamma_{xy}(X^\ell \tau))
&= \sum_{m \le \ell} \binom{\ell}{m} 
\Gamma_{xy} X^{\ell-m} \bigl(\Gamma_{xy}\CI_{k+m}(\tau) - \CI_{k+m}(\Gamma_{xy}\tau)\bigr)\;.
\end{equs}
\end{lemma}

\begin{proof}
We first consider the identity for $\Pi_x$.
We can assume that $k=0$ since the general case then follows at once by simply setting
$\tilde \CI_\ell = \CI_{k+\ell}$.
We also assume without loss of generality that $\tau$ is homogeneous of degree $\alpha$, 
so that 
\begin{equs}
\bigl(\Pi_x \CI_0(X^\ell \tau)\bigr)(y) &= \int (z-x)^\ell K(y,z)\, \bigl(\Pi_x \tau\bigr)(dz) \label{e:firstIden}\\
&\quad - \sum_{k} {1\over k!} \int (z-x)^\ell (y-x)^k \,D_x^k K(x,z)\, \bigl(\Pi_x \tau\bigr)(dz)\;,
\end{equs}
where the sum is constrained by $|k|_\s < \alpha + \beta + |\ell|_\s$.
On the other hand, one has
\begin{equs}
\sum_{m} \binom{\ell}{m} &\bigl(\Pi_x X^m \CI_{n}(\tau))\bigr)(y) = 
\sum_m \binom{\ell}{m} (y-x)^m \int (z-y)^{\ell-m} K(y,z)\, \bigl(\Pi_x \tau\bigr)(dz) \\
&- \sum_{m,k} \binom{\ell}{m}(y-x)^m \int {(y-x)^{k} \over k!}D_x^k \big((z-x)^{\ell-m} K(x,z)\bigr)\, \bigl(\Pi_x \tau\bigr)(dz)\;,
\end{equs}
where the sum this time is constrained by $|k|_\s < \alpha + \beta + |\ell-m|_\s$.
The first term in this expression clearly equals the first term of \eqref{e:firstIden}
by the binomial theorem, so we only need to consider the second term. The integrand can be written as
\begin{equs}
\sum_{m,k} &\binom{\ell}{m}(y-x)^m {(y-x)^{k} \over k!}D_x^k \big((z-x)^{\ell-m} K(x,z)\bigr) \\
&= \sum_{m,k,p} {  (-1)^p \ell! \over m! (\ell-m-p)! p! (k-p)!}
(y-x)^{k+m} (z-x)^{\ell-m-p} D_x^{k-p} K(x,z)\;,
\end{equs}
where $p$ is constrained by $p \le k \wedge (\ell-m)$.
Setting $q = k-p$ and $r = k+m$, this can be written as
\begin{equ}
\sum_{q,r,p} {(-1)^{p} \ell! \over (r-q-p)! (\ell-r+q)! p! q!}
(y-x)^{r} (z-x)^{\ell-r+q} D_x^{q} K(x,z)\;,
\end{equ}
where the sum is constrained by the fact that all three variables are positive mutiindices and furthermore
\begin{equ}
|r|_\s < \alpha + \beta + |\ell|_\s\;,\quad p+q \le r\;,\quad
r \le \ell+q\;.
\end{equ}
It follows that for any fixed values of $q$ and $r$ the above sum vanishes, except when
$r-q = 0$, so that it equals
\begin{equ}
\sum_{q} {1 \over q!}
(y-x)^{q} (z-x)^{\ell} D_x^{q} K(x,z)\;,
\end{equ}
constrained by $|q|_\s < \alpha + \beta + |\ell|_\s$, which is precisely the 
integrand appearing in the second term of \eqref{e:firstIden}.

The second identity immediately follows from the first one. Indeed, 
both sides take values in the Taylor polynomials by the definition of an admissible model.
Furthermore, the first identity implies that
\begin{equs}
\Pi_x \CI(\Gamma_{xy}(X^\ell \tau))
&= \sum_{p \le \ell} \binom{\ell}{p} (x-y)^{\ell-p} \Pi_x \CI(X^p \Gamma_{xy} \tau) \label{e:calculation}\\
&= \sum_{p \le \ell} \sum_{m \le p} \binom{\ell}{p} \binom{p}{m} (x-y)^{\ell-p} \Pi_x X^{p-m} \CI_{m}(\Gamma_{xy}\tau) \\
&= \sum_{m \le \ell} \sum_{n \le \ell-m} \binom{\ell}{m} \binom{\ell-m}{n} (x-y)^{\ell-m-n} \Pi_x X^{n} \CI_{m}(\Gamma_{xy}\tau)\\
&= \sum_{m \le \ell} \binom{\ell}{m}  \Pi_x \bigl((\Gamma_{xy} X^{\ell-m}) \CI_{m}(\Gamma_{xy}\tau)\bigr)\;.
\end{equs}
As a consequence of combining this with the first identity of \eqref{e:combinatorial}, 
both sides of the second identity of \eqref{e:combinatorial} are equal after applying $\Pi_x$
to them. Since furthermore both sides belong to the space of Taylor polynomials on which 
$\Pi_x$ is injective, the claim follows.
\end{proof}

\begin{corollary}\label{cor:commutation}
In the context of Lemma~\ref{lem:commute}, write $\CK_k$ for the integration operator
associated to $\CI_k$ as in \cite{regularity} and write $\CL_\gamma$ for the
Taylor lift $\CC^\gamma \to \CD^\gamma$. Then, for every $F \in \CD_\chi^\gamma$ with
$\chi \le 0 < \gamma$
and $g \in \CC^{\theta-\chi}$ with $\theta \in (0,\gamma]$ (and such that $\theta + \beta \not \in \Z$), 
one can find a function $\phi \in \CC^{\theta + \beta}$ such that, setting 
\begin{equ}[e:defG]
G = \CL_{\theta+\beta} \phi + \sum_{|\ell|_\s < \theta-\chi}{1\over \ell!} \CL_{\theta-\chi -|\ell|_\s}(D^\ell g) \,\CK_{\ell} \bigl(F\bigr) \;,
\end{equ}
one has $G \in \CD^{\theta + \beta}$ and $\CR G = K (g\CR F)$.
\end{corollary}

\begin{proof}
A straightforward calculation virtually identical to \eqref{e:calculation}
shows that, as a consequence of the first identity
of Lemma~\ref{lem:commute}, one has
\begin{equ}[e:identityBasic]
\Pi_x \CI(\CL_{\theta-\chi}(g)(x)F(x)) = 
\sum_{|\ell|_\s < \theta-\chi}{1\over \ell!}  \,\Pi_x \bigl(\CL_{\theta-\chi-|\ell|_\s}(D^\ell g)(x) \CI_{\ell} \bigl(F(x)\bigr)\bigr)\;.
\end{equ}

For $F\in \CD^\gamma_\chi$, write $\CN_k (F) = \CK_k(F) - \CI_k(F)$, so that $\CN_k (F)$ takes values in the 
Taylor polynomials of degree at most $\gamma+\beta$, and set
\begin{equ}
\Phi = \CN_0 (\CL_{\theta-\chi}(g)F) - \sum_{|\ell|_\s < \theta-\chi}{1\over \ell!} \CL_{\theta-\chi-|\ell|_\s}(D^\ell g) \,\CN_{\ell} \bigl(F\bigr)\;.
\end{equ}
We then set
\begin{equs}
 G &= \Phi + \sum_{|\ell|_\s < \theta-\chi}{1\over \ell!} \CL_{\theta-\chi-|\ell|_\s}(D^\ell g) \,\CK_{\ell} \bigl(F\bigr) \\
&= \CN_0 (\CL_{\theta-\chi}(g)F) + \sum_{|\ell|_\s < \theta-\chi}{1\over \ell!} \CL_{\theta-\chi-|\ell|_\s}(D^\ell g) \,\CI_{\ell} \bigl(F\bigr)\;.
\end{equs}
Combining this with \eqref{e:identityBasic}, it follows that, setting
$\bar G = \CK_0(\CL_{\theta-\chi}(g)F)$, one has
\begin{equ}[e:TaylorG]
\Pi_x  G(x) =
\Pi_x \CN_0 (\CL_{\theta-\chi}(g)F)(x) + \Pi_x \CI(\CL_{\theta-\chi}(g)(x)F(x))
= \Pi_x \bar G(x)\;.
\end{equ}
Combining this with the second identity of Lemma~\ref{lem:commute} and writing $\bar\CQ$
for the projection onto the Taylor polynomials, we conclude that
\begin{equ}
\bar\CQ\bigl( G(x) - \Gamma_{xy} G(y)\bigr) =
\bar\CQ\bigl(\bar G(x) - \Gamma_{xy}\bar G(y)\bigr)\;.
\end{equ} 
Since furthermore $\bar G$  and each of the terms
$\CL_{\theta-\chi-|\ell|_\s}(D^\ell g) \,\CK_{\ell} \bigl(F\bigr)$
belong to $\CD^{\theta + \beta}$, we conclude that one
necessarily has $ G \in \CD^{\theta + \beta}$. This in turn implies that
$\Phi \in \CD^{\theta + \beta}$, so that it is the lift of a 
function $\phi \in \CC^{\theta+\beta}$. By \eqref{e:TaylorG},
we furthermore have $\CR G = \CR \bar G = K (g\CR F)$, thus concluding the proof.
\end{proof}

One simple but very useful corollary of this result can be formulated as follows.

\begin{corollary}\label{cor:commute}
Let $\zeta \in \CC^\chi$ with $\chi \le 0$, let $K$ and $K_k$ be as above, and let $g \in \CC^{\theta-\chi}$ with 
$\theta > 0$ and $\theta + \beta \not \in \Z$. Then,
\begin{equ}
K(g\zeta) - \sum_{|\ell|_\s < \theta-\chi} {D^\ell g \over \ell!} K_\ell \zeta \in \CC^{\theta + \beta}\;. 
\end{equ}
\end{corollary}

\begin{proof}
It suffices to consider the case of a regularity structure with symbol $\sXi$ (plus Taylor 
polynomials and their products with $\sXi$) and model mapping $\sXi$ to $\zeta$. We 
then apply the reconstruction operator to both sides of \eqref{e:defG} with $F = \sXi$.
\end{proof}

In our context, we will need an analogous result, but 
for the spaces $\CD^{\gamma,\eta}$ of \cite[Sec.~6]{regularity}. 
Furthermore, we will need to be able to cover situations in which 
\cite[Prop.~6.16]{regularity} does not apply because we consider elements taking
values in a sector of regularity below $-2$, so that the reconstruction theorem
 \cite[Prop.~6.9]{regularity} fails. We therefore make use instead of \cite[Lem.~4.12]{Mate}
which, given $F\in \CD^{\gamma,\eta}(V)$ with $V$ of regularity $\alpha \ge \eta$, allows us to 
specify a distribution $\zeta \in \CC^\eta$ such that 
$(\CR F)(\phi) = \zeta(\phi)$ for every test function $\phi$ whose support does not intersect
the plane $P = \{(t,x)\,:\, t = 0\}$. We then have the following result.

\begin{corollary}\label{cor:commutationSingular}
Let $\gamma > 0$, let $V$ be a sector of regularity $\alpha\le 0$, and let $w = (\eta,\sigma,\mu)$
with $\eta, \sigma, \mu \le \alpha$ and $\eta \le \sigma \wedge \mu$, $\eta  + \beta > -2$, 
$\sigma+\beta > -1$. 
For each admissible model $\PPi$, 
let $\CB$ be a Banach space equipped with a bounded map 
$\iota \colon \CB \to \CD^{\gamma,w}(V)$. (Since the latter depends on
$\PPi$ in general, this can also be the case for $\CB$ and / or $\iota$. In this case, we 
assume that $\iota$ is bounded independently of the underlying model.)
Let furthermore $\hat \CR \colon \CB \to \CC^{\eta}$
be a continuous linear operator such that 
$(\hat \CR F)(\phi) = (\CR \iota F)(\phi)$ for every $F \in \CB$ and every test function 
$\phi$ whose support does not intersect the two boundaries.

We furthermore assume that we
have a continuous bilinear map  
\begin{equ}
\CC^{\gamma - \eta} \times \CB \to \CB \qquad
(g,F) \mapsto g\,F
\end{equ}
such that $\iota(gF) = \CL_{\gamma-\eta}(g)\iota F$ and such that
\begin{equ}
\hat \CR(g\,F) = g\,\hat \CR(F)\;,
\end{equ}
where the right hand side is meaningful thanks to the fact that $\gamma > 0$. 
Then, with $\bar \gamma$
and $\bar w$ as in \cite[Lem.~4.12]{Mate}, one can find $\phi \in \CC^{\bar \gamma,\bar w}$ such that
the modelled distribution $G$ given by \eqref{e:defG} (with suitably defined $\CK_\ell$, see the proof) belongs to $\CD^{\bar\gamma,\bar w}$ and 
satisfies $\CR G = K_0 (g\hat \CR F)$.

If furthermore one has a sequence of models $\PPi_n \to \PPi$ with associated 
subspaces $\CB_n \subset \CD^{\gamma,w}(V)$ and reconstruction
operators $\hat \CR_n$, as well as a sequence $F_n \in \CB_n$
such that $\$\iota F_n; \iota F\$_{\gamma,w} \to 0$ and 
$\|\hat \CR_n F_n - \hat \CR F\|_{\eta} \to 0$, then the sequence of modelled
distributions $G_n$ constructed in the first part of the statement converges to
$G$ in the sense that $\$G_n; G\$_{\bar \gamma,\bar w} \to 0$.
\end{corollary}

\begin{proof}
The proof is virtually identical to that of Corollary~\ref{cor:commutation}. 
The only difference is that we use \cite[Lem.~4.12]{Mate} to define maps
$\CK_k \colon \CB \to \CD^{\bar \gamma_k,\bar w_k}$ (with $\bar \gamma_k$ and $\bar w_k$
defined like $\bar \gamma$ and $\bar w$, but with $\beta$ replaced by $\beta + |k|_\s$) such that 
$\CR \CK_k F = K_k \hat \CR F$. (Since $\eta+\beta > -2$ by assumption, the assumptions
of that lemma are satisfied and $\CR \CK_k F$ is always well-defined
as a distribution on the whole space-time.)
\end{proof}

\section{Reconstruction theorem}
\label{sec:reconstruction}

In this section, we present a version of the reconstruction theorem that allows to 
bypass to some extent the condition $\nu > -1$ appearing in \cite{Mate}.
This appendix was written in collaboration with Máté Gerencsér.
In this section we assume that $P_0 = \{0\}\times \R^d$, $P_1 = \R \times \d D$,
and $P = P_0 \cup P_1$, and we write $|x|_{P_i}$ for the parabolic distance between
$x$ and $P_i$. Generic points $x$, $y$, etc are \textit{space-time} points.

Throughout this section, we fix a regularity structure $(T,G)$ containing the polynomial structure
and such that the product with elements of the polynomial structure $\bar T$ is well-defined in $T$.
We also only consider models such that $\Gamma_{x,y}$ acts on $\bar T$ by translations by $y-x$.
Let us recall from \cite[Sec~8]{regularity} that, given any regularity structure $(T,G)$, a model 
$(\Pi,\Gamma)$ can alternatively be described by a linear map 
$\PPi\colon T\to\CD'$, together with a continuous map $F\colon\R^{d}\to G$ such that,
setting $\Pi_x=\PPi F_x$, $\Gamma_{xy}=(F_x)^{-1}F_y$, the analytic bounds for models are satisfied.
The main assumption we impose on our models in the present setting is the following.
\begin{assumption}\label{as:decomposition}
Setting $T_{<} = \bigoplus_{\alpha \le 1} T_\alpha$,
there exist linear maps $\PPi^+, \PPi^-\colon T_{<} \to \CD'$ 
with $\PPi^+ + \PPi^- = \PPi \restr T_{<}$ and such that
\begin{claim}
\item $\big(\PPi^+\tau\big)(\psi)=0$ for all $\tau \in T_<$ and all $\psi$ supported in $\R\times D^c$ and
$\big(\PPi^- \tau\big)(\psi)=0$ for all $\psi$ supported in $\R\times D$;
\item Setting $\Pi^+_x = \PPi^+F_x$, $\Pi^-_x:=\PPi^-F_x$ the pairs $(\Pi^+,\Gamma)$ and $(\Pi^-,\Gamma)$
are models on $(T_<,G)$ in the sense of \cite{regularity}. (But they are not admissible in general!)
\end{claim}
\end{assumption}

We will always write $\underline \alpha$ for the lowest degree appearing in our ambient regularity structure $T$.
We also fix $\gamma > 0$ as well as exponents $\eta$ and $\sigma$ on which we make the following assumption.

\begin{assumption}\label{as:exponents}
The exponents satisfy the condition
\begin{equ}\label{eq:exponents}
0>\sigma>-1\geq\underline\alpha \ge \eta >-2\;.
\end{equ}
\end{assumption}

We also use the shorthand $w = (\eta,\sigma,\eta)$ similarly to \cite{Mate} (except that we make the simplifying assumption that 
the ``corner exponent'' coincides with $\eta$ which is not essential but simplifies our argument).
One crucial ingredient for our result is the following.

\begin{lemma}
For every $f \in \CD^{\gamma,w}$, there exist $f_\pm \in \CD^{\sigma,\eta}$ such that 
$f_+(x) = f(x)$ for $x \in \R \times D$ and $f_-(x) = f(x)$ for $x \in \R \times D^c$.
\end{lemma}

\begin{proof}
It follows from the definition of the spaces $\CD^{\gamma,w}$ that the restriction of $\CQ_{<\sigma}f$ to
either $\R \times \bar D$ or $\R \times \overline{D^c}$ belongs to $\CD^{\sigma,\eta}$. In particular, components of
$f$ of degree below $\sigma$ can be extended continuously to $(\R\setminus \{0\}) \times D$. (Note that $\sigma < 0$ though!)
The claim then follows from an adaptation of Whitney's extension theorem to the setting
of regularity structures, see for example \cite[Thm~5.3.16]{WhitneyRS}.
\end{proof}

\begin{remark}
In general there is no reason for $f_+$ and $f_-$ to coincide on $\R \times \d D$.
\end{remark}

We define spaces $\CC^{\alpha,\eta}$ with $\eta \le \alpha < 0$ as consisting of those distributions
$\zeta \in \CC^{\eta}(\R^{1+d}) \cap \CC^{\alpha}(\R^{1+d} \setminus P_0)$ such that
\begin{equ}
\big|\zeta(\psi_x^\lambda)\big| \lesssim \lambda^\alpha |x|_{P_0}^{\eta-\alpha}\;,
\end{equ}
uniformly over $x$ in compacts away from $P_0$, $2\lambda \in (0, 2 \wedge |x|_{P_0}]$,
and test functions $\psi \in \CB$. Here and below we write $\CB$ for the set of functions supported in
the centred (parabolic) ball of radius $1$ and with $r$ derivatives bounded by $1$, where $r$ 
is some fixed sufficiently large value.
Note that if $\eta > -2$, $\zeta \in \CC^{\alpha,\eta}$ is uniquely
determined by its action on test functions supported outside $P_0$, see \cite{regularity}.

\begin{theorem}\label{thm:reconstructDomain}
Under Assumptions~\ref{as:decomposition} and~\ref{as:exponents}, 
there then exists a unique continuous linear operator
$\CR:\CD^{\gamma,w}_P\to\CC^{\underline\alpha,\eta}$ such that
$\CR f(\psi)=\tilde\CR f(\psi)$ for all $\psi$ supported in 
$\R^d\setminus P$,
and such that one has the bound
\begin{equ}\label{eq:reco on boundary}
\big(\CR f-\Pi_x^+ f_+(x)-\Pi_x^- f_-(x)\big)(\psi_x^\lambda)\lesssim \lambda^\sigma|x|_{P_0}^{\eta-\sigma}
\end{equ}
uniformly over $x\in P_1\setminus P_0$ (in compacts), over $\lambda\in(0,1]$ such that
$2\lambda\leq |x|_{P_0}$, and over $\psi\in\CB$.
\end{theorem}

\begin{remark}
We did not specify the continuity of the map $\CR$ with respect to different models.
We will continue to omit continuity statements in the sequel, on one hand for the sake of easing the presentation, 
and on the other hand due to the fact that since all the operations discussed here are linear,
the ``linearisation trick'' of \cite[Prop~3.11]{HP}  automatically implies all the 
required continuity properties.
\end{remark}

\begin{proof}
The proof is very similarly to that of \cite[Thm~4.10]{Mate}, but
due to the central role of the statement in our proof, we provide some detail.
The uniqueness part is quite straightforward: take two $\CC^{\underline\alpha,\eta}$ distributions $\xi_1,\xi_2$
that have the properties claimed for $\CR f$ in the theorem.
Their difference then vanishes away from $P$, and thanks to the bound \eqref{eq:reco on boundary},
must belong to $\CC^{\sigma,\eta}$.
Since $\sigma > -1$ such a distribution has to vanish on $P_1$ and since $\eta > -2$
it has to vanish on $P_0$, so that $\xi_1 = \xi_2$.

To construct $\CR$, we use essentially the same construction as in the proof of \cite[Prop.~6.9]{regularity}. 
Similarly to the construction of the functions $\phi_{x,n}$ performed there,
we can find, for every $n \in \N$, a countable index set $\Xi_n$ and functions $\phi_{x,n}$ with $n \in \N$
and $x \in \Xi_n$ with the following properties. There exist constants $c_i > 0$ such that:
\begin{claim}
\item[(i)] For every $n \in \N$ and $x \in \Xi_n$ there exists $\psi \in \CB$, $y \in \R^{d+1}$ with 
$|y|_{P_1} = 2^{-n}$ such that, setting $\lambda = 2^{-n-1}$, one has $\phi_{x,n} = c_1 \lambda^{d+2} \psi_{y}^\lambda$.
\item[(ii)] For every ball $B$ of (parabolic) radius $\lambda$, there exist at most $c_2\lambda^{d+1}$ elements $x \in \Xi_n$ with 
$\supp \phi_{x,n} \cap B \neq \emptyset$.
\item[(iii)] For every $y \in \R^{d+1}$ with $0 < |y|_{P_1} \le 1$, one has
$\sum_{n \in \N}\sum_{x \in \Xi_n} \phi_{x,n}(y) = 1$.
\end{claim}
(Note that the sum appearing in the last claim always converges since, by the first two properties, it is guaranteed to only
contain finitely many terms.)

We now write $\CR_\pm$ for the reconstruction operators for $\CD^\gamma$ spaces with $\gamma < 0$
associated to the models $\Pi^\pm$ as in the second part of \cite[Thm~3.10]{regularity}, we fix
$y \in \R^{d+1} \setminus P_0$, $\lambda \le 1 \wedge |y|_{P_0}/C$ for some large enough
(but fixed) constant $C$, and $\psi \in \CB$, and we define
\begin{equ}
(\CR f)(\psi_y^\lambda) = \big(\CR_+ f_+ + \CR_- f_-\big)(\psi_y^\lambda) + 
\sum_{n \ge 0}\sum_{x \in \Xi_n}  \bigl(\tilde \CR f - \CR_+ f_+ - \CR_- f_-\bigr) \big(\phi_{x,n}\psi_y^\lambda\big)\;.
\end{equ}
If one also has $\lambda \le |y|_{P_1}/2$, then the second sum only contains finitely many terms and one
has $(\CR f)(\psi) = (\tilde \CR f)(\psi)$. Since in this case it follows from \cite[Def.~3.1]{Mate} combined with 
\cite[Lem.~6.7]{regularity} that one has the bound
\begin{equ}
|(\tilde \CR f)(\psi_y^\lambda)| \lesssim |y|_{P_0}^{\eta - \underline\alpha} \lambda^{\underline\alpha}\;,
\end{equ}
it remains to consider the convergence of the second term.
Since we can restrict ourselves to the case $|y|_{P_1} \le 2\lambda$, we can assume without loss of 
generality that $y \in P_1 \setminus P_0$. In particular, the terms in the sum vanish unless
$2^{-n} \lesssim \lambda$.

We also note that for these terms, one can write 
\begin{equ}
\phi_{x,n}\psi_y^\lambda = \bar\lambda^{d+2} \lambda^{-d-2} \tilde \psi_{z}^{\bar \lambda}\;,
\end{equ}
for some $z$, some $\tilde \psi \in \CB$, and $\bar \lambda = 2^{-n-1}$.
Provided that $C$ is sufficiently large, properties (i) and (ii) guarantee furthermore that 
$z$ is such that $\bar \lambda \simeq |z|_{P_1}\le |z|_{P_0}$.
It follows from this that, when restricted to the support of $\tilde \psi_{z}^{\bar \lambda}$,
one has $\|\CQ_{<\delta} f\|_{\CD^\delta} \lesssim  |z|_{P_0}^{\eta-\sigma} \bar \lambda^{\sigma-\delta}$,
provided that $\delta \in [\sigma,\gamma]$, so that \cite[Lem.~6.7]{regularity} yields again
\begin{equs}
\bigl|\bigl(\tilde \CR f - \Pi_z f(z) \bigr) \big(\tilde \psi_{z}^{\bar \lambda}\big)\bigr|
&\lesssim  |z|_{P_0}^{\eta-\sigma} \bar \lambda^{\sigma}\;,\\
\bigl|\bigl(\CR_\pm f_\pm - \Pi_z^\pm f_\pm(z) \bigr) \big(\tilde \psi_{z}^{\bar \lambda}\big)\bigr|
&\lesssim |z|_{P_0}^{\eta-\sigma} \bar \lambda^{\sigma}\;.
\end{equs}
Finally, since the support of $\tilde \psi$ is either fully contained in $\R \times D$ or
fully contained in $\R \times D^c$, it follows that 
\begin{equs}
\big|\bigl(\Pi_z f(z) - \Pi_z^+ f_+(z) &- \Pi_z^- f_-(z)\bigr)(\tilde \psi_{z}^{\bar \lambda})\big|
= \sum_{\sigma \le \alpha < \gamma} \big|\bigl(\Pi_z Q_\alpha f(z)\bigr)(\tilde \psi_{z}^{\bar \lambda})\big| \\
&\lesssim |z|_{P_0}^{\eta-\sigma} \sum_{\sigma \le \alpha < \gamma}  |z|_{P_1}^{\sigma-\alpha} \bar \lambda^\alpha
\lesssim |z|_{P_0}^{\eta-\sigma} \bar \lambda^{\sigma}\;,
\end{equs}
and we conclude that 
\begin{equ}
\bigl|\bigl(\tilde \CR f - \CR_+ f_+ - \CR_- f_-\bigr) \big(\phi_{x,n}\psi_y^\lambda\big)\bigr|
\lesssim |z|_{P_0}^{\eta-\sigma} \lambda^{-d-2} 2^{-(\sigma+d+2) n}\;.
\end{equ}
Since there are at most $2^{n(d+1)}$ such terms and since $\sigma > -1$ by assumption, both the
convergence of the sum and the required bound \eqref{eq:reco on boundary} follow.
\end{proof}

\endappendix

\bibliographystyle{Martin}
\bibliography{refs}

\end{document}